\documentclass{tac}  
\usepackage[all,2cell]{xy}
\usepackage{mathrsfs,amssymb,stmaryrd,bbm,mathtools, mathabx}

\newtheorem{theo}{Theorem}[section]
\newtheorem{lem}[theo]{Lemma}
\newtheorem{prop}[theo]{Proposition}
\newtheorem{cor}[theo]{Corollary}
\newtheoremrm{defi}[theo]{Definition}
\newtheoremrm{rem}[theo]{Remark}

\def\mathrmdef#1{\expandafter\def\csname#1\endcsname{{\rm#1}}}
\def\mathsfdef#1{\expandafter\def\csname#1\endcsname{{\rm\mathsf#1}}}
\def\mathcaldef#1{\expandafter\def\csname#1\endcsname{{\mathcal#1}}}
\mathrmdef{Id}\mathrmdef{op}
\mathsfdef{Alg} \mathrmdef{Ps} \mathsfdef{Set} \mathsfdef{CAT}\mathcaldef{Ran}\mathcaldef{Lan}\mathrmdef{bi}
\mathrmdef{j}\mathrmdef{h}\mathsfdef{Mod}\mathsfdef{Ring}\mathcaldef{Desc}\mathsfdef{Grph}\mathrmdef{t}\mathrmdef{obj}\mathrmdef{Id}
\mathrmdef{id}\mathcaldef{RAN}\mathrmdef{m}\mathrmdef{n}\mathrmdef{i}\mathsfdef{Cat}\mathsfdef{Top}
\mathrmdef{suc}\mathrmdef{Suc}\mathcaldef{Bicat}\mathcaldef{Icon}
\mathrmdef{bilim}\mathrmdef{bicolim}
\def\aaa{\mathfrak{A}}
\def\bbb{\mathfrak{B}}
\def\ccc{\mathfrak{C}}

\def\xxxx{\mathfrak{x}}
\def\yyyy{\mathfrak{y}}
\def\aaaa{\mathfrak{a}}
\def\bbbb{\mathfrak{b}}

\def\uuuu{\mathfrak{u}}
\def\llll{\mathfrak{l}}
\def\tttt{\mathfrak{t}}
\def\gggg{\mathfrak{g}}

\def\CCC{\mathbbmss{C}}
\def\DDD{\mathbbmss{D}}
\def\EEEE{\mathbbmss{E}}
\def\BB{\mathbbmss{B}}
\def\WW{\mathbbmss{W}}

\def\AAA{\mathcal{A}}
\def\BBB{\mathcal{B}}
\def\DDDD{\mathcal{D}}
\def\CCCC{\mathcal{C}}
\def\TTT{\mathcal{T}}
\def\YYY{\mathcal{Y}}
\def\WWW{\mathcal{W}}
\def\UUU{\mathcal{U}}
\def\LLL{\mathcal{L}}
\def\III{\mathcal{I}}
\def\KKK{\mathcal{K}}
\def\RRRR{\mathcal{R}}

\def\RRR{\mathtt{R}}
\def\EEE{\mathtt{E}}
\def\hhh {\mathfrak{H}}

\def\S{\mathcal{S}}

\begin{document}  
\title{Pseudo-Kan Extensions and Descent Theory}  
\author{Fernando Lucatelli Nunes}
\address{CMUC, Department of Mathematics, University of Coimbra, 3001-501 Coimbra, Portugal}
\eaddress{fernandolucatelli@gmail.com}

\thanks{This work was supported by CNPq, National Council for Scientific and Technological Development -- Brazil (245328/2012-2),  and by the Centre for Mathematics of the
University of Coimbra -- UID/MAT/00324/2013, funded by the Portuguese
Government through FCT/MCTES and co-funded by the European Regional Development Fund through the Partnership Agreement PT2020.}

\maketitle
\begin{abstract}
There are two main constructions in classical descent theory: the category of algebras and the descent category, which are known to be examples of  weighted bilimits. We give a formal approach to descent theory, employing formal consequences of commuting properties of bilimits to prove classical and new theorems in the context of Janelidze-Tholen ``Facets of Descent II'', such as B\'{e}nabou-Roubaud Theorems, a Galois Theorem, embedding results and formal ways of getting effective descent morphisms. In order to do this, we develop the formal part of the theory on commuting bilimits via pseudomonad theory, studying idempotent pseudomonads and proving a $2$-dimensional version of the adjoint triangle theorem. Also, we work out the concept of pointwise pseudo-Kan extension, used as a framework to talk about bilimits, commutativity and the descent object. As a subproduct, this formal approach can be an alternative perspective/guiding template for the development of higher descent theory.
\end{abstract}

\section*{Introduction}
Descent theory is a generalization of a solution given by Grothendieck to a problem related to modules over rings \cite{Grothendieck, Facets3, Facets}. There is a pseudofunctor 
$\Mod :\Ring\to \CAT $ which associates each ring $ \RRRR $ with the category $\Mod ( \RRRR )$ of right $\RRRR $-modules. The original problem of descent is the following: given a morphism $ f: \RRRR\to \S $ of rings, we wish to understand what is the image of 
$\Mod ( f ): \Mod( \RRRR )\to \Mod( \S )$.
The usual approach to this problem in descent theory is somewhat indirect: firstly, we characterize the morphisms $f$ in $ \Ring$ such that $\Mod ( f ) $ is a functor that forgets some ``extra structure''. Then, we would get an easier problem: verifying which objects of $\Mod( \S )$ could be endowed with such extra structure (see, for instance, \cite{Facets3}).

Given a category $\CCC $ with pullbacks and a pseudofunctor $ \AAA : \CCC ^{\op} \to \CAT $, for each morphism $p : E\to B$ of $ \CCC $, the \textsl{descent data} plays the role of such ``extra structure'' in the basic problem (see \cite{Facets, Facets2, Street2}). More precisely, in this context, there is a natural construction of a category $\Desc_ \AAA (p)$, called descent category, such that the objects of $\Desc _ \AAA (p) $ are objects of $\AAA (E)$ endowed with descent data, which encompasses the $2$-dimensional analogue for equality/$1$-dimensional descent: one invertible $2$-cell plus coherence. This construction comes with a comparison functor and a factorization; that is to say, we have the commutative diagram below, in which $\Desc _ \AAA (p)\to \AAA (E) $ is the functor which forgets the descent data (see \cite{Facets, Facets2}).  
\begin{equation*}\tag{Descent Factorization}\label{JTDESCENT}
\xymatrix{ \AAA  (B) \ar[r]^-{\phi _ p}\ar[rd]_{ \AAA ( p )}& {\Desc }_\AAA ( p ) \ar[d]\\
&\AAA (E)}
\end{equation*}
Therefore the problem is reduced to investigating whether the \textit{comparison} functor $ \phi _ p $ is an equivalence. If it is so, $p$ is is said to be of \textit{effective $\AAA$-descent} and the image of $\AAA (p) $ are the objects of $\AAA (E) $ that can be endowed with descent data. Pursuing this strategy, it is also usual to study cases in which $ \phi _ p $ is fully faithful or faithful: in these cases, $p$ is said to be, respectively, of \textit{$\AAA $-descent} or of almost \textit{$\AAA$-descent}.

Furthermore, we may consider that the descent problem (in dimension $2$) is, in a broad context, the characterization of the image (up to isomorphism) of any given functor $ F:\mathtt{b}\to \mathtt{e} $. In this case, using the strategy described above, we investigate if $\mathtt{b} $ can be viewed as a category of objects in $\mathtt{e} $ with some extra structure (plus coherence). Thereby, taking into account the original basic problem, we can ask, hence, if $F$ is (co)monadic. Again, we would get a factorization, the Eilenberg-Moore factorization: 
$$\xymatrix{\mathtt{b} \ar[r]^-{\phi }\ar[rd]_{F}& (Co)\Alg \ar[d]\\
&\mathtt{e} }$$
This approach leads to what is called ``monadic descent theory''. B\'{e}nabou and Roubaud proved that, if $F = \AAA (p) $ in which $\AAA : \CCC ^{\op} \to \CAT $ is a pseudofunctor
satisfying the Beck-Chevalley condition, then ``monadic $\AAA$-descent theory'' coincides with ``Grothendieck $\AAA$-descent theory''. More precisely, in this case, $p$ is of effective $\AAA$-descent if and only if $\AAA (p )$ is monadic \cite{Equi, Facets, Sobral, Ivan}.

Thereby, in the core of classical descent theory, there are two constructions: the category of algebras and the descent category. These constructions are known to be examples of
$2$-categorical limits (see \cite{Street2, Street3}). Also, in a $2$-categorical perspective, we can say that the general idea of category of objects with ``extra structure (plus coherence)'' is, indeed, captured by the notion of $2$-dimensional limits.   

Not contradicting such point of view, Street considered that (higher) descent theory is about the 
higher categorical notion of limit \cite{Street2}. Following this posture, we investigate whether 
pure formal methods and commuting properties of bilimits are useful to prove classical and new theorems in the 
classical context of descent theory of \cite{Facets, Facets2, Facets3, Gd}. 

Willing to give such formal approach, we employ the following perspective: 
\textit{the problems of descent theory are usually reduced  to the study of the 
image of a (pseudo)monadic (pseudo)functor}. 
We restrict our attention
to idempotent pseudomonads and prove formal results on pseudoalgebra structures, 
such as a biadjoint triangle theorem and lifting theorems.

In order to apply such formal approach to get theorems on commutativity of bilimits, we employ
a bicategorical analogue of the concept of (pointwise) Kan extension: (pointwise)
 pseudo-Kan extension, introduced in \cite{FLN}. 

By successive applications of these formal results, we get results 
within the context of  \cite{Facets, Facets2},
such as the B\'{e}nabou-Roubaud Theorem, embedding results and theorems on 
effective descent morphisms
of bilimits of categories. We also apply this approach to get results
on effective descent morphisms of categories of small enriched categories $V\textrm{-}\Cat $
provided that $V$ satisfies suitable hypotheses.

In this direction, the fundamental standpoint on ``classical descent theory'' of this paper is the following: 
the ``descent object'' of a (pseudo)cosimplicial object in a given context is the image of the initial object 
of the 
appropriate notion of Kan extension of such cosimplicial object. More precisely, in our context of dimension 
$2$ (which is the same context of \cite{Facets, Facets2}), 
we get the following result (Theorem \ref{indeeddescent}): 
\textit{The descent category of a pseudocosimplicial object $\AAA : \Delta \to \CAT $ is equivalent to 
$\Ps\Ran _ \j \AAA (\mathsf{0}) $}, in which $\j : \Delta\to \dot{\Delta } $ is the full inclusion of the 
category of finite nonempty ordinals into the category of finite ordinals and order preserving functions, 
and $\Ps\Ran _ \j \AAA $ denotes the right pseudo-Kan extension of $\AAA $ along $\j $. In particular, 
we show abstract features of the ``classical theory of descent'' 
as a theory (of pseudo-Kan extensions) of pseudocosimplicial objects or 
pseudofunctors $\dot{\Delta } \to\CAT $.

This work was motivated by three main aims. Firstly, to get formal proofs of classical results of 
descent theory. Secondly, to prove new results in the classical context -- 
for instance, formal ways of getting sufficient conditions for a morphism to be of effective descent. 
Thirdly, to get proofs of descent theorems that could be recovered in other contexts, 
such as in the development of higher descent theory (see, for instance, the work of Hermida \cite{Her} and 
Street \cite{Street2} in this direction).

In Section \ref{Basic Problem}, we give an idea of our scope within the context of  \cite{Facets, Facets2}: 
we show the main results classically used to deal 
with the problem of characterization of effective descent morphisms and we 
present classical results, which are proved using results on commutativity in Sections 
\ref{proofsclassical} and \ref{Further}. Namely, the embedding results (Theorems \ref{facil1} and 
\ref{mergulho}) and 
the B\'{e}nabou-Roubaud Theorem (Theorem \ref{Roubaud}). 
At the end of Section \ref{Basic Problem}, we establish a theorem on pseudopullbacks of categories 
(Theorem \ref{pseudopullback}) which is proved in Section \ref{Further}.

Section \ref{Formal} contains most of the abstract results of our formal approach to descent
via pseudomonad theory. We start by establishing our main setting: 
the tricategory of $2$-categories, pseudofunctors and pseudonatural transformations.
In  \ref{idempotent}, we define and study basic aspects of idempotent pseudomonads. 
Then, in  \ref{btt}, we study pseudoalgebra structures w.r.t. idempotent pseudomonads, 
proving a Biadjoint Triangle Theorem (Theorem \ref{biadjointtriangle}) and giving a result related to the 
study of pseudoalgebra structures in commutative squares (Corollary \ref{square}).

We deal with the technical situation of considering objects that 
cannot be endowed with pseudoalgebra structures but have comparison morphisms belonging to a special class 
of morphisms in  \ref{tecnicoalmost}.

Section \ref{KAN} explains why we do not use the usual 
enriched Kan extensions to study commutativity of the $2$-dimensional limits related to descent theory: 
the main point is that we like to have results which work for bilimits in general (not only flexible ones). 
In \ref{pseudoKAN}, we define pseudo-Kan extensions and, then, we give the associated factorizations in 
\ref{factorization}. Particular cases of these factorizations are  the Eilenberg-Moore factorization of
an adjunction and the descent factorization described above.

We give further background material in \ref{Spointwise}, studying 
weighted bilimits and proving the first result that relates pseudo-Kan extensions and weighted bilimits.
Then, in \ref{introducedpseudoend}, we introduce pseudoends  and prove basic results such as a version of 
Fubini's theorem for pseudoends (Theorem \ref{Fubini})  and the fundamental equivalence for pseudoends (Proposition \ref{Q}).
In \ref{pointwisepseudokanextension}, we finally introduce the appropriate notion of pointwise pseudo-Kan extension (Definition \ref{pointwise3})  and,
using the results on pseudoends, we
prove that a pointwise pseudo-Kan extension exists if and only if suitable weighted bilimits exist (Theorem \ref{pointwise} and Corollary \ref{pointwise2}).

In \ref{Psidem}, \ref{commu} and \ref{tecalmost}, we fit the study of pseudo-Kan extensions into the 
perspective of  Section \ref{Formal}. We apply the results of \ref{Formal} to the special case of 
weighted bilimits and pseudo-Kan extensions: we get, then, results on 
commutativity of weighted bilimits/pseudo-Kan extensions and exactness/\textit{(almost/effective) 
descent diagrams}.

Section \ref{descent} studies \textit{descent objects}. We prove that the classical descent object (category) 
 is given by the pseudo-Kan extension of a pseudocosimplicial object (as explained above).
 In particular, this means that \textit{descent objects are conical bilimits of pseudocosimplicial objects}. We
 adopt this description as our definition of descent object of a pseudocosimplicial object. 
We finish Section \ref{descent} presenting also the strict version of a descent object, which is given by a
Kan extension of a special type of $2$-diagram. We get, then, the strict factorization of descent theory.

Section \ref{Elementary Examples} gives elementary examples of our context of 
\textit{effective descent diagrams}. Every weighted bilimit can be seen as an example, 
but we focus in examples that we use in applications. As mentioned above, the most important examples of 
bilimits in descent theory are descent objects and Eilenberg-Moore objects: thereby, Section \ref{EilenBERG} 
is dedicated to explain how Eilenberg-Moore objects fit in our context, via the free adjunction $2$-category of 
\cite{Street}.

In Section \ref{BECK}, we study the Beck-Chevalley condition: 
by doctrinal adjunction~\cite{DoctrinalAdjunction}, this is the necessary and sufficient condition to guarantee that a pointwise 
adjunction between pseudoalgebras can be, actually, extended to an adjunction between such pseudoalgebras. 
We show how it is related to commutativity of weighted bilimits, giving our first version of a  
B\'{e}nabou-Roubaud Theorem (Theorem \ref{BFirst}).

We apply our results to the usual context~\cite{Facets, Facets2} of descent theory in 
Section \ref{proofsclassical}: we prove a general version (Theorem \ref{facil1provado}) of the 
embedding results (Theorem \ref{facil1}), we prove another B\'{e}nabou-Roubaud Theorem 
(Theorem \ref{formalbenabou}) and, finally, we give a weak version of Theorem \ref{pseudopullback}.

We finish the paper in Section \ref{Further}: there, we give a stronger result on commutativity 
(Theorem \ref{GALOISI}) and we apply our results to descent theory, proving
Theorem \ref{pseudopullback} and the Galois result of \cite{StreetJane} (Theorem \ref{GALOISII}). 
Finally, we prove that $V$-$\Cat$ can be nicely embedded in the category of internal categories $\Cat (V) $
provided that $V$ satisfies suitable hypotheses. In this situation, we apply
Theorem \ref{pseudopullback} to get effective descent morphisms of the category of 
enriched categories $V$-$\Cat $. We give instances of this result, getting effective descent morphisms of $\Top $-$\Cat $ and $\Cat $-$\Cat$.

This work was realized during my PhD program at University of Coimbra. 
I am grateful to my supervisor Maria Manuel Clementino for her precious help, support and attention. 
I also thank all the speakers of our informal seminar on descent theory for their insightful talks: 
Maria Manuel Clementino, George Janelidze, Andrea Montoli, Dimitri Chikhladze, Pier Basile and Manuela Sobral. 
Finally, I wish to thank Stephen Lack for our brief conversations which helped me to understand 
aspects related to this work about $2$-dimensional category theory, Kan extensions and coherence.

\section{Basic Problem}\label{Basic Problem}

In the context of \cite{Sobral, Facets, Facets2, Facets3, Ivan, THR, Clementino1}, the very 
basic problem of descent is the characterization of effective descent morphisms w.r.t. the basic fibration. 
As a consequence of B\'{e}nabou-Roubaud Theorem~\cite{Equi}, this problem is trivial for  
suitable categories (for instance, for locally cartesian closed categories). 

However there are remarkable examples of nontrivial characterizations. 
The topological case, solved by  Tholen and Reiterman~\cite{THR} and reformulated by Clementino and 
Hofmann~\cite{reforCleHof, CleJane}, is an important example.

Below,  we present some theorems classically used as a framework to deal with this basic problem. 
In this paper, we show that most of these theorems are consequences of a formal theorem presented in 
Section \ref{Formal}, while others are consequences of theorems about bilimits.

Firstly, the most fundamental features of descent theory are the descent category and its related factorization. 
Assuming that $\CCC $ is a category with pullbacks, 
if $\AAA : \CCC ^{\op }\to \CAT $ is a pseudofunctor, the \ref{JTDESCENT} 
 is described by Janelidze and Tholen in \cite{Facets2}.

We show in Section \ref{descent} that the concept of pseudo-Kan extension encompasses these features. 
In fact, the comparison functor and the \ref{JTDESCENT} (up to isomorphism) come from the unit and the 
triangular invertible modification of the (bi)adjunction 
$\left[\t , \CAT\right]_{({PS})}\dashv (\Ps)\Ran _ \t $.

Secondly, for the nontrivial problems, the usual approach to study (basic/universal) 
effective/almost descent morphisms is the embedding in well behaved categories, 
in which ``well behaved category'' means just that we know which are the effective descent morphisms 
of this category. 
For this matter, there are some theorems in \cite{Facets, Sobral}. 
We state below examples of these results: 
\begin{theo} \label{facil1}
Let $U: \CCC\to\DDD $ be a pullback preserving functor between categories with pullbacks.
\begin{enumerate}

\item If $U$ is faithful, then $U$ reflects almost descent morphisms;
\item If $U$ is fully faithful, then $U$ reflects descent morphisms.
\end{enumerate}
\end{theo}
\begin{rem}
The result on descent morphisms above can be seen as a consequence of Proposition of 2.6 of \cite{Facets}.
\end{rem}

\begin{theo}[\cite{Facets, Sobral}]\label{mergulho}
Let $\CCC $ and $ \DDD $ be categories with pullbacks. If $U: \CCC\to\DDD $ is a fully faithful pullback preserving functor and $U(p)$ is of effective descent in $ \DDD $, then $p$ is of effective descent if and only if it satisfies the following property: whenever the diagram below is a pullback in $\DDD$, there is an object $C$ in $\CCC $ such that $U(C)\cong A$.

$$\xymatrix{ {U(P)}
    \ar[r]^-{}\ar[d]_{}
    &
    {A}
    \ar[d]^{}
    \\
    {U(E)}
    \ar[r]_-{U(p)}
    &
    {U(B)}}$$

\end{theo}

We show in Section \ref{proofsclassical} that Theorem \ref{facil1} is a 
very easy consequence of formal and commuting properties of pseudo-Kan extensions 
(Corollary \ref{comutatividade} and Corollary \ref{comutatividadealmostdescent}) that
follow directly from results of Section \ref{Formal}, 
while we show in Section \ref{Further} 
that Theorem \ref{mergulho} is a consequence of a theorem on bilimits (Theorem \ref{strong commutativity}) 
which also implies the generalized Galois Theorem of \cite{StreetJane}. 
It is interesting to note that, since Theorems \ref{facil1} and \ref{mergulho} are just formal properties, 
they can be applied in other contexts -- for instance, for morphisms between 
pseudofunctors $\AAA : \CCC ^{\op }\to \CAT $ and $\BBB : \DDD ^{\op }\to \CAT $, 
as it is explained in Section \ref{proofsclassical}.

Finally, B\'{e}nabou-Roubaud Theorem~\cite{Equi, Facets} is a celebrated result of Descent Theory 
which allows us to understand some problems via monadicity: it says that monadic $\AAA $-descent theory 
is equivalent to Grothendieck $\AAA$-descent theory in suitable cases, such as the basic fibration. 
We demonstrate in Section \ref{proofsclassical} that it  is also a corollary of formal results of 
Section \ref{Formal}. 

\begin{theo}[B\'{e}nabou-Roubaud~\cite{Equi, Facets}]\label{Roubaud}
Let $\CCC $ be a category with pullbacks. If $\AAA : \CCC ^{\op }\to \CAT $ is a pseudofunctor such that, 
for every morphism $p: E\to B$ of $\CCC $, $A(p)$ has left adjoint $A(p)!$ and the invertible 
$2$-cell induced by $\AAA $ below satisfies the Beck-Chevalley condition, 
then the \ref{JTDESCENT} is pseudonaturally equivalent to the Eilenberg-Moore factorization. 
In other words, assuming the hypotheses above, \textit{Grothendieck $\AAA $-descent theory} 
is equivalent to \textit{monadic descent theory}.

$$\xymatrix{ {\AAA (B) }\ar[d]_-{\AAA (p) }\ar[r]^-{\AAA (p) }&
    \AAA (E) \ar[d]^-{}\ar@{}|{\cong}[dl]\\
    \AAA (E) \ar[r]_-{} & \AAA ( E\times _ p E)}$$

\end{theo}

\subsection{Open problems}\label{p}

Clementino and Hofmann~\cite{Clementino1} studied the problem of characterization
of effective descent morphisms for $(T,V)$-categories provided that $V$ is a lattice. 
To deal with this problem, 
they used the embedding  $(T,V)\textrm{-}\Cat\to (T,V)\textrm{-}\Grph$ and 
Theorems \ref{facil1} and \ref{mergulho}. However, 
for more general monoidal categories $V$, 
such inclusion is not fully faithful and the characterization of effective descent morphisms still is 
an open problem even for the simpler case of the category of enriched categories $V\textrm{-}\Cat $.
 
As an application, 
we give some results about effective descent morphisms of $V$-$\Cat $. 
They are consequences of formal results given in this paper on effective descent morphisms of 
categories constructed from other categories:
more precisely, $2$-dimensional limits of categories. 

More precisely, firstly we prove Theorem \ref{pseudopullback} in Section \ref{Further}.
Then, we prove that, if $V$ is a cartesian closed category satisfying suitable hypotheses and
$\Cat ( V ) $ is the category of internal categories, 
there is a full inclusion 
$V$-$\Cat\to \Cat( V ) $
which is the pseudopullback of a suitable fully faithful functor $\Set\to V $ along the projection of the underlying 
object of objects 
$\Cat( V )\to V $. 
In this case, 
we conclude that the inclusion reflects effective descent morphisms by Theorem \ref{pseudopullback}. 
Since the characterization of effective descent morphisms for the category of 
internal categories in our setting was already done 
by Le Creurer~\cite{Ivan}, we actually get effective descent morphisms for the category of
 $V$-enriched 
categories.

\begin{theo}\label{pseudopullback}
Assume that the diagram of categories with pullbacks  
\[\xymatrix{ \BB \ar[r]^S \ar[d]_ Z & \CCC \ar[d]^F\\
\ar@{}|{\cong}[ru]\DDD \ar[r]_G &\EEEE }\]
is a pseudopullback such that all the functors are pullback preserving functors. 
If $p$ is a morphism in $\BB $ such that $ S(p), Z(p) $ are of effective descent and $FS(p)$ is a descent morphism, 
then $p$ is of effective descent.
\end{theo}

\section{Formal Results}\label{Formal}

Our perspective herein is that, instead of considering the problem
of understanding the image of a generic (pseudo)functor,  
the main theorems of descent theory usually 
deal with the problem  of 
understanding the pseudoalgebras of  (fully) property-like 
(pseudo)monads~\cite{KELLYLACK1997}. 
It is easier to study these pseudoalgebras: they are just the objects that can 
be endowed with a unique pseudoalgebra structure 
(up to isomorphism), or, more appropriately, the effective descent points/objects.

Thereby results on pseudoalgebra structures are in the core of our formal approach. In this section, 
we give the main results of this paper in this direction, restricting the scope to idempotent
pseudomonads. This setting is  sufficient to deal with the classical descent problem of \cite{Facets, Facets2}. 

We start by recalling basic  
results of bicategory theory~\cite{BE, Street4, Street5}. 
To fix notation, we give the definition of the \textit{tricategory of $2$-categories, pseudofunctors, 
pseudonatural transformations and modifications}, denoted by $2$-$\CAT $. We refer to \cite{FLN} for the omitted 
coherence axioms of \ref{startsettingPseudofunctor} to \ref{adjunction} and for the proof of Lemma \ref{counitadjunction}.

Henceforth, in a given $2$-category, we always denote by $\cdot $ the vertical composition of 
$2$-cells and by $\ast $ their horizontal composition.

\begin{defi}[Pseudofunctor]\label{startsettingPseudofunctor}
Let $\aaa , \bbb $ be $2$-categories. A \textit{pseudofunctor} $\AAA:\aaa\to \bbb $ is a pair $(\AAA , \aaaa ) $ with the following data:
\begin{itemize}\renewcommand\labelitemi{--}
\item Function $\AAA : \obj (\aaa )\to \obj (\bbb ) $;
\item Functors $\AAA _{{}_{XY}}: \aaa (X,Y)\to \bbb (\AAA (X), \AAA (Y)) $;
\item For each pair $g: X\to Y , h: Y\to Z $ of $1$-cells in $\aaa $, an invertible $2$-cell in $\bbb $: $\aaaa _ {{}_{hg}}: \AAA (h) \AAA (g)\Rightarrow \AAA (hg) $;
\item For each object $X$ of $\aaa $, an invertible $2$-cell $\aaaa _ {{}_{X}}: \Id _{{}_{\AAA X}}\Rightarrow \AAA (\Id _ {{}_X} )$ in $\bbb $;
\end{itemize} 
subject to \textit{associativity}, \textit{identity} and \textit{naturality} coherence axioms.
\end{defi}

If $\AAA = (\AAA , \aaaa): \aaa\to\bbb $ and $(\BBB , \bbbb ): \bbb\to\ccc $
are pseudofunctors, we define the composition as follows: $ \BBB\circ \AAA := \left( \BBB\AAA, \left(\bbbb\aaaa\right)\right)$,
in which $(\bbbb\aaaa )  _ {{}_{hg}}:= \BBB(\aaaa _ {{}_{hg}})\cdot \bbbb _ {{}_{\AAA(h)\AAA(g)}} $ and $(\bbbb\aaaa )  _ {{}_{X}}:=\BBB(\aaaa _ {{}_{X}})\cdot \bbbb _ {{}_{\AAA(X)}} $. This composition is associative and it has trivial identities. A pseudonatural transformation between pseudofunctors $\AAA\longrightarrow\BBB $ is a natural transformation in which the usual (natural) commutative squares are replaced by invertible $2$-cells plus coherence. 

\begin{defi}[Pseudonatural transformation]\label{pseudonaturaltransformation}
If $\AAA , \BBB:\aaa\to \bbb $ are pseudofunctors, a \textit{pseudonatural transformation} $\alpha : \AAA\longrightarrow\BBB $ is defined by:
\begin{itemize}\renewcommand\labelitemi{--}
\item For each object $X$ of $\aaa $, a $1$-cell $\alpha _{{}_X}: \AAA (X)\to\BBB (X) $ of $\bbb $;
\item For each $1$-cell $g:X\to Y $ of $\aaa $, an invertible $2$-cell $\alpha _{{}_g}: \BBB (g) \alpha _{{}_X}\Rightarrow \alpha _{{}_Y}\AAA (g)  $ of $\bbb $;
\end{itemize} 
such that coherence axioms of associativity, identity and naturality hold.
\end{defi}

Firstly, the vertical composition, denoted by $\beta\alpha $, of two pseudonatural transformations $\alpha : \AAA\Rightarrow \BBB$, $\beta : \BBB\Rightarrow \CCCC $ is defined by
$$(\beta\alpha) _ {{}_W} :=\beta _{{}_W}\alpha _ {{}_W} $$

$$\xymatrix{  
\AAA (W)\ar[r]^{\beta _{{}_W}\alpha _ {{}_W}}\ar[d]_{\AAA (f) }\ar@{}[dr]|{\xLeftarrow{(\beta\alpha ) _{{}_f}} } 
&\CCCC (W)\ar@{}[drr]|{:=}
\ar[d]^{\CCCC (f)}
&& 
\AAA (W)\ar[rr]^{\alpha _{{}_W}}\ar[d]^{\AAA (f)}
\ar@{}[drr]|{\xLeftarrow{\alpha  _{{}_f}} } 
&& \BBB (W)\ar[d]_{\BBB (f) } \ar[r]^{\beta _ {{}_W}} 
\ar@{}[dr]|{\xLeftarrow{\beta  _{{}_f}} }
& \CCCC (W) \ar[d]^{\CCCC (f)} 
\\
\AAA (X)\ar[r]_{\beta _ {{}_X}\alpha _ {{}_X}} & \CCCC (X)
&& \AAA (X)\ar[rr]_ {\alpha _ {{}_X}} && \BBB (X)\ar[r]_{\beta _ {{}_X}} &\CCCC (X)
}$$
\normalsize
Secondly, let $(\UUU ,\uuuu ), (\LLL , \llll ) : \bbb\to \ccc$ and $ \AAA , \BBB : \aaa\to\bbb $ be pseudofunctors. If $\alpha : \AAA\longrightarrow\BBB $, $\lambda :\UUU\longrightarrow \LLL $ are pseudonatural transformations, then the horizontal composition of $\UUU $ with $\alpha $, denoted by $\UUU\alpha $, is defined by: $(\UUU\alpha ) _{{}_W} := \UUU(\alpha _{{}_W})$ and $(\UUU\alpha )_ {{}_f} := \left(\uuuu _{{}_{\alpha _{{}_X}\AAA (f)}}\right) ^{-1} \cdot \UUU(\alpha _{{}_f})\cdot \uuuu _{{}_{\BBB (f)\alpha _{{}_W}}}$, while the composition $\lambda \AAA $ is defined trivially. Thereby, we get the definition of the horizontal composition 
$$\left( \lambda\ast\alpha \right) := (\lambda \BBB )(\UUU\alpha  )\cong  (\LLL\alpha )(\lambda \AAA ).$$
Similarly, we get the three types of compositions of modifications.

\begin{defi}[Modification]
Let $\AAA , \BBB:\aaa\to \bbb $ be pseudofunctors. If $\alpha , \beta : \AAA\Rightarrow\BBB $ are pseudonatural transformations, a \textit{modification} 
$ \Gamma : \alpha\Longrightarrow \beta $
is defined by the following data:
\begin{itemize}\renewcommand\labelitemi{--}
\item For each object $X$ of $\aaa $, a $2$-cell $\Gamma _{{}_X}: \alpha _ {{}_X}\Rightarrow\beta _{{}_X} $ of $\bbb $
subject to one coherence axiom of naturality.		
\end{itemize} 								
\end{defi}

It is straightforward to verify that $2\textrm{-}\CAT $ is a tricategory which is locally a $2$-category. In particular, we denote by $[\aaa , \bbb ]_{PS} $ the $2$-category of pseudofunctors $\aaa\to\bbb $, pseudonatural transformations and modifications. Also, we have the \textit{bicategorical Yoneda lemma}~\cite{Street4} and, hence, the bicategorical Yoneda embedding $\YYY :\aaa\to [\aaa ^{\op }, \CAT ]_ {PS}$ is locally an equivalence (\textit{i.e.} it induces equivalences between the hom-categories).

\begin{lem}[Bicategorical Yoneda Lemma~\cite{Street4}]
The Yoneda embedding 
$$\YYY :\aaa\to [\aaa ^{\op }, \CAT ]_ {PS} : \qquad X\mapsto \aaa (- , X) $$
is locally an equivalence.
\end{lem}

\begin{defi}[Bicategorically representable]
A pseudofunctor $\AAA : \aaa\to\CAT $ is called \textit{bicategorically representable} if there is an object $W$ of $\aaa $ such that $\AAA $ is pseudonaturally equivalent to $ \aaa (W,-): \aaa\to\CAT $. In this case, $W$ endowed with a pseudonatural equivalence $\AAA\simeq\aaa (W,-) $ is called the \textit{bicategorical representation} of $\AAA $. 
\end{defi}

By the bicategorical Yoneda lemma, if it exists, a bicategorical representation of a pseudofunctor is unique up to equivalence. 

\begin{defi}[Bicategorical reflection]\label{bicategoricalreflection}
Let $\LLL : \aaa\to\bbb $ be a pseudofunctor and $X$ an object of $\bbb $ such that $\bbb (\LLL-, X): \aaa ^{\op }\to\CAT $ has a bicategorical representation 
$\UUU (X) $.
If $\varepsilon _ {{}_ {X}}:\LLL\UUU (X)\to X $ denotes the image
  of the identity  on $\UUU (X) $ by the equivalence 
$$\aaa (\UUU (X), \UUU (X)) \simeq \bbb (\LLL\UUU (X), X), $$
the pair $(\UUU (X), \varepsilon _ {{}_ {X}} ) $ is called the \textit{right bicategorical reflection} of $X$ along $\LLL $.
In this case, we often omit the morphism and say that $\UUU (X) $ is the right bicategorical reflection and 
$\varepsilon _ {{}_{X}}$ is the \textit{universal arrow} or \textit{counit} of the right bicategorical reflection.
\end{defi}

Since bicategorical representations are unique up to equivalence, right bicategorical reflections are unique up to equivalence as well. 
This means that, whenever $(\UUU (X)', \varepsilon _ {{}_ {X}}' ) $ and $(\UUU (X), \varepsilon _ {{}_ {X}} ) $ are right bicategorical reflections
of $X$ along $\LLL $, there exists an equivalence $\underline{v}: \UUU (X)'\simeq \UUU (X) $ such that there is an invertible 
$2$-cell $\varepsilon _ {{}_ {X}} \LLL (\underline{v} )\cong \varepsilon _ {{}_ {X}}' $.

\begin{rem}\label{rightbicategoricalreflectioncounit}
In the context of the definition above, it is easy to verify that $(\UUU (X), \varepsilon _ {{}_ {X}} ) $ is a right bicategorical reflection of
$X$ along $\LLL $ if and only if 
$$\aaa (-, \UUU (X)) \to \bbb (\LLL -, X), \quad f\mapsto  \varepsilon _ {{}_ {X}}\, \LLL (f) $$
defines a pseudonatural equivalence.   
\end{rem}

\begin{rem}
The dual notion is that of \textit{left bicategorical reflection}. Namely, if it exists, the left bicategorical reflection of $X$ along $\LLL $ is
the right bicategorical reflection of $$\LLL ^\op : \aaa ^\op\to\bbb ^\op .$$ Hence, if it exists, it consists of 
a pair $(X _ \LLL , \rho _ {{}_{X}} ) $ in which $X _ \LLL $ is an object of $\aaa $ and  $\rho _ {{}_X}: X\to \LLL (X_ \LLL ) $ is a morphism
in $\bbb $ such that 
$$\aaa (X_\LLL , -) \to \bbb (X, \LLL - ), \quad g\mapsto  \LLL (g)\, \rho _ {{}_ {X}}  $$
is a pseudonatural equivalence.
\end{rem}

We say that $\LLL $ is \textit{left biadjoint} to $\UUU : \bbb\to\aaa$ if, for every object $X$ of 
$\bbb $, $\UUU(X) $  is the  right bicategorical reflection of $X$ along $\LLL$. In this case, we say that 
$\UUU $ is \textit{right biadjoint} to $\LLL $. This definition of biadjunction is equivalent to 
Definition \ref{adjunction}.

\begin{defi}[Biadjunction]\label{adjunction}
A pseudofunctor $\LLL : \aaa\to \bbb $ is
\textit{left biadjoint} to $\UUU $ if there exist
\begin{enumerate}
\item
pseudonatural transformations $\eta :\Id _ {\aaa } \longrightarrow \UUU\LLL$ and
$\varepsilon :\LLL\UUU \longrightarrow \Id _ { \bbb }$
\item
invertible modifications
$s : \Id _{\LLL } \Longrightarrow (\varepsilon \LLL) \cdot (\LLL\eta)$
and
$t : (\UUU\varepsilon) \cdot (\eta \UUU) \Longrightarrow \Id _{\UUU }$
\end{enumerate}
satisfying coherence equations. In this case, $(\LLL\dashv \UUU , \eta , \varepsilon , s, t ): \aaa\to\bbb $ is a \textit{biadjunction}. Sometimes we omit the invertible modifications, denoting a biadjunction by $(\LLL\dashv \UUU , \eta , \varepsilon ) $.
\end{defi}

By the bicategorical Yoneda lemma, if $\LLL : \aaa\to\bbb $ is left biadjoint, its right biadjoint $\UUU : \bbb\to\aaa $ is unique up to pseudonatural equivalence. 
Furthermore, if $\LLL$ is left $2$-adjoint, it is left biadjoint.

\begin{defi}\label{lcaolequivalence}
A pseudofunctor $\UUU $ is a \textit{local equivalence} if
it induces equivalences between the hom-categories. 
\end{defi}

\begin{lem}\label{counitadjunction}
A right biadjoint $\UUU $ is a local equivalence if and only if the counit of the biadjunction 
is a pseudonatural equivalence.
\end{lem}

\subsection{Idempotent Pseudomonads}\label{idempotent}

Since we deal only with idempotent pseudomonads, 
we give an elementary approach focusing on them. 
The main benefit of this approach is that idempotent pseudomonads 
have only free pseudoalgebras. For this reason, assuming that $\eta $ is the unit of an 
idempotent pseudomonad $\TTT $,
an object $X$ can be endowed with a $\TTT $-pseudoalgebra structure if and only 
if $\eta _ {{}_X}: X\to \TTT(X) $ is an equivalence.

Recall that a \textit{pseudomonad} $\TTT$ on a $2$-category $\hhh $ consists of a sextuple 
$(\TTT, \mu , \eta ,  \Lambda, \rho, \Gamma)$, in which $\TTT :\hhh\to\hhh $ is a 
pseudofunctor, $\mu : \TTT ^2\longrightarrow \TTT,  \eta : \Id _ {{}_\hhh }\longrightarrow \TTT $ 
are pseudonatural transformations and
$$\xymatrix{  \TTT\ar@/_4ex/@{=}[dr]\ar[r]^-{\eta _ {{}_\TTT }}\ar@{}[dr]|-{\xLeftarrow{\hskip .3em \Lambda\hskip .3em }} 
& \TTT ^2\ar[d]|-{\mu}
& \TTT\ar[l]_-{\TTT\eta }\ar@/^4ex/@{=}[dl]\ar@{}[dl]|-{\xLeftarrow{\hskip .3em \rho\hskip .3em }} &&
\TTT ^3\ar[r]^{\TTT\mu}\ar[d]_{\mu _ \TTT}\ar@{}[dr]|{\xLeftarrow{\hskip .4em \Gamma\hskip .4em }}&
\TTT^2\ar[d]^{\mu }
\\
&\TTT&  && 
\TTT^2 \ar[r]_ {\mu } &
\TTT                    }$$ 
are invertible modifications satisfying the following coherence equations~\cite{Marmolejo1, FLN}:
\begin{itemize}\renewcommand\labelitemi{--}
\item Identity:
$$\xymatrix{ &
\TTT ^2\ar[dl]_-{\TTT \eta \TTT }\ar[dr]^-{\TTT \eta\TTT }\ar[dd]|-{\Id _ {{}_{\TTT ^2}}}
&
&&
&
\TTT ^2\ar[d]|-{\TTT \eta\TTT }
&
\\
\TTT ^3\ar[dr]_-{\mu\TTT}\ar@{}[r]|-{\xLeftarrow{\rho\TTT} }
&
&
\TTT ^3\ar@{}[l]|-{\xLeftarrow{\widehat{\TTT\Lambda}}}\ar[dl]^-{\TTT  \mu }
&&
&\TTT ^3\ar[dl]|-{\mu\TTT }\ar[dr]|-{\TTT \mu }
&
\\
&
\TTT ^2\ar[d]|-{\mu }                 
&
&=&
\TTT ^2\ar@{}[rr]|-{\xLeftarrow{\hskip 0.2cm \Gamma\hskip 0.2cm } }\ar[dr]|-{\mu }
&&
\TTT^2\ar[dl]|-{\mu }
\\
&
\TTT
&
&&
&
\TTT
&
}$$ 

\item Associativity:
$$\xymatrix{ \TTT ^4\ar[r]^-{\TTT ^2 \mu }\ar[dr]|-{\TTT\mu\TTT}\ar[d]_{\mu\TTT ^2}
&
\TTT ^3\ar[dr]^-{\TTT\mu}\ar@{}[d]|-{\xLeftarrow{\widehat{\TTT\Gamma} }}
&
&&
\TTT ^4\ar[r]^-{\TTT ^2\mu}\ar@{}[dr]|-{\xLeftarrow{\mu _ {{}_{{}_{\mu } }}^{-1}}}\ar[d]_-{\mu  \TTT ^2}
&
\TTT ^3\ar[d]|-{\mu \TTT}\ar[dr]^-{\TTT\mu }
&
\\
\TTT ^3\ar[dr]_-{\mu\TTT}\ar@{}[r]|{\xLeftarrow{\Gamma \TTT }}
&
\TTT ^3\ar[r]|-{\TTT\mu }\ar[d]|-{\mu\TTT }\ar@{}[dr]|-{ \xLeftarrow{\hskip 0.1cm\Gamma \hskip 0.1cm} }
&
\TTT ^2\ar[d]^-{\mu }
&=&
\TTT ^3\ar[r]|-{\TTT\mu }\ar[dr]_-{\mu \TTT}
&
\TTT ^2\ar@{}[r]|-{\xLeftarrow{\hskip 0.1cm \Gamma\hskip 0.1cm }}\ar[dr]\ar@{}[d]|-{\xLeftarrow{\hskip 0.1cm\Gamma\hskip 0.1cm } }
&
\TTT ^2\ar[d]^{\mu }
\\
&
\TTT ^2\ar[r]_{\mu}
&
\TTT 
&&
&
\TTT ^2\ar[r]_ {\mu }
&
\TTT
}$$ 
\end{itemize}
in which
$$\widehat{\TTT \Lambda } := \left( \tttt _ {{}_{\TTT }} \right)^{-1} \left(\TTT\Lambda \right) \left(\tttt _ {{}_{(\mu)(\eta\TTT) }}\right),\qquad\qquad
\widehat{\TTT \Gamma }  :=  \left( \tttt _{{}_{(\mu) (\mu\TTT) }}\right)^{-1}\left(\TTT \Gamma\right) \left( \tttt _{{}_{(\mu) (\TTT\mu ) }}\right) .$$

\begin{defi}[Idempotent pseudomonad]\label{idempotentdefinition}
A pseudomonad $(\TTT, \mu , \eta ,  \Lambda, \rho, \Gamma)$ is \textit{idempotent} if there is an invertible modification
$\eta\TTT \cong \TTT \eta$.
\end{defi} 

Similarly to $1$-dimensional monad theory, the name \textit{idempotent pseudomonad} is justified by Lemma \ref{multidem}, which says that
the multiplication of an idempotent pseudomonad is a pseudonatural equivalence.

\begin{lem}\label{multidem}
A pseudomonad $(\TTT, \mu , \eta ,  \Lambda, \rho, \Gamma)$ is idempotent if and only if the multiplication $\mu $ is a pseudonatural equivalence. 
In this case, $\eta\TTT  $ is an inverse equivalence of $\mu $.
\end{lem}
\begin{proof}
Since $\mu (\eta\TTT )\cong \Id _ {{}_\TTT } \cong \mu (\TTT\eta ) $, it is obvious that, if $\mu $ is a pseudonatural equivalence, then $\eta\TTT\cong \TTT\eta $. Therefore $\TTT $ is idempotent and $\eta\TTT $ is an equivalence inverse of $\mu $.

Reciprocally, assume that $\TTT $ is idempotent. By the definition of pseudomonads, there is an invertible modification $\mu (\eta\TTT )\cong \Id _ {{}_\TTT } $. And, since $\eta\TTT\cong\TTT\eta $, we get the invertible modifications 
$$(\eta\TTT )\mu \cong  (\TTT\mu) (\eta\TTT ^2)
                     \cong  (\TTT\mu ) (\TTT\eta\TTT )
										\cong \TTT (\mu (\eta\TTT) )
										\cong  \Id _ {{}_{\TTT ^2}} $$
which prove that $\mu $ is a pseudonatural equivalence and $\eta\TTT $ is a pseudonatural equivalence inverse. 
\end{proof}

The reader familiar with lax-idempotent/Kock-Z\"{o}berlein pseudomonads will notice that an idempotent pseudomonad is just a Kock-Z\"{o}berlein pseudomonad whose adjunction $\mu\dashv\eta\TTT $ is actually an adjoint equivalence. Hence, idempotent pseudomonads are fully property-like pseudomonads~\cite{KELLYLACK1997}.

Every biadjunction induces a pseudomonad~\cite{SteveLack2, FLN}. In fact, we get the multiplication $\mu $ from the counit, and the invertible modifications $\Lambda , \rho , \Gamma $ come from the invertible modifications of Definition~\ref{adjunction}. Of course, a biadjunction $\LLL\dashv\UUU $ induces an idempotent pseudomonad if and only if its unit $\eta $ is such that 
$\eta\UUU\LLL\cong \UUU\LLL \eta$.
As a consequence of this characterization, we have Lemma \ref{suportfact} which is necessary to give 
the Eilenberg-Moore factorization for idempotent pseudomonads. 

\begin{lem}\label{suportfact}
A biadjunction $(\LLL \dashv \UUU , \eta , \varepsilon )$ induces an idempotent pseudomonad if and only if 
$\eta\UUU : \UUU\longrightarrow\UUU\LLL\UUU $
is a pseudonatural equivalence.
\end{lem}
\begin{proof}
By the triangle invertible modifications of Definition~\ref{adjunction}, if $\varepsilon $ is the counit of the biadjunction $\LLL\dashv\UUU $,
$(\UUU\varepsilon) (\eta\UUU)\cong \Id _ {{}_{\UUU }}. $
Also, since $\UUU\LLL\eta \cong \eta\UUU\LLL $,  we have the following invertible modifications
$$
(\eta\UUU)\cdot (\UUU\varepsilon ) \cong  (\UUU\LLL\UUU\varepsilon ) (\eta\UUU\LLL\UUU )
                                   \cong  (\UUU\LLL\UUU\varepsilon ) (\UUU\LLL\eta \UUU )
																	 \cong  \UUU\LLL ( \Id _ {{}_{\UUU }} )
																	 \cong  \Id _ {{}_{\UUU\LLL\UUU }} $$
Therefore $\eta\UUU $ is a pseudonatural equivalence. 

Reciprocally, if $\eta\UUU $ is a pseudonatural equivalence, so is $\eta\UUU\LLL $. Therefore the multiplication of the induced pseudomonad is an inverse equivalence of $\eta\UUU\LLL $ and, by Lemma \ref{multidem}, we conclude that the induced pseudomonad is idempotent.
\end{proof}

We can avoid the coherence equations~\cite{Marmolejo1, SteveLack2, FLN} 
used to define the $2$-category of pseudoalgebras of a pseudomonad $\TTT $ 
when assuming that $\TTT $ is idempotent. 

\begin{defi}[Pseudoalgebras]\label{pseudoalgebraidem}
Let $(\TTT, \mu , \eta ,  \Lambda, \rho, \Gamma)$ be an idempotent pseudomonad on a $2$-category $\hhh $. We define the $2$-category of $\TTT $-\textit{pseudoalgebras} $\mathsf{Ps}\textrm{-}\TTT\textrm{-}\Alg $ as follows:
\begin{itemize}\renewcommand\labelitemi{--}
\item Objects: the objects of $\mathsf{Ps}\textrm{-}\TTT\textrm{-}\Alg $ are the objects $X$ of $\hhh $ such that 
$$\eta _ {{}_X}: X\to \TTT (X) $$
is an equivalence;
\item The inclusion $\obj(\mathsf{Ps}\textrm{-}\TTT\textrm{-}\Alg)\to \obj (\hhh ) $ extends to a full inclusion $2$-functor
$$\III :\mathsf{Ps}\textrm{-}\TTT\textrm{-}\Alg\to \hhh $$ 
\end{itemize}
In other words, the inclusion $\III : \mathsf{Ps}\textrm{-}\TTT\textrm{-}\Alg\to\hhh $ is defined to be final among the full inclusions $\widehat{\III }: \aaa\to\hhh $ such that $\eta\widehat{\III } $ is a pseudonatural equivalence.

If $\eta _ {{}_X}: X\to \TTT (X) $ is an equivalence, $X$ can be endowed with a pseudoalgebra structure and the 
left adjoint $a:\TTT (X)\to X $ to $\eta _ {{}_X}: X\to \TTT (X) $ is called a pseudoalgebra structure to $X$.
Because we could describe $\mathsf{Ps}\textrm{-}\TTT\textrm{-}\Alg$ by means of 
pseudoalgebras/pseudoalgebra structures, 
we often denote the objects of $\mathsf{Ps}\textrm{-}\TTT\textrm{-}\Alg$ by small letters $a, b $.
\end{defi}

\begin{theo}[Eilenberg-Moore biadjunction]\label{FACTEILENBERG}
Let $(\TTT, \mu , \eta ,  \Lambda, \rho, \Gamma)$ be an idempotent pseudomonad on a $2$-category $\hhh $. There is a unique pseudofunctor $\LLL ^{{}^\TTT }$ such that
$$\xymatrix{ \hhh\ar[rr]|-{\TTT }\ar[dr]|-{\LLL ^{{}^\TTT }}&&\hhh\\
&\mathsf{Ps}\textrm{-}\TTT\textrm{-}\Alg\ar[ur]|-{\III } &
}$$ 
is a commutative diagram. Furthermore, $\LLL ^{{}^\TTT } $ is left biadjoint to $\III $.
\end{theo}
\begin{proof}
Firstly, we define $\LLL ^{{}^\TTT } (X): = \TTT (X) $. On the one hand, it is well defined, since, by Lemma~\ref{multidem}, 
$$\eta\TTT : \TTT\longrightarrow \TTT ^2 $$
is a pseudonatural equivalence. On the other hand, the uniqueness of $\LLL ^{{} ^\TTT} $ is a consequence of  $\III $ being a monomorphism. 

Now, it remains to show that $\LLL ^{{}^\TTT } $ is left biadjoint to $\III $.  By abuse of language, if $a$ is an object of $\mathsf{Ps}\textrm{-}\TTT\textrm{-}\Alg $, we denote by $a$ its pseudoalgebra structure (of Definition~\ref{pseudoalgebraidem}). Then we define the mutually inverse equivalences below
\begin{equation*}
\begin{aligned}
\mathsf{Ps}\textrm{-}\TTT\textrm{-}\Alg (\TTT (X), b) &\to  \hhh ( X, \III (b))&\\
                                              f&\mapsto  f\eta _ {{}_X}&\\
																							\alpha &\mapsto \alpha \ast\Id _ {{}_{\eta _ {{}_X}}} &
\end{aligned}
\qquad
\begin{aligned}																																										
\hhh ( X, \III (b)) &\to  \mathsf{Ps}\textrm{-}\TTT\textrm{-}\Alg (\TTT (X), b)&\\
 g &\mapsto  bT(g)&\\
 \beta &\mapsto \Id _ {{}_b}\ast T(\beta )&
\end{aligned}     
\end{equation*}
It completes the proof that $\LLL ^{{}^\TTT }\dashv \III $.             																							
\end{proof}

Theorem~\ref{Eilenbergidem} shows that this biadjunction $\LLL ^{{}^\TTT } \dashv \III $ satisfies the expected universal property~\cite{SteveLack2} of the $2$-category of pseudoalgebras, which is the Eilenberg-Moore factorization. In other words, we prove that our definition of $\mathsf{Ps}\textrm{-}\TTT\textrm{-}\Alg$ for idempotent pseudomonads $\TTT $ agrees with the usual definition~\cite{CME, SteveLack2, Marmolejo1, Street4} of pseudoalgebras for a pseudomonad.

\begin{theo}[Eilenberg-Moore]\label{Eilenbergidem}
If $\LLL\dashv \UUU$ is a biadjunction which induces an idempotent pseudomonad $(\TTT, \mu , \eta ,  \Lambda, \rho, \Gamma)$, then we have a unique comparison pseudofunctor $\KKK :\bbb\to \mathsf{Ps}\textrm{-}\TTT\textrm{-}\Alg $ such that 
$$\xymatrix{\bbb \ar[r]^-{\KKK }\ar[rd]_-{\UUU }& \mathsf{Ps} \textrm{-} \TTT\textrm{-}\Alg \ar[d]^-{\III }& \aaa \ar[r]^-{\LLL ^{{}^\TTT } }\ar[rd]_-{\LLL }& \mathsf{Ps} \textrm{-} \TTT\textrm{-}\Alg \\
&\aaa &&\bbb\ar[u]_-{\KKK } }$$
commute. 
\end{theo}
\begin{proof}
It is enough to define $\KKK (X) = \UUU (X) $ and $\KKK (f) = \UUU (f) $. This is well defined, since, by Lemma~\ref{suportfact}, 
$\eta\UUU : \UUU \longrightarrow \TTT\UUU  $
is a pseudonatural equivalence.
\end{proof}

Actually, in $2\textrm{-}\CAT$, every biadjunction $\LLL\dashv \UUU$ induces a comparison pseudofunctor 
and an Eilenberg-Moore factorization~\cite{CME} as above, 
in which $\TTT = \UUU\LLL $ denotes the induced pseudomonad.  
When the comparison pseudofunctor $\KKK $ is a biequivalence, we say that $\UUU $ is pseudomonadic. 
Although there is the Beck's theorem for pseudomonads~\cite{CME, Her, FLN}, the setting of idempotent pseudomonads is simpler.

\begin{theo}\label{Idem}
Let  $\LLL\dashv\UUU $ be a biadjunction. The pseudofunctor $\UUU $ is a local equivalence (or, equivalently, the counit is a pseudonatural
equivalence) if and only if $\UUU $ is pseudomonadic and the induced pseudomonad is idempotent.
\end{theo}
\begin{proof}
Firstly, if the counit $\varepsilon $ of the biadjunction of $\LLL\dashv\UUU $ is a pseudonatural equivalence, then $\mu : = \UUU\varepsilon\LLL $ is a pseudonatural equivalence as well. And, thereby, the induced pseudomonad is idempotent. Now, if $a:\TTT (X)\to X $ is a pseudoalgebra structure to $X$, we have that 
$$\xymatrix{\KKK (\LLL (X)) = \TTT (X) \ar[r]_-{\simeq }^-{a} & X.}$$
Thereby $\UUU $ is pseudomonadic.

Reciprocally, if $\LLL\dashv \UUU $ induces an idempotent pseudomonad and $\UUU $ is pseudomonadic, then we have that $\III\circ\KKK = \UUU $, $\KKK $ is a biequivalence and $\III $ is a local equivalence. Thereby $\UUU $ is a local equivalence and $\varepsilon $  is a pseudonatural equivalence.   
\end{proof}

In descent theory, one needs conditions to decide if a given object can be endowed with a pseudoalgebra structure. Idempotent pseudomonads provide the following simplification.

\begin{theo}\label{algebra}
Let $\TTT = (\TTT, \mu , \eta ,  \Lambda, \rho, \Gamma)$ be an idempotent pseudomonad on $\hhh $. Given an object $X$ of $\hhh $, the following conditions are equivalent:
\begin{enumerate}
\item The object $X$ can be endowed with a $\TTT$-pseudoalgebra structure;
\item $\eta _ {{}_X}:X\to\TTT (X) $ is a pseudosection, \textit{i.e.} there is $a:\TTT (X)\to X $ such that 
$ a\eta _ {{}_X}\cong \Id _ {{}_X} $;
\item $\eta _ {{}_X}:X\to\TTT (X) $ is an equivalence.
\end{enumerate}
\end{theo}
\begin{proof}
Assume that $\eta _ {{}_X}:X\to\TTT (X) $ is a pseudosection. By hypothesis, there is $a:\TTT (X)\to X $ such that 
$ a\eta _ {{}_X}\cong \Id _ {{}_X}$.
Thereby
$$\eta _ {{}_X}a \cong  \TTT (a)\eta _{{}_{\TTT (X)}}
               \cong  \TTT (a)\TTT (\eta _ {{}_X})
							 \cong  \TTT (a\eta _ {{}_X})
							\cong  \Id _ {{}_{\TTT (X)}}	.$$	
Hence $\eta _ {{}_X} $ is an equivalence.
\end{proof}

\subsection{Biadjoint Triangle Theorem}\label{btt}

The main result of this formal approach is somehow related to distributive laws of 
pseudomonads~\cite{Marmolejo1, Marmolejo2}.  
However, we choose a more direct approach, 
avoiding some technicalities of distributive laws unnecessary to our setting. 
To give such direct approach, we use the Biadjoint Triangle Theorem \ref{biadjointtriangle}.

Precisely, we give a bicategorical analogue (for idempotent pseudomonads) 
of an adjoint triangle theorem~\cite{Dubuc, Barr, Power88}. It is important to note that this bicategorical version holds for more general 
biadjoint triangles~\cite{FLN, FLN3, FLN2}, so that our restriction to the idempotent version is due to our scope.

\begin{lem}\label{impor}
Let $(\LLL \dashv \UUU , \eta , \varepsilon )$ and $(\widehat{\LLL } \dashv \widehat{\UUU }, \widehat{\eta } , \widehat{\varepsilon } ) $ be biadjunctions. Assume that $\widehat{\LLL }\dashv \widehat{\UUU }$ induces an idempotent pseudomonad and that there is a pseudonatural equivalence
$$\xymatrix{  \aaa\ar@{}[drr]|{\simeq } &&\bbb\ar[ll]|-{\EEE }\\
&\ccc\ar[ul]|-{\LLL }\ar[ur]|-{\widehat{\LLL }}&                                                                               
}$$ 
If $\eta _ {{}_X} $ is a pseudosection, then $\widehat{\eta } _ {{}_X} $ is an equivalence.
\end{lem}
\begin{proof}
Let $X$ be an object of $\ccc $ such that $ \eta_ {{}_X} : X\to \UUU\LLL (X) $ is  pseudosection. By Theorem \ref{algebra}, it is enough to prove that $\widehat{\eta } _ {{}_X} $ is a pseudosection, because the pseudomonad induced by $\widehat{\LLL}\dashv \widehat{\UUU } $ is idempotent.

To prove that $\widehat{\eta } _ {{}_X} $ is a pseudosection, we construct a pseudonatural transformation $\alpha : \widehat{\UUU}\widehat{\LLL }\longrightarrow \UUU\LLL $ such that there is an invertible modification
$$\xymatrix{  &\Id _ {{}_{\ccc}}\ar@{}[d]|{\cong }\ar[ld]|{\widehat{\eta } }\ar[rd]|{\eta }&\\
\widehat{\UUU}\widehat{\LLL }\ar[rr]|{\alpha } && \UUU\LLL
}$$
Without losing generality, we assume that $\EEE\circ\widehat{\LLL}=\LLL $. Then we define
$\alpha := (\UUU\EEE\widehat{\varepsilon }\widehat{\LLL }) (\eta \widehat{\UUU }\widehat{\LLL }) $. 
Indeed,
$$
\alpha \widehat{\eta } =  (\UUU \EEE \widehat{\varepsilon }\widehat{\LLL }) \left( \eta \widehat{\UUU }\widehat{\LLL }\right) (\widehat{\eta })
                            \cong  (\UUU \EEE \widehat{\varepsilon }\widehat{\LLL }) \left( \UUU\LLL\widehat{\eta }\right) (\eta )
														\cong  (\UUU \EEE \widehat{\varepsilon }\widehat{\LLL }) \left( \UUU\EEE\widehat{\LLL }\widehat{\eta }\right) (\eta )
														\cong \eta $$
Therefore, if $\eta _ {{}_X} $ is a pseudosection, so is $\widehat{\eta} _ {{}_X} $. And, as mentioned, by Theorem \ref{algebra}, if $\widehat{\eta} _ {{}_X} $ is a pseudosection, it is an equivalence.
\end{proof}

Let $\widehat{\TTT }$ be the idempotent pseudomonad induced by 
$\widehat{\LLL }\dashv \widehat{\UUU } $ and $\TTT $ the pseudomonad induced by 
$\LLL\dashv\UUU $. Then Lemma~\ref{impor} could be written as follows:

\textit{If $X$ is an object of $\ccc $ that can be endowed with a $\TTT $-pseudoalgebra structure, 
then $X$ can be endowed with a $\widehat{\TTT } $-pseudoalgebra structure, 
provided that there is a pseudonatural equivalence} $\EEE\widehat{\LLL }\simeq \LLL $.

\begin{theo}\label{biadjointtriangle}
Let $(\LLL \dashv \UUU , \eta , \varepsilon )$ and $(\widehat{\LLL }  \dashv \widehat{\UUU }, \widehat{\eta } , \widehat{\varepsilon } ) $ be biadjunctions such that their right biadjoints are local equivalences. If there is a pseudonatural equivalence
$$\xymatrix{  \aaa\ar@{}[drr]|{\simeq } &&\bbb\ar[ll]|-{\EEE }\\
&\ccc\ar[ul]|-{\LLL }\ar[ur]|-{\widehat{\LLL }}&
}$$ 
then $\EEE $ is left biadjoint to a pseudofunctor $\RRR $ which is a local equivalence.
\end{theo}
\begin{proof}
It is enough to define $\RRR := \widehat{\LLL } \UUU$. By Lemma \ref{impor}, 
$\left(\widehat{\eta } {\UUU }\right): \UUU\longrightarrow \widehat{\UUU }\widehat{\LLL }\UUU = \widehat{ \UUU } \RRR $ is a pseudonatural equivalence. 
Thereby we get
$$
\aaa (\EEE (b), a) \simeq  \aaa (\EEE \widehat{\LLL }\widehat{\UUU } (b) , a)\simeq  \aaa (\LLL\widehat{\UUU } (b), a)
								 \simeq  \ccc (\widehat { \UUU } (b) , \UUU (a) )\simeq  \ccc ( \widehat { \UUU } (b) , \widehat{\UUU}\RRR (a) )
								 \simeq  \bbb (b, \RRR (a) ).
$$
This completes the proof that $\RRR $ is right biadjoint to $\EEE $.
\end{proof}

Assume that $\AAA : \aaa\to\bbb $ and $\BBB :\bbb\to\ccc $ are pseudomonadic pseudofunctors, and their induced pseudomonads are idempotent. Then it is obvious that $\BBB\circ \AAA : \aaa\to\ccc $ is also pseudomonadic and induces an idempotent pseudomonad. Indeed, by Theorem~\ref{Idem}, this statement is equivalent to: \textit{compositions of right biadjoint local equivalences are right biadjoint local equivalences as well}.

\begin{cor}\label{square}
Assume that there is a pseudonatural equivalence
$$\xymatrix{  \aaa\ar@{}[dr]|-{\simeq } &\hhh\ar[l]|-{\EEE }\\
\bbb\ar[u]|-{\LLL _ {{}_{{}_{{}_\AAA}}} } &\ccc\ar[l]|-{\LLL _ {{}_{{}_{{}_\BBB}}}}\ar[u]|-{\LLL _ {{}_{{}_{{}_\CCCC}}}}
}$$ 
such that $\LLL _ {{}_{{}_{{}_\AAA}}}\dashv \AAA $, $\LLL _ {{}_{{}_{{}_\BBB}}}\dashv \BBB $ and $\LLL _ {{}_{{}_{{}_\CCCC}}}\dashv \CCCC $ are pseudomonadic biadjunctions inducing idempotent pseudomonads $\TTT _ {{}_ \AAA}, \TTT _ {{}_ \BBB}, \TTT _ {{}_ \CCCC }$. 
Then $\EEE\dashv\RRR $ and $\RRR $ is a local equivalence. 

In particular, if $(X, a) $ is a $\TTT _ {{}_ \BBB}$-pseudoalgebra that can be endowed with a $\TTT _ {{}_ \AAA}$-pseudoalgebra structure, then $X$ can be endowed with a $\TTT _ {{}_ \CCCC }$-pseudoalgebra structure as well.
\end{cor}

Lemma~\ref{impor} and Corollary~\ref{square} are results on our formal approach to descent theory, 
\textit{i.e.} they give conditions to decide whether a given object can be endowed 
with a pseudoalgebra structure. 
In fact, most of the theorems proved in this paper are consequences of successive applications of these results, 
including B\'{e}nabou-Roubaud Theorem and other theorems within the context of \cite{Facets, Facets2}. 
However it does not deal with the technical ``almost descent'' aspects, which follow from the results on
$\mathfrak{F}$-comparisons below.

\subsection{$\mathfrak{F}$-comparisons}\label{tecnicoalmost}
Instead of restricting attention to objects that can be endowed with a pseudoalgebra structure, 
we often are interested in almost descent and descent objects as well. 
In the context of idempotent pseudomonads, these are objects that possibly do not have 
pseudoalgebra structure but have comparison $1$-cells belonging to special classes of morphisms.

In this subsection, every $2$-category $\hhh $ is assumed to be endowed with a 
special subclass of morphisms $\mathfrak{F}_ {{}_{\hhh}}$ satisfying the following properties:
\begin{itemize}\renewcommand\labelitemi{--}
\item Every equivalence of $\hhh $ belongs to $\mathfrak{F}_ {{}_{\hhh}}$;
\item $\mathfrak{F}_ {{}_{\hhh}}$ is closed under composition;
\item If there is an invertible $2$-cell $f\Rightarrow h$ in $\hhh $ such that $ f\in \mathfrak{F}_ {{}_{\hhh}}$, then $h\in \mathfrak{F}_ {{}_{\hhh}}$;
\item (Left) cancellation property: if $fg$ and $f$ belong to $\mathfrak{F}_ {{}_{\hhh}}$, $g$ belongs to $\mathfrak{F}_ {{}_{\hhh}}$ as well.
\end{itemize}
If $f$ is a morphism of $\hhh $ that belongs to $\mathfrak{F}_ {{}_{\hhh}}$, we say that $f$ is an \textit{$\mathfrak{F}_ {{}_{\hhh}}$-morphism}.

\begin{rem}
Recall that a morphism in a $2$-category is faithful/fully faithful if its images by the representable $2$-functors are faithful/fully faithful.
Given any $2$-category $\hhh $, there are at least three important examples of subclasses of morphisms satisfying the properties above.
The first class is the class of equivalences of $\hhh $.  The others are respectively the classes of faithful and fully faithful morphisms of $\hhh $.
\end{rem}

\begin{defi}\label{almostdescent}
Let $(\TTT, \mu , \eta ,  \Lambda, \rho, \Gamma)$ be an idempotent pseudomonad on a $2$-category $\hhh $. An object $X$ is an \textit{$(\mathfrak{F}_ {{}_{\hhh}}, \TTT )$-object} if the comparison
$\eta _ {{}_X} : X\to\TTT (X) $
is an $\mathfrak{F}_ {{}_{\hhh}}$-morphism.

We say that a pseudofunctor $\mathcal{E}:\hhh\to\hhh $ \textit{preserves $(\mathfrak{F}_ {{}_{\hhh}}, \TTT )$-objects} if it takes $(\mathfrak{F}_ {{}_{\hhh}}, \TTT )$-objects to $(\mathfrak{F}_ {{}_{\hhh}}, \TTT )$-objects.  
\end{defi}

Theorem \ref{descenTsquare} follows from the construction given in the proof of Lemma~\ref{impor}. 
Similarly to Corollary \ref{square}, Theorem \ref{descenTsquare} is  a commutativity result for $(\mathfrak{F}_ {{}_{\hhh}}, \TTT )$-objects.

\begin{theo}\label{descenTsquare}
Let
$$\xymatrix{  \aaa\ar@{}[dr]|-{\simeq } &\hhh\ar[l]|-{\EEE }\\
\bbb\ar[u]|-{\LLL _ {{}_{{}_{{}_\AAA}}} } &\ccc\ar[l]|-{\LLL _ {{}_{{}_{{}_\BBB}}}}\ar[u]|-{\LLL _ {{}_{{}_{{}_\CCCC}}}}
}$$ 
be a pseudonatural equivalence such that $\LLL _ {{}_{{}_{{}_\AAA}}}\dashv \AAA $, $\LLL _ {{}_{{}_{{}_\BBB}}}\dashv \BBB $ and $\LLL _ {{}_{{}_{{}_\CCCC}}}\dashv \CCCC $ are biadjunctions inducing pseudomonads $\TTT _ {{}_ \AAA}, \TTT _ {{}_ \BBB}, \TTT _ {{}_ \CCCC }$. Also, we denote by $\TTT $ the pseudomonad induced by the biadjunction $\LLL _ {{}_{{}_{{}_\AAA}}}\LLL _ {{}_{{}_{{}_\BBB }}}\dashv \BBB\AAA $.  

Assume that all the right biadjoints are local equivalences, $\BBB $ takes $\mathfrak{F}_ {{}_{\bbb}}$-morphisms to $\mathfrak{F}_ {{}_{\ccc}}$-morphisms and $\TTT _ {{}_ \CCCC }$ preserves $(\mathfrak{F}_ {{}_{\ccc }}, \TTT ) $-objects. If $X$ is a $(\mathfrak{F}_ {{}_{\ccc }}, \TTT _ {{}_{\BBB }} ) $-object of $\ccc $
and $\LLL _ {{}_{{}_{{}_\BBB}}}(X) $ is a $(\mathfrak{F} _ {{}_{\bbb }}, \TTT _ {{}_\AAA } )$-object, then $X$ is a $(\mathfrak{F} _ {{}_{\ccc }}, \TTT _ {{}_\CCCC })$-object as well.
\end{theo}
\begin{proof}
By the composition of biadjunctions, the unit $\eta $ of $\TTT $ is such that, for each object $Y$ of $\ccc $, $$\eta _ {{}_Y} \cong \left(\BBB\eta ^{{}^\AAA }\LLL _ {{}_{{}_{{}_\BBB }}}\right)_ {{}_Y}\, \, \, \eta ^{{}^\BBB } _ {{}_Y}.$$ 

Let $X$ be an object satisfying the hypotheses of the theorem. We have that $\left(\BBB\eta ^{{}^\AAA }\LLL _ {{}_{{}_{{}_\BBB }}}\right)_ {{}_X}$
and  $\eta ^{{}^\BBB } _ {{}_X}$ are  $\mathfrak{F}_{{}_{\ccc}}$-morphisms. Hence, by the closure under composition and by the 
invertible $2$-cell above, we conclude that $\eta _ {{}_X}$ is an $\mathfrak{F}_{{}_{\ccc}}$-morphism.

By the proof of Lemma~\ref{impor}, there is a pseudonatural transformation
$\alpha :\TTT _ {{}_\CCCC }\longrightarrow \TTT  $
such that we have in particular an invertible $2$-cell
\small
$$\xymatrix{  &X\ar@{}[d]|{\cong }\ar[ld]|{\eta ^{{}^\CCCC }_{{}_X} }\ar[rd]|{\eta _ {{}_X} }&\\
\TTT _ {{}_\CCCC }(X) \ar[rr]|-{\alpha _ {{}_X}} && \TTT (X).
}$$ 
\normalsize
 
By the left cancellation property of the subclass $\mathfrak{F} _ {{}_{\ccc }} $ and by the invertible $2$-cell above,
we only need to prove that  $\alpha _ {{}_X} $ is an $\mathfrak{F}_{{}_{\ccc}}$-morphism to complete our proof that
$\eta ^{{}^\CCCC }_{{}_X} $ is an $\mathfrak{F}_{{}_{\ccc }}$-morphism.

Recall that $\alpha_ {{}_X}$ is defined by $$\alpha _ {{}_X}:= (\BBB\AAA\varepsilon ^{{}^\CCCC } \LLL _ {{}_{{}_{{}_\CCCC}}}) _ {{}_X}\,\,\, (\eta \TTT _{{}_\CCCC } ) _ {{}_X}, $$ 
in which $\varepsilon ^{{}^\CCCC }$ is the counit of the biadjunction $\LLL _ {{}_{{}_{{}_\CCCC}}}\dashv \CCCC $. 

Firstly, $(\BBB\AAA\varepsilon ^{{}^\CCCC } \LLL _ {{}_{{}_{{}_\CCCC}}}) _ {{}_ X}$ is an equivalence. Secondly, since $\TTT _{{}_\CCCC }$ preserves $(\mathfrak{F}_ {{}_{\ccc }}, \TTT ) $-objects, $(\eta \TTT _{{}_\CCCC }) _ {{}_X}\cong (\TTT _{{}_\CCCC }\eta ) _ {{}_X}$ is an $\mathfrak{F}_{{}_{\ccc }}$-morphism. Therefore $\alpha _ {{}_X}$ is a composition of  $\mathfrak{F}_{{}_{\ccc }}$-morphisms and, hence, an $\mathfrak{F}_{{}_{\ccc }}$-morphism as well.
\end{proof}

The result above can be seen as a generalization of 
Corollary \ref{square}, since we get that corollary from  Theorem \ref{descenTsquare} by defining the classes $\mathfrak{F}_ {{}_{\aaa }}$, $ \mathfrak{F}_ {{}_{\bbb }}$, $\mathfrak{F}_ {{}_{\ccc }}$ to be the classes of equivalences.

\section{Pseudo-Kan Extensions}\label{KAN}

It is known that the descent category and the category of 
algebras are $2$-categorical limits (see, for instance, \cite{Street3, Street4, Galois}).
Thereby, our standpoint is to deal with the context of \cite{Facets, Facets2} strictly guided by bilimits results. 

For the sake of this aim, we focus our study on the pseudomonads coming from the bicategorical
analogue of the notion of 
right Kan extension. Actually, since the concept of ``right Kan extension'' plays the leading role in this 
work, ``(pseudo-)Kan extension'' means always right (pseudo-)Kan extension, 
while we always make the word ``left'' explicit when we refer to the dual notion. 

We explain below why we need to use a bicategorical notion of Kan extension, 
instead of employing the fully developed theory of enriched Kan extensions. 
The natural setting of (classical) descent theory is $2\textrm{-}\CAT $.
Although we can construct the bilimits related to descent theory as (enriched/strict) 
Kan extensions of $2$-functors in the $3$-category 
of $2$-categories, $2$-functors, $2$-natural transformations and modifications 
(see \cite{Street3, Power89, FLN, FLN3}), 
the necessary replacements~\cite{SteveLack,FLN} do 
not make computations and formal manipulations any easier. 

Furthermore, most of the transformations between $2$-functors that are necessary in the development of the 
theory are pseudonatural. 
Thus, to work within the ``strict world'' 
without employing repeatedly coherence theorems (such as the general coherence result of 
\cite{SteveLack}), we would need to add hypotheses to assure that usual Kan extensions of 
pseudonaturally equivalent diagrams are pseudonaturally equivalent. 
This is not true in most of the cases: it is easy to construct examples of pseudonaturally isomorphic 
diagrams such that their usual Kan extensions are not pseudonaturally equivalent. 
For instance, consider the $2$-category $\aaa $ below.
$$\xymatrix{ \mathsf{1} \ar@<0.3 ex>[r]^{d^0 } \ar@<-0.3 ex>[r]_{d^1 } & \mathsf{2} }$$
The $2$-category $\aaa $ has no nontrivial $2$-cells. Assume that $\dot{\aaa } $ is the $2$-category obtained from $\aaa $ adding an initial object $\mathsf{0}$, 
with full inclusion $\t : \aaa \to \dot{\aaa } $. Now, 
if $\ast $ is the terminal category and $\nabla \mathsf{2} $ is the category with two objects and one 
isomorphism between them (\textit{i.e.} $\nabla \mathsf{2} $ is the localization of the preorder 
$\mathsf{2}$ w.r.t. all morphisms), then there are two $2$-natural isomorphism classes of diagrams 
$\aaa \to \CAT $ of the type below, while all such diagrams are pseudonaturally isomorphic.
$$\xymatrix{ \ast \ar@<0.3 ex>[r] \ar@<-0.3 ex>[r] & \nabla \mathsf{2} }$$
These $2$-natural isomorphism classes give pseudonaturally nonequivalent Kan extensions along $\t $.  
More precisely, if $\mathcal{X},\mathcal{Y}: \aaa \to \CAT $ are such that  
$\mathcal{X}(\mathsf{1} )=\mathcal{Y}(\mathsf{1} )= \ast$, 
$\mathcal{X}(\mathsf{2} )= \mathcal{Y}(\mathsf{2} )= \nabla \mathsf{2} $, 
$\mathcal{X}( d^0 ) \neq \mathcal{X}( d^1 ) $ and 
$\mathcal{Y}( d^0 ) = \mathcal{Y}( d^1 ) $; then 
$\Ran _ \t \mathcal{X}(\mathsf{0} ) = \emptyset $, 
while $\Ran _\t \mathcal{Y}(\mathsf{0} ) = \ast $. 
Therefore $\Ran _\t \mathcal{X}$ and 
$\Ran _ \t \mathcal{Y}$ are not pseudonaturally equivalent, 
while $\mathcal{X}$ is pseudonaturally isomorphic to $\mathcal{Y}$. 
 
The usual Kan extensions behave well if we add extra hypotheses related to flexible diagrams 
(see \cite{Flexible, Power, Power89, SteveLack, FLN}). However, we do not give such 
restrictions and technicalities.  
Thereby we deal with the problems natively in the tricategory $2$-$\CAT$, 
without employing further coherence results. 
The first step is, hence, to understand the appropriate notion of Kan extension in this tricategory.

\subsection{The Definition}\label{pseudoKAN}

In a given tricategory, if $\t : a\to b$, $f: a\to c $ are $1$-cells, 
we might consider that the formal right Kan extension of $f$ along $\t $ is the right $2$-reflection 
of $f$ along the $2$-functor $[\t , c]:[b,c]\to [a,c] $. That is to say, 
if it exists for all $f:a\to c $, the (formal) global Kan extension along $\t : a\to b $ would be a $2$-functor 
$ [a,c]\to [b,c] $ right $2$-adjoint to $ [\t , c] : [b, c]\to [a,c] $. 
But, in important cases, such concept is very restrictive, because 
it does not take into account the bicategorical structure of the hom-$2$-categories of the tricategory. 
Hence, it is possible to consider other notions of Kan extension, 
corresponding to the two other important notions of adjunctions between $2$-categories~\cite{Gray3}, 
that is to say, lax adjunction and biadjunction. 
For instance, Gray~\cite{Gray} studied the notion of lax-Kan extension.

We also consider an alternative notion of Kan extension in our tricategory $2$-$\CAT $, 
that is to say, the notion of pseudo-Kan extension, introduced in \cite{FLN}.
In our case, the need of this concept comes from the fact that, even with many assumptions, 
the (formal) Kan extension of a pseudofunctor may not exist. 
Furthermore, we prove in Section \ref{descent} that the descent object (descent category) and the 
Eilenberg-Moore object (Eilenberg-Moore category) can be easily described using our language.

\textit{Henceforth, $\aaa, \bbb $ always denote small $2$-categories.} 
If $\t : \aaa\to \bbb $ and $\AAA : \aaa\to\hhh $ are pseudofunctors, a right  
\textit{pseudo-Kan extension} of $\AAA $ along $\t $, denoted by $\Ps\Ran _\t \AAA $, is, if it exists, 
a right bicategorical reflection of $\AAA : \aaa\to\hhh $ along the pseudofunctor 
$\left[ \t , \hhh \right] _{PS} : \left[ \bbb ,\hhh\right] _ {PS} \to  \left[ \aaa ,\hhh\right] _ {PS} $.
Although it is omitted in our notation, every right pseudo-Kan extension comes with a universal arrow $$ \varepsilon _ {{}_{\AAA }} : \left(\Ps\Ran _\t \AAA\right) \circ \t\longrightarrow \AAA $$
 by Definition \ref{bicategoricalreflection} of right bicategorical reflection. Furthermore, by Remark \ref{rightbicategoricalreflectioncounit} we could actually give the definition of pseudo-Kan extension via the property of 
this universal arrow. 
That is to say, $(\Ps\Ran _\t \AAA , \varepsilon _ {{}_{\AAA }}) $ is the right pseudo-Kan extension of $\AAA $ along $\t $ if and only if
$$ \left[ \bbb ,\hhh\right] _ {PS} (-, \Ps\Ran _\t \AAA )\to \left[ \aaa ,\hhh\right] _ {PS} (- \circ \t,  \AAA ) : \qquad \alpha \mapsto \varepsilon _ {{}_{\AAA }}\, \left( \alpha \t \right)
$$
defines a pseudonatural equivalence. By uniqueness (up to equivalence) of bicategorical reflections, pseudo-Kan extensions are unique up to pseudonatural equivalence.

The \textit{global right pseudo-Kan extension} along $\t : \aaa\to\bbb $ w.r.t. a $2$-category $\hhh $ is the right biadjoint of 
$\left[ \t , \hhh \right] _{PS} $, provided that it exists. That is to say, a 
pseudofunctor $\Ps\Ran _ \t :\left[ \aaa ,\hhh\right] _ {PS} \to  \left[ \bbb ,\hhh\right] _ {PS} $
such that $\left[ \t , \hhh \right] _{PS}\dashv \Ps\Ran _ \t $.

Herein, the expression \textit{Kan extension} refers to the usual notion of Kan extension 
in $\CAT$-enriched category theory. 
That is to say, if $\t : \aaa\to \bbb $ and $\AAA : \aaa\to\hhh $ 
are $2$-functors, the (right) Kan extension of $\AAA $ along $\t $, 
denoted by $\Ran _ \t \AAA : \bbb\to\hhh $, is (if it exists) the right $2$-reflection of $\AAA $ 
along the $2$-functor $\left[ \t , \hhh \right] $.
And the global Kan extension is a right $2$-adjoint of 
$\left[ \t , \hhh \right]: \left[ \bbb ,\hhh\right] \to  \left[ \aaa ,\hhh\right] ,$
in which $\left[ \bbb ,\hhh\right]$ denotes the $2$-category of 
$2$-functors $\bbb\to\hhh $, $\CAT $-natural transformations and modifications.

If $\Ran _ \t \AAA $ exists, it is not generally true that $\Ran _ \t \AAA $ is 
pseudonaturally equivalent to $\Ps\Ran _ \t \AAA $. This is a coherence problem, 
related to  flexible diagrams~\cite{Power, SteveLack, Flexible, FLN} and to the 
construction of bilimits via strict $2$-limits~\cite{Street3, Street4}. For instance, in particular, 
using the results of \cite{FLN}, we can easily prove   
that, for a given pseudofunctor $\AAA : \aaa\to\hhh $ and a $2$-functor $\t : \aaa\to\bbb $, we can replace 
$\AAA $ by a pseudonaturally equivalent $2$-functor $\AAA ':\aaa\to\hhh $ such that 
$\Ran _ \t \AAA ' $ is equivalent to $\Ps\Ran _\t \AAA '\simeq \Ps\Ran _\t \AAA $, 
provided that $\hhh $ satisfies some completeness conditions 
(for instance, if $\hhh $ is $\CAT $-complete).      

In Section \ref{descent} we show that the descent category, as defined and studied in 
\cite{Street5, Facets, Facets2}, of a pseudocosimplicial object 
$D :\Delta\to \CAT $ is equivalent to $\Ps\Ran _\j D (0) $, 
in which $\j : \Delta\to \dot{\Delta } $ is the inclusion of the category of 
nonempty finite ordinals into the category of finite ordinals. 
Observe that the Kan extension of a cosimplicial object does not give the descent object: 
it gives an equalizer (which is the notion of descent for dimension $1$), 
although we might give the descent object via a Kan extension after replacing the 
(pseudo)cosimplicial objects by suitable strict versions of 
pseudocosimplicial objects as it is done in \ref{strictdescent}.

\subsection{Factorization}\label{factorization}
Our setting often reduces to the study of right pseudo-Kan extensions of 
pseudofunctors $\AAA : \aaa \to\hhh $ along $\t $, in which $\t : \aaa\to\dot{\aaa } $ is the full 
inclusion of a small $2$-category $\aaa $ into a small $2$-category $\dot{\aaa } $ which has only one 
extra object $\mathsf{a} $. 

\begin{defi}[$\mathsf{a}$-inclusion]\label{Adefinition}
A $2$-functor $\t : \aaa\to\dot{\aaa } $ is called an \textit{$\mathsf{a}$-inclusion}, 
if $\mathsf{a} $ is an object of $\dot{\aaa } $ and $\t$ is a fully faithful functor between small $2$-categories in which 
$$\obj (\dot{\aaa } ) = \obj (\aaa )\cup \left\{ \mathsf{a}\right\} $$
is a disjoint union.
\end{defi}

\begin{rem}\label{Bdefinition}
The terminology established in the definition above makes reference to the extra object. Hence, using this terminology, 
a full inclusion $\bbb\to\dot{\bbb } $ between small $2$-categories in which 
$\obj (\dot{\bbb } ) = \obj (\bbb )\cup \left\{ \mathsf{b}\right\} $
is called a $\mathsf{b} $-inclusion.
\end{rem}

\begin{rem}
Theorem \ref{pointwise} shows that a right pseudo-Kan extension of a pseudofunctor
 $\AAA : \aaa\to  \hhh $ along an $\mathsf{a}$-inclusion $\t $ is precisely $\AAA $ extended with a weighted
bilimit $\bilim (\dot{\aaa } (\mathsf{a} , \t - ), \AAA ) $ whenever such weighted bilimit exists. Thereby $\mathsf{a} $-inclusions are precisely
what we need to give statements and proofs on (weighted) bilimits via pseudo-Kan extensions.   
\end{rem}

In this setting, we have factorizations 
for pseudo-Kan extensions along 
$\mathsf{a}$-inclusions, which follow formally from the biadjunction 
$\left[ \t , \hhh\right] _ {PS}\dashv \Ps\Ran _ \t $.

\begin{theo}[Factorization]\label{fact}
Assume that 
$(\left[ \t , \hhh\right] _ {PS}\dashv \Ps\Ran _ \t, \eta, \varepsilon) $ is a biadjunction and $\t : \aaa\to\dot{\aaa } $ is an $\mathsf{a}$-inclusion.
If $\AAA : \dot{\aaa } \to\hhh $ is a pseudofunctor, $\mathsf{a}\neq b $ 
and $f:b\to\mathsf{a}$, $g:\mathsf{a}\to b $ are morphisms of $\dot{\aaa } $, 
we get induced ``factorizations'' (actually, invertible $2$-cells): 
$$\xymatrix@C=1em{ \AAA (b)\ar[rr]|{\AAA (f)}\ar[dr]|{f_{{}_{{}_\AAA }} }&& \AAA (\mathsf{a})\ar[ld]|{\eta ^\mathsf{a}_ {{}_{{}_{\AAA } }}} &
\AAA (\mathsf{a})\ar[rr]|{\AAA (g)}\ar[dr]|{\eta ^\mathsf{a}_ {{}_{{}_{\AAA } }} }&& \AAA (b)\\
&\Ps\Ran _\t (\AAA\circ\t ) (\mathsf{a})\ar@{}[u]|{\cong } & & &\Ps\Ran _\t (\AAA\circ\t ) (\mathsf{a})\ar@{}[u]|{\cong }\ar[ru]|{g_{{}_{{}_\AAA }}} &
}$$ 
\normalsize
in which 
$$
f_{{}_{{}_\AAA }} : = \Ps\Ran _ \t (\AAA\circ\t )(f)\circ\eta ^b_ {{}_{{}_{\AAA } }}\qquad\qquad
g_{{}_{{}_\AAA }} :=  \varepsilon _ {{}_{{}_{(\AAA\circ\t ) } }}^b\circ \Ps\Ran _ \t (\AAA\circ\t )(g)
$$
and $\eta ^\mathsf{a}_ {{}_{{}_{\AAA } }}$, $\varepsilon _ {{}_{{}_{(\AAA\circ\t )} }}^{b} $ are the $1$-cells induced by the components of $\eta $ 
and $\varepsilon $.
\end{theo}
\begin{proof}
By the triangular invertible modifications of Definition~\ref{adjunction},
$$
g_{{}_{{}_\AAA }}\circ\eta ^\mathsf{a}_ {{}_{{}_{\AAA } }} =  \varepsilon _ {{}_{{}_{(\AAA\circ\t )} }}^{b}\circ \Ps\Ran _\t (\AAA\circ\t )(g)\circ\eta ^\mathsf{a}_ {{}_{{}_{\AAA } }}\cong \varepsilon _ {{}_{{}_{(\AAA\circ\t )} }}^{b}\circ\eta ^{b}_ {{}_{{}_{\AAA } }}\circ \AAA (g)\cong  \AAA (g) $$
The factorization involving $\AAA (f) $ follows from the pseudonaturality of $\eta $.  
\end{proof}

\begin{rem}
Using the results of \ref{pointwisepseudokanextension} below, we find the factorizations above to be properties of 
(weighted) bilimits as we show in Section \ref{Elementary Examples}. For instance, we get the usual \ref{JTDESCENT} and the Eilenberg-Moore factorization, respectively,
in Sections \ref{descent} and \ref{EilenBERG}.
\end{rem}

\subsection{Weighted bilimits}\label{Spointwise}
Similarly to the usual approach for (enriched) Kan extensions, 
we define what should be called \textit{pointwise pseudo-Kan extension}. 
We prove that pseudo-Kan extensions exist and are pointwise whenever the codomain
has suitable weighted bilimits.

Pointwise (left) pseudo-Kan extensions  are constructed with weighted 
bi(co)limits~\cite{FLN, Power89, Street4, Street5}, the bicategorical analogue of (enriched) weighted 
(co)limits~\cite{Kelly, Dubuc2}. Thereby 
we list some needed results on weighted bilimits.

\begin{defi}[Weighted bilimit]
Let $\WWW : \aaa \to \CAT$ and $ \AAA :\aaa\to\hhh $ be pseudofunctors.
If it exists, a \textit{weighted bilimit} of $\AAA $ with weight $\WWW $ is an object of $\hhh $, denoted by 
$\left\{ \WWW, \AAA\right\}_ {\bi} $ or by $\bilim (\WWW, \AAA )$,  endowed with an equivalence (pseudonatural in $X$) 
$$\hhh (X,\bilim (\WWW, \AAA ) )  \simeq  \left[\aaa , \CAT\right]_{PS}(\WWW, \hhh (X, \AAA -) ),$$
which means that  $\bilim (\WWW, \AAA )$ is a bicategorical representation of
$$\hhh ^\op \to \CAT : \qquad X\mapsto \left[\aaa , \CAT\right]_{PS}(\WWW, \hhh (X, \AAA -) ).
$$
In other words, a weighted bilimit is, if it exists, the left bicategorical reflection of $\WWW $ along the  $2$-functor
$$ \hhh ^\op  \to \left[\aaa , \CAT\right]_{PS} : \qquad  X\mapsto\hhh (X, \AAA -), \quad f\mapsto \hhh (f, \AAA -). $$
By the uniqueness of the right bicategorical reflection,  $\bilim (\WWW, \AAA ) $ is 
unique up to equivalence. We refer to it as \textit{the} ($\WWW $-weighted) bilimit (of $\AAA $). 
\end{defi}

By Remark \ref{rightbicategoricalreflectioncounit}, in the context of the definition above, 
the weighted bilimit is a pair  $(\bilim (\WWW, \AAA ), \rho _{{}_{\AAA }})$
in which $\bilim (\WWW, \AAA )$ is an object of $\hhh $ and 
 $$\rho _{{}_{\bilim(\WWW , \AAA) }}: \WWW \longrightarrow \hhh ( \bilim (\WWW, \AAA ) , \AAA -) $$ 
is a pseudonatural transformation such that
$$ \hhh (X,\bilim (\WWW, \AAA ) )\to \left[\aaa , \CAT\right]_{PS}(\WWW, \hhh (X, \AAA -) ), \quad 
f\mapsto \hhh (f, \AAA -)\, \rho _{{}_{\bilim(\WWW , \AAA) }} $$
is an equivalence pseudonatural in $X$.

\begin{rem}[Weighted bicolimit]
The dual notion is called \textit{weighted bicolimit}. If it exists, given a pseudofunctor  $\WWW ' : \aaa ^\op \to \CAT$,
the \textit{$\WWW ' $-weighted bicolimit of $\AAA $}, denoted by  
$\WWW '\Asterisk _ {\bi }  \AAA $ or by $\bicolim (\WWW ',\AAA ) $, is an object of $\hhh $
endowed with an equivalence (pseudonatural in $X$)
$$[\aaa , \CAT]_{PS}(\WWW ', \hhh ( \AAA - , X ) )\simeq \hhh ( \bicolim (\WWW ' , \AAA) , X ) .$$
\end{rem}

\begin{rem}[Conical Bilimit]\label{conical}
Analogously to the enriched case, denoting by $\top$
the appropriate $2$-functor constantly
equal to the terminal category, the $\top $-weighted bi(co)limit of a pseudofunctor
 $\AAA : \aaa\to\hhh $ is 
 the \textit{conical bi(co)limit} of $\AAA $ provided that it exists.
\end{rem}

The $2$-category $\CAT $ is bicategorically complete, that is to say, it has all (small) weighted bilimits. 
Indeed, if $\WWW , \AAA : \aaa \to \CAT$ are pseudofunctors, we have that 
$$\bilim ( \WWW, \AAA )  \simeq  [\aaa , \CAT ]_{PS}(\WWW, \AAA ) .$$  
Moreover, from the bicategorical Yoneda lemma, we get the strong bicategorical Yoneda lemma. 

\begin{lem}[Yoneda Lemma]\label{YonedaBILIMIT}
Let $\AAA : \aaa \to\hhh $ be a pseudofunctor. There is an equivalence
$\bilim ( \aaa (X, -), \AAA )\simeq\AAA (X) $ pseudonatural in $X$.
\end{lem}

Before giving results on pointwise pseudo-Kan extensions, the following result, which is mainly used in Section~\ref{descent}, already gives a glimpse of the relation between weighted bilimits and pseudo-Kan extensions.

\begin{theo}\label{relation}
Let $\t : \aaa\to\bbb $, $\WWW : \aaa\to\CAT $ be pseudofunctors. If the left pseudo-Kan extension $\Ps\Lan _ \t \WWW $ exists and $\AAA : \bbb\to\hhh $ is a pseudofunctor, then there is an equivalence
$$\bilim ( \WWW , \AAA\circ\t ) \simeq \bilim ( \Ps\Lan _ {\t} \WWW , \AAA ) $$ 
either side existing if the other does.
\end{theo} 
\begin{proof}
Assuming that $\Ps\Lan _ \t \WWW $ exists, we have a pseudonatural equivalence between the $2$-functors 
$$X\mapsto \left[ \bbb , \CAT \right] _ {PS} (\WWW , \hhh (X, \AAA\circ\t -)) \quad\mbox{and}\quad X\mapsto \left[ \aaa , \CAT \right] _ {PS} (\Ps\Lan _ \t \WWW , \hhh (X , \AAA -)). $$
Therefore, assuming that any of the $2$-functors above has a bicategorical representation $Y$, we get that $Y$ is indeed a bicategorical representation of both $2$-functors. 
\end{proof}

\subsection{Pseudoends}\label{introducedpseudoend}
There is one important notion remaining to study weighted bilimits: the bicategorical analogue of the \textit{end}~\cite{Kelly, Dubuc2}. 
Below and in \cite{FLN}, we give a direct definition of pseudoend, avoiding the unnecessary work on
bicategorical analogues of the extranatural transformations
of the classical enriched case. In order to do this, we consider the usual characterization of the  
end of a $2$-functor $T: \aaa ^\op \times \aaa  \to \CAT$ in the strict/enriched case given by:    
$$\left[ \aaa ^\op \times \aaa , \CAT\right] (\aaa (-,-), T)  .$$
Herein, we do not work with ends and, hence, we reserve the integral sign to denote pseudoends. 
\begin{defi}[Pseudoend]
The \textit{pseudoend} of a pseudofunctor $T: \aaa ^\op \times \aaa  \to \CAT $ is  
$$ \int _ \aaa T := [\aaa ^\op\times \aaa , \CAT] _ {PS} ( \aaa (-,-), T)\simeq \bilim (\aaa (-,-), T) .$$
\end{defi}

In order to avoid possible confusions when stating results that involve iterated pseudoends, we often 
adopt a terminology in which the pseudoend is indexed by the variable as well. 
That is to say, we use the following terminology:
$$ \int _ {a\in\aaa } T(a,a) : = \int _ \aaa T . $$

This definition allows us to get analogues of the usual fundamental results on ends of the strict/enriched case~\cite{Kelly, Dubuc2}.
We start by:

\begin{prop}[Fundamental equivalence for pseudoends]\label{Q}
Let
$\AAA , \BBB : \aaa\to\hhh $ be pseudofunctors. There is a pseudonatural equivalence
$$ \int _ {a\in\aaa } \hhh (\AAA (a), \BBB (a) ) \simeq  [\aaa , \hhh ] _ {PS}( \AAA , \BBB ) .$$
\end{prop}
\begin{proof}
Firstly, observe that a pseudonatural transformation $$\alpha : \aaa (-,-)\longrightarrow \hhh (\AAA -, \BBB - ) $$ corresponds to a collection of $1$-cells
$\alpha _{{}_{(W,X)}}: \aaa (W,X)\to  \hhh (\AAA (W), \BBB (X) ) $ and collections of invertible $2$-cells 
\begin{eqnarray*}
\alpha _ {{}_{(Y,f)}} : \hhh (\AAA (Y) , \BBB (f))\alpha _ {{} _ {(Y,W)}}&\cong &  \alpha _ {{} _ {(Y,X)}}\aaa ( Y ,  f)\\  
\alpha _ {{}_{(f,Y)}} : \hhh (\AAA (f) , \BBB (Y))\alpha _ {{} _ {(X,Y)}}&\cong &  \alpha _ {{} _ {(W,Y)}}\aaa ( f , Y )
\end{eqnarray*}
such that, for each object $Y$ of $\aaa $,  $\alpha _{{}_{(Y,-)}} $ and $\alpha _{{}_{(-,Y)}} $ (with the invertible $2$-cells above) are pseudonatural transformations. In other words, pseudonatural transformations are transformations which are pseudonatural in each variable.

By the bicategorical Yoneda lemma, we get what we want:  such a pseudonatural transformation corresponds (up to isomorphism) to a collection of $1$-cells $$\gamma _ {{}_W}:= \alpha _ {{}_{W,W}}(\Id _ {{}_W}) : \AAA (W)\to\BBB (W) $$ with (coherent) invertible $2$-cells $\BBB (f)\circ\gamma _ {{}_W}\cong \gamma _ {{}_W}\circ \AAA (f) $.  
\end{proof}

Hence, the original bicategorical Yoneda lemma may be reinterpreted. Given a pseudofunctor $\AAA : \aaa \to\CAT $, we have an equivalence
$$ \int _ {a\in\aaa } \CAT (\aaa (X, a), \AAA (a) ) \simeq \AAA (X) $$
pseudonatural in $X$. 
Theorem \ref{Fubini} is the bicategorical analogue of the Fubini's Theorem in the enriched context.

\begin{theo}\label{Fubini}
Given a pseudofunctor $T: \aaa ^{\op }\times \bbb ^{\op }\times \aaa \times \bbb ^\op \to \CAT $,
there are pseudofunctors
$T^{{}^\bbb }: \aaa ^{\op }\times \aaa \to \CAT $ and $ T ^{{}^\aaa }: \bbb ^{\op }\times \bbb \to \CAT $  
such that
$$ T^{{}^\bbb }(X,Y) \cong \int _ {b\in\bbb } T(X, b, Y, b)\quad\mbox{and}\quad T^{{}^\aaa }(W, Z)\cong \int _ {a\in \aaa } T(a, W, a, Z). $$
Furthermore, $ \displaystyle\int _ {\aaa \times\bbb } T \simeq \int _ {\aaa }T^{{}^\bbb }  \simeq \int _ {\bbb }T^{{}^\aaa } $.
\end{theo}

We usually denote the iterated pseudoends of the result above by:
$$ \int _{a\in\aaa} \int _ {b\in \bbb } T(a,b,a,b) := \int _ {\aaa }T^{{}^\bbb }  \quad\mbox{and}\quad \int _ {b\in\bbb }\int _ {a\in\aaa } T(a,b,a,b) := \int _ {\bbb }T^{{}^\aaa }. $$

\subsection{Pointwise pseudo-Kan extensions}\label{pointwisepseudokanextension}
If we consider the full $2$-subcategory $ \hhh _ \YYY $ of $[\bbb ^{\op }, \CAT ]_ {PS} $ such that the objects of $ \hhh _ \YYY $ are the bicategorically representable pseudofunctors
of a $2$-category $\hhh $, the Yoneda embedding $\YYY : \hhh \to \hhh _ \YYY $ is a biequivalence: that is to say, we can choose a pseudofunctor $ I: \hhh _ \YYY\to \hhh $ and pseudonatural equivalences $\YYY I\simeq \Id $ and $I\YYY\simeq \Id $.

Therefore if $\hhh $ has all the weighted bilimits of a pseudofunctor $\AAA : \aaa \to\hhh $, 
there is a pseudofunctor 
$\bilim (-, \AAA ):\left[\aaa , \CAT\right] _ {PS}^{ \op }\to \hhh $ which is unique up 
to pseudonatural equivalence and which gives the bilimits of $\AAA $~\cite{Street3, FLN}. 
More precisely, in this case, 
we are actually assuming that the pseudofunctor 
$F: \left[ \aaa , \CAT\right] _ {PS}^{\op } \to \left[ \hhh ^{\op }, \CAT \right] _ {PS}$,
in which
$$
F(\WWW): \bbb ^{\op }\to \CAT : \qquad  X \mapsto \left[\aaa , \CAT \right] _ {PS}(\WWW, \hhh (X , \AAA - ))
$$
is such that $F(\WWW )$ has a bicategorical representation for every weight $\WWW : \aaa\to\CAT $.
Therefore $F$ can be lifted to a pseudofunctor $ F: \left[ \aaa , \CAT\right] _ {PS}^{\op } \to \hhh _ \YYY$. Hence we can take $\bilim ( -,\AAA ) := IF $.

\begin{theo}\label{pointwise}
Assume that $\AAA: \aaa\to\hhh, \t : \aaa\to \bbb $ are pseudofunctors. If $\hhh $ has the weighted bilimit 
$\bilim ( \bbb (b, \t -), \AAA ) $ for every object $b$ of $\bbb $, then $\Ps\Ran _ \t \AAA $ exists. 
Furthermore, there is an equivalence
$$\Ps\Ran _ \t \AAA (b)\simeq \bilim ( \bbb (b, \t -), \AAA )
$$
pseudonatural in $b$.
\end{theo}

\begin{proof}
In this proof, we denote by $\RAN _ \t \AAA $ the pseudofunctor defined by
$$\RAN _ \t \AAA (b) := \bilim (\bbb (b, \t -), \AAA ) .$$
By the propositions presented in this section, we have the following pseudonatural equivalences (in $S$):

\begin{eqnarray*}
\left[\bbb , \hhh \right] _ {PS}(S, \RAN _ \t \AAA ) 
&\simeq & \int _ {b\in\bbb } \hhh (S(b), \RAN _ \t \AAA (b))\\
&\simeq & \int _{b\in\bbb } \hhh (S(b), \bilim (\bbb (b, \t -), \AAA ))  \\
& \simeq & \int _ {b\in\bbb } \left[ \aaa , \CAT \right] _ {PS}(\bbb (b, \t -), \hhh (S(b), \AAA -))\\
& \simeq & \int _ {b\in\bbb } \int _ {a\in\aaa } \CAT (\bbb (b, \t (a)), \hhh (S(b), \AAA (a)))\\
& \simeq & \int _ {a\in\aaa } \int _ {b\in\bbb } \CAT (\bbb (b, \t (a)), \hhh (S(b), \AAA (a)))\\
&\simeq & \int _ {a\in\AAA } \hhh (S\circ\t (a), \AAA (a))\\
&\simeq & \left[ \AAA , \hhh\right]_ {PS} (S\circ\t , \AAA ). 
\end{eqnarray*}
More precisely, the first, fourth, sixth and seventh pseudonatural equivalences come from the fundamental equivalence for pseudoends (Proposition \ref{Q}), while the second and third are, respectively, the definitions of $\RAN _ \t \AAA$ and the definition of weighted bilimit. The remaining pseudonatural equivalence follows from Theorem \ref{Fubini}. 
 These pseudonatural equivalences show that $\Ps\Ran _ \t \AAA $ exists and $\Ps\Ran _ \t \AAA \simeq \RAN _ \t \AAA $.
\end{proof}

\begin{cor}\label{coroCATpoint}
Assume that $\AAA: \aaa\to\CAT, \t : \aaa\to \bbb $ are pseudofunctors. Then $\Ps\Ran _ \t \AAA$ exists and 
$$\Ps\Ran _ \t \AAA (b) \simeq \left[\aaa , \CAT \right] _ {PS} (\bbb (b, \t -) , \AAA   ).
$$
\end{cor}
\begin{proof}
This result follows from Theorem \ref{pointwise} and from the fact that $$\bilim (\bbb (b, \t -) , \AAA   ) \simeq \left[\aaa , \CAT \right] _ {PS} (\bbb (b, \t -) , \AAA   ).$$
\end{proof}

\begin{rem}\label{leftpseudo}
It is clear that Theorem \ref{pointwise} has a dual. That is to say, we have an equivalence 
$\Ps\Lan _ \t \AAA (b) \simeq \bicolim (\bbb (\t -,b), \AAA) $
pseudonatural in $b$ whenever $\bicolim (\bbb (\t -,b), \AAA )$ exists for each $b$ in $\bbb $.
\end{rem}

\begin{rem}\label{conicalpseudo}
By Remark \ref{conical} and Theorem \ref{pointwise}, if $\AAA : \aaa\to \hhh $ is a pseudofunctor, if it exists, 
the conical bilimit of $\AAA $ is equivalent to 
$\Ps\Ran _ \t (\AAA ) (\mathsf{a})\simeq \bilim ( \dot{\aaa }(\mathsf{a}, \t - ), \AAA  )$ 
in which 
$\t :\aaa\to\dot{\aaa } $ is the $\mathsf{a}$-inclusion such that 
$\mathsf{a}$ is the initial object added to $\aaa $.   
\end{rem}

\begin{defi}[Preservation of pseudo-Kan extensions]
Let $U : \hhh\to\hhh ' $, $\AAA : \aaa\to \hhh $ and $\t : \aaa \to \bbb $ be pseudofunctors. Assume that $\Ps\Ran _\t \AAA $ 
exists and has the universal arrow $\varepsilon _ {{}_{\AAA }}$. We say that $U$ \textit{preserves} this pseudo-Kan extension if 
$$\Ps\Ran _ \t (U\circ\AAA ) \simeq U\circ\Ps\Ran _ \t \AAA\mbox{ and }
U\varepsilon _ {{}_{\AAA }} : (U\circ\Ps\Ran _ \t \AAA )\circ\t \longrightarrow U\circ \AAA $$
is the universal arrow of $\Ps\Ran _ \t (U\circ\AAA )$.
\end{defi}

\begin{lem}
If $U$ has a left biadjoint, then it preserves all the existing right pseudo-Kan extensions.
\end{lem}

\begin{defi}[Pointwise pseudo-Kan extension]\label{pointwise3}
Let $\t : \aaa \to \bbb , \AAA : \aaa\to\hhh $ be pseudofunctors such that $\Ps\Ran _ \t \AAA $ exists. We say that $\Ps\Ran _ \t \AAA $
is a \textit{pointwise} (right) pseudo-Kan extension of $\AAA $ along $\t $ if $\Ps\Ran _ \t \AAA $ is preserved by all representable pseudofunctors. 
\end{defi}

\begin{rem}
The definition of pointwise pseudo-Kan extension is motivated by the usual definition of pointwise Kan extension in the strict/enriched case. 
See page 240 of \cite{MACLANE} or page 52 and 54 of \cite{Dubuc2}.
\end{rem}

Analogously to the enriched case (see \cite{Dubuc2}), pointwise pseudo-Kan extensions are pointwise constructed  as in Theorem
\ref{pointwise} by weighted bilimits. Corollary \ref{pointwise2}
makes this statement precise.

\begin{cor}\label{pointwise2}
Assume that $\AAA: \aaa\to\hhh, \t : \aaa\to \bbb $ are pseudofunctors. The pseudofunctor
$\AAA $ has a pointwise pseudo-Kan extension $\Ps\Ran _ \t \AAA $ if and only if $\hhh $ has the weighted bilimit 
$\bilim ( \bbb (b, \t -), \AAA ) $ for every object $b$ of $\bbb $. In this case, there is an equivalence
$$\Ps\Ran _ \t \AAA (b)\simeq \bilim ( \bbb (b, \t -), \AAA ) 
$$
pseudonatural in $b$.
\end{cor}
\begin{proof}
Firstly, assume that $\Ps\Ran _ \t \AAA $ is pointwise. Since it is pointwise, we have an equivalence
$$\hhh (X, \Ps\Ran _ \t \AAA (b) ) \simeq   \Ps\Ran _ \t \hhh (X, \AAA - ) (b) $$
pseudonatural in $b$, while, by the bicategorical Yoneda lemma and the universal property of $\Ps\Ran _ \t \hhh (X, \AAA - )$,
we have an equivalence
\begin{eqnarray*}
\Ps\Ran _ \t \hhh (X, \AAA - ) (b)&\simeq & \left[\bbb , \CAT \right] _ {PS}(\bbb (b, -) ,  \Ps\Ran _ \t \hhh (X, \AAA - ) )\\
&\simeq &  \left[\aaa , \CAT \right] _ {PS}(\bbb (b, \t -) ,  \hhh (X, \AAA - ) )
\end{eqnarray*}
pseudonatural in $b$. This completes the proof that $\Ps\Ran _ \t \AAA (b)$ is the left bicategorical reflection of $\bbb (b, \t -)$ along 
$$\hhh ^\op  \to \left[\aaa , \CAT\right]_{PS} :  \qquad X\mapsto   \hhh (X, \AAA - ) ,$$ 
which means that $\bilim (\bbb (b, \t -), \AAA ) $ exists and $\Ps\Ran _ \t \AAA (b)\simeq \bilim (\bbb (b, \t -), \AAA ) $
pseudonatural in $b $.

Reciprocally, assume that $\bilim ( \bbb (b, \t -), \AAA ) $ exists for every $b $ in $\bbb $. In this case, by Theorem \ref{pointwise}
$\Ps\Ran _ \t \AAA$ exists and
there is an equivalence
$$\bilim ( \bbb (b, \t -), \AAA )\simeq \Ps\Ran _ \t \AAA (b)
$$
pseudonatural in $b $. 
Given any $X $ in $\hhh $, we have equivalences 
$$\hhh (X, \Ps\Ran _ \t \AAA (b))\simeq \hhh (X, \bilim ( \bbb (b, \t -), \AAA ))\simeq  \left[\aaa , \CAT \right] _ {PS} (\bbb (b, \t -) , \hhh (X, \AAA - )   ) $$
pseudonatural in $b$. Since by Corollary \ref{coroCATpoint} we have an equivalence 
$$\left[\aaa , \CAT \right] _ {PS} (\bbb (b, \t -) , \hhh (X, \AAA - )   ) \simeq \Ps\Ran _ \t \hhh (X, \AAA- ) (b) $$
pseudonatural in $b $, the proof is complete.
\end{proof}

In the enriched case~\cite{Dubuc2, Kelly}, in the presence of weighted limits, pointwise Kan extensions are constructed pointwise by equalizers of
products of cotensor products, since every weighted limit can be seen as such. In the bicategorical case, weighted bilimits are descent objects
of bicategorical products and cotensor products whenever they exist and, therefore, the result above shows that, in suitable cases, pointwise pseudo-Kan extensions can be
constructed pointwise by them (see \cite{FLN}).

In this paper, for simplicity, 
we always assume that \textit{$\hhh $ is a bicategorically complete $2$-category}, 
or at least $\hhh $ has enough bilimits to construct the considered (right) pseudo-Kan extensions as 
pointwise pseudo-Kan extensions.

\begin{rem}
The pointwise pseudo-Kan extension was studied originally in \cite{FLN} using 
the Biadjoint Triangle Theorem proved therein. The construction presented above is similar to the 
usual approach of the enriched case~\cite{Kelly, Dubuc2}, 
while the argument via biadjoint triangles  of \cite{FLN} is not.
\end{rem}

\subsection{The pseudomonad $\left\langle \t \right\rangle : = \Ps\Ran _\t (-\circ\t )$}\label{Psidem}
By the (bicategorical) Yoneda lemma, whenever $\t $ is a local equivalence, if the pseudo-Kan extension 
$\Ps\Ran _\t\AAA $ exists, it is actually a 
pseudoextension. 
More precisely:

\begin{theo}\label{Psisidem}
If $\t : \aaa\to\dot{\aaa } $ is a local equivalence
and there is a biadjunction $$\left[ \t ,  \hhh\right] _{PS}\dashv \Ps\Ran _\t ,$$ 
its counit is a pseudonatural equivalence.  Thereby
$\Ps\Ran _ \t: \left[ \aaa , \hhh \right] _ {PS}\to \left[\dot{\aaa }, \hhh\right] _ {PS} $
is a local equivalence and, hence, pseudomonadic and
the induced pseudomonad is idempotent.
\end{theo}
\begin{proof}
It follows from the bicategorical Yoneda lemma. By Lemma \ref{YonedaBILIMIT}, 
if $X$ is an object of $\aaa $, 
$\bilim ( \dot{\aaa } ( \t (X), \t -), \AAA )\simeq \bilim ( \aaa (X , -), \AAA )\simeq \AAA (X) $.
\end{proof}

In the context of the result above, we denote the idempotent pseudomonad induced by the biadjunction $\left[ \t ,  \hhh\right] _{PS}\dashv \Ps\Ran _\t $ by $\Ps\Ran _ \t (-\circ\t )$ or, for short, $\left\langle \t \right\rangle $. 

Our interest is to study the objects of $\left[ \dot{\aaa } , \hhh\right] _ {PS} $ that can be endowed 
with $\left\langle \t \right\rangle $-pseudoalgebra structure. Assuming that $\t $ is a local equivalence, this means that our interest is to study the image of the forgetful Eilenberg-Moore $2$-functor 
$$\left\langle \t \right\rangle \textrm{-}\Alg\to\left[ \dot{\aaa } , \hhh\right] _ {PS}.$$

\begin{defi}[Effective Diagrams]
Let $\t : \aaa\to\dot{\aaa } , \AAA : \dot{\aaa }\to\hhh $ be pseudofunctors. 
$\AAA : \dot{\aaa }\to\hhh $ is  
of \textit{effective $\t $-descent} if $\AAA $ can be endowed with a 
$\left\langle \t \right\rangle  $-pseudoalgebra structure. 
\end{defi}

We now can apply the results of Section \ref{Formal} on idempotent pseudomonads. 
Firstly, by Theorem \ref{algebra}, we can easily study the $\left\langle \t \right\rangle  $-pseudoalgebra 
structures on diagrams, 
using the unit of the pseudomonad. 

\begin{theo}\label{pseudoalgebrakan}
Let $\t : \aaa\to\dot{\aaa } $ be a local equivalence and $\AAA :\dot{\aaa }\to\hhh $ a pseudofunctor. 
The following conditions are equivalent:
\begin{itemize}
\renewcommand\labelitemi{--}
  \item $\AAA $ is of effective $\t $-descent;
	\item The \textsl{component of the unit on $\AAA $/comparison} $\eta _ {{}_\AAA } : \AAA\to \Ps\Ran _ \t  (\AAA\circ\t ) $ is a pseudonatural equivalence;
	\item The \textsl{comparison} $\eta _ {{}_\AAA } : \AAA\to \Ps\Ran _ \t (\AAA\circ\t ) $ is a pseudonatural pseudosection.
\end{itemize}
\end{theo}

The component of the unit $\eta _ {{}_\AAA } : \AAA\longrightarrow \Ps\Ran _ \t  (\AAA\circ\t ) $ is a 
pseudonatural equivalence if and only if all components of $\eta _ {{}_\AAA } $ are equivalences. 
By Theorem \ref{Psisidem}, assuming that $\t : \aaa\to\dot{\aaa } $ is an $\mathsf{a}$-inclusion, 
$\eta _ {{}_\AAA }^{b} $ is an equivalence for all ${b}$ in $\aaa $. Thereby we get: 

\begin{lem}\label{pseudoalgebrakanres}
Let $\t : \aaa\to\dot{\aaa } $ be an $\mathsf{a}$-inclusion. A pseudofunctor $\AAA : \dot{\aaa } \to\hhh $ is of effective $\t $-descent if and only if
$\eta _ {{}_\AAA } ^\mathsf{a} : \AAA (\mathsf{a} )\to \Ps\Ran _ \t (\AAA\circ\t ) (\mathsf{a}) $
is an equivalence.
\end{lem}

\subsection{Commutativity}\label{commu}

Let $\t : \aaa\to\dot{\aaa }$ and $\h : \bbb\to\dot{\bbb } $ be, respectively, an 
$\mathsf{a}$-inclusion and a $\mathsf{b}$-inclusion (see Definition \ref{Adefinition} and Remark \ref{Bdefinition}).
Unless we explicitly state otherwise, 
henceforth we always consider right pseudo-Kan extensions along such type of inclusions. 

In general, we have that (see \cite{Street5}):
$\left[\dot{\aaa }\times\dot{\bbb }, \hhh\right] _ {PS} 
\approx\left[\dot{\aaa }, \left[ \dot{\bbb }, \hhh\right] _ {PS}\right] _ {PS}
\cong \left[\dot{\bbb }, \left[ \dot{\aaa}, \hhh\right] _ {PS}\right] _ {PS}$.
Thereby every pseudofunctor $\AAA : \dot{\aaa }\times\dot{\bbb }\to\hhh $ can be seen 
(up to pseudonatural equivalence) as a pseudofunctor 
$\AAA : \dot{\aaa } \to \left[ \dot{\bbb }, \hhh\right] _ {PS} $. Also, 
$\AAA : \dot{\aaa } \to \left[ \dot{\bbb }, \hhh\right] _ {PS} $ can be seen as a pseudofunctor 
$\AAA : \dot{\bbb } \to \left[ \dot{\aaa }, \hhh\right] _ {PS} $. 	

Applying our formal approach of Section \ref{Formal}
to our context of pseudo-Kan extensions, we get theorems on commutativity as we show below.

\begin{theo}\label{okbeck}
If $\AAA : \dot{\aaa } \to \hhh $ is an effective $\t $-descent pseudofunctor and $\TTT $ is an idempotent 
pseudomonad on $\hhh $ such that $\AAA\circ\t $ can be factorized through 
$\mathsf{Ps}\textrm{-}\TTT\textrm{-}\Alg\to\hhh $, then $\AAA (\mathsf{a}) $ can be endowed with a $\TTT $-pseudoalgebra 
structure.
\end{theo}
\begin{proof}
Let $\LLL\dashv\UUU $ be the biadjunction induced by $\TTT $ and $\widehat{\hhh }:=\mathsf{Ps}\textrm{-}\TTT\textrm{-}\Alg $ (see Definition \ref{pseudoalgebraidem} and Theorem \ref{FACTEILENBERG}). 
Observe that the pseudonatural equivalence
$$\xymatrix{  \left[ \aaa ,\widehat{\hhh }\right] _ {PS}\ar@{}[ddrr]|-{ = } 
&&\left[ \dot{\aaa } ,\widehat{\hhh }\right] _ {PS}\ar[ll]|-{\left[ \t ,\widehat{\hhh }\right] _ {PS} }\\
&&\\
\left[ \aaa ,\hhh \right] _ {PS}\ar[uu]|-{\left[ \aaa ,\LLL \right] _ {PS} } &&\left[ \dot{\aaa } ,\hhh \right] _ {PS}\ar[ll]|-{\left[ \t ,\hhh \right] _ {PS}}\ar[uu]|-{\left[ \dot{\aaa } ,\LLL \right] _ {PS}}
}$$
satisfies the hypotheses of Corollary~\ref{square}. 

If $\AAA :\dot{\aaa } \to \hhh $ is an effective $\t $-descent pseudofunctor such that all the objects 
of the image of $\AAA\circ\t $ have $\TTT $-pseudoalgebra structure, 
it means that $\AAA $ satisfies the hypotheses of Corollary~\ref{square}. That is to say,  
$\AAA $ is a $\left\langle \t\right\rangle$-pseudoalgebra that can be endowed with a 
$\left[ \aaa , \TTT\right]_ {PS}$-pseudoalgebra structure. 
Thereby, by Corollary \ref{square}, $\AAA $ can be endowed with a $\left[ \dot{\aaa }, \TTT\right] _ {PS} $-pseudoalgebra structure.  
\end{proof}

\begin{cor}\label{comutatividade}
Let $\AAA : \dot{\aaa }\to \left[ \dot{\bbb } , \hhh\right] _ {PS}$ be an 
effective $\t $-descent pseudofunctor such that the diagrams in the image of 
$\AAA\circ\t $ are of effective $\h $-descent, then $\AAA (\mathsf{a} ) $ is of effective $\h $-descent 
as well.
\end{cor}
\begin{proof}
Since $\left\langle \h\right\rangle$ is idempotent, this result is Theorem \ref{okbeck} applied to the case $\TTT = \left\langle \h\right\rangle $.
\end{proof}

\begin{cor}\label{comutatividadeequi}
Assume that the pseudofunctors
$
\widehat{\AAA }: \dot{\aaa }\to \left[ \dot{\bbb } , \hhh\right] _ {PS}$ and
$
\bar{\AAA }: \dot{\bbb }\to \left[ \dot{\aaa } , \hhh\right] _ {PS}            
$
are mates such that the diagrams in the image of $\widehat{\AAA }\circ\t$ and $ \bar{\AAA }\circ\h $  
are respectively of effective $\h $- and $\t$-descent. 
We have that $\widehat{\AAA }(\mathsf{a} ) $ is of effective $\h $-descent if and only if $\bar{\AAA } (\mathsf{b} ) $ is of effective $\t $-descent.
\end{cor}

\subsection{Almost descent pseudofunctors}\label{tecalmost}
Recall that a $1$-cell in a $2$-category $\hhh $ is called faithful/fully faithful if its 
images by the (covariant) representable $2$-functors are faithful/fully faithful.

\begin{defi}\label{almost descent}
Let $\t : \aaa\to\dot{\aaa } $ be an $\mathsf{a}$-inclusion. A pseudofunctor 
$\AAA : \dot{\aaa } \to\hhh $ is of \textit{almost  $\t $-descent}/\textit{$\t $-descent} if 
$\eta _ {{}_\AAA } ^\mathsf{a} : \AAA (\mathsf{a} )\to \Ps\Ran _ \t (\AAA\circ\t ) (\mathsf{a}) $
is faithful/fully faithful. 
\end{defi}

Consider the class $\mathfrak{F} _ {{}_{\left[ \dot{\aaa } , \hhh\right] _ {PS}}}$ of 
pseudonatural transformations in $\left[ \dot{\aaa } , \hhh\right] _ {PS}$ whose components are 
faithful. This class satisfies the properties described in \ref{tecnicoalmost}. Also, 
a pseudofunctor $\AAA : \dot{\aaa }\to\hhh $
is of almost descent if and only if $\AAA $ is a 
$(\mathfrak{F} _ {{}_{\left[ \dot{\aaa } , \hhh\right] _ {PS}}}, \left\langle \t \right\rangle )$-object.

Analogously, if we take the class $\mathfrak{F} _ {{}_{\left[ \dot{\aaa } , \hhh\right] _ {PS}}}'$ of 
objectwise fully faithful pseudonatural transformations, $\AAA : \dot{\aaa }\to\hhh $ is of descent if and 
only if $\AAA $ is a 
$(\mathfrak{F} _ {{}_{\left[ \dot{\aaa } , \hhh\right] _ {PS}}}', \left\langle \t \right\rangle )$-object.

Since in our context of right pseudo-Kan extensions along local equivalences the hypotheses of 
Theorem \ref{descenTsquare} hold, we get the corollaries below. Again, we are considering full inclusions 
$\t : \aaa\to\dot{\aaa }$, $\h : \bbb\to\dot{\bbb } $ as in \ref{commu}.

\begin{cor}\label{comutatividadealmostdescent}
Let $\AAA : \dot{\aaa }\to \left[ \dot{\bbb } , \hhh\right] _ {PS}$ be an almost $\t$-descent pseudofunctor such that the pseudofunctors in the image of $\AAA\circ\t $ are of almost $\h$-descent. In this case, $\AAA (\mathsf{a} ) $ is also of almost $\h $-descent.

Similarly, if $\AAA $ is of $\t$-descent and the pseudofunctors of the image of $\AAA\circ\t $ are of $\h$-descent, then $\AAA (\mathsf{a} ) $ is of $\h $-descent as well. 
\end{cor}
\begin{proof}
In order to show that both cases fit in the technical conditions of the hypothesis of \ref{descenTsquare}, we only need to 
observe that any poitwise (right) pseudo-Kan extension pseudofunctor preserves pointwise faithful and pointwise fully faithful pseudonatural
transformations. In order to verify that, it is enough to see that, given any small $2$-category $\ccc $, we have that
$$\left[ \ccc , \CAT \right]_{PS} (\BBB , \alpha ), $$
is pointwise (fully) faithful whenever $\alpha $ is so.
\end{proof}

\begin{cor}\label{comutatividadealmostcara}
Assume that the mates
$
\widehat{\AAA }: \dot{\aaa }\to \left[ \dot{\bbb } , \hhh\right] _ {PS}$ and
$
\bar{\AAA }: \dot{\bbb }\to \left[ \dot{\aaa } , \hhh\right] _ {PS}            
$
are such that the diagrams in the image of $\widehat{\AAA }\circ\t $ and $\bar{\AAA }\circ\h $ are respectively of almost $\h$- and $\t $-descent. In this case, 
\begin{center}
$\widehat{\AAA }(\mathsf{a} ) $ is of almost $\h$-descent if and only if $\bar{\AAA } (\mathsf{b} ) $ is of almost  $\t $-descent.
\end{center}
If, furthermore, the pseudofunctors in the image of $\widehat{\AAA }\circ\t $ and $\bar{\AAA }\circ\h $ are respectively of $\h$- and $\t $-descent, then: 
\begin{center}
$\widehat{\AAA }(\mathsf{a} ) $ is of $\h$-descent if and only if $\bar{\AAA } (\mathsf{b} ) $ is of $\t $-descent.
\end{center}
\end{cor}

\section{Descent Objects}\label{descent}
In this section, we give a description of the descent category, as defined in classical descent theory,
via pseudo-Kan extensions. The results of the first part of this section is 
hence important to fit the context of \cite{Facets, Facets2}
within our framework.

Let $\j : \Delta\to \dot{\Delta } $ be the full inclusion of the category of finite nonempty ordinals 
into the category of finite ordinals and order preserving functions. 
Recall that $\dot{\Delta  } $ is generated by its degeneracy and face maps. That is to say, $\dot{\Delta } $ is generated by the diagram
$$
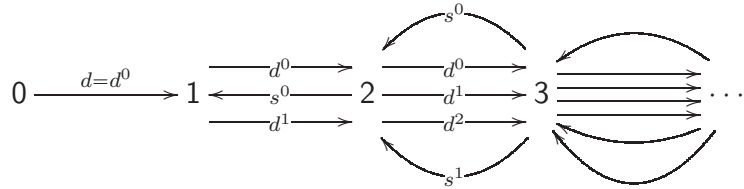
\xymatrix{  \mathsf{0} \ar[rr]^-{d=d^0} && \mathsf{1}\ar@<2ex>[rr]|-{d^0}\ar@<-2ex>[rr]|-{d^1} && \mathsf{2}\ar[ll]|-{s^0}
\ar@<2 ex>[rr]|-{d ^0}\ar[rr]|-{d ^1}\ar@<-2ex>[rr]|-{d ^2} && \mathsf{3}\ar@/_4ex/@<-2 ex>[ll]|-{s^0}\ar@/^4ex/@<2ex>[ll]|-{s^1}\ar@<1.5ex>[rr]\ar@<0.5ex>[rr]\ar@<-0.5ex>[rr]\ar@<-1.5ex>[rr]&& \cdots\ar@/^2ex/@<1.5ex>[ll]\ar@/^5ex/@<1.5ex>[ll]\ar@/_3ex/@<-1.5 ex>[ll] } $$
with the following relations: 
\begin{equation*}
\begin{aligned}
d ^k d ^i&=& d^{i}d^{k-1}, &\hspace{1mm}\mbox{if  }\hspace{1mm} i<k ;& \\
s^ks^i &=& s^is^{k+1}, &\hspace{1mm}\mbox{if  }\hspace{1mm} i\leq k ;&  \\
s^k d^i &=& d^i s^{k-1}, &\hspace{1mm}\mbox{if  }\hspace{1mm} i< k ;&
\end{aligned}
\qquad\qquad\qquad
\begin{aligned}
d^0 d & = d^1 d ;\\
s^k d^i &= \id ,\hspace{1mm}\mbox{if  }\hspace{1mm} i=k\mbox{ and } i=k+1 ; \\
s^k d^i &= d^{i-1}s^k, \hspace{1mm}\mbox{if  }\hspace{1mm} i>k+1 . 
\end{aligned}
\end{equation*}

\begin{rem}\label{chosenmonoid}
The category $\dot{\Delta } $ has an obvious strict monoidal structure $(+,\mathsf{0} )$ that turns
$(\dot{\Delta }, +, \mathsf{0}, \mathsf{1}) $ into the initial object of the category of monoidal categories with
a chosen monoid~\cite{Lawvere}. 
\end{rem}

\begin{rem}\label{notationofmorphisms}
There is a full inclusion $\dot{\Delta }\to\CAT $ 
such that the image of each $\mathsf{n} $ is the corresponding ordinal. 
This is the reason why we may consider that $\dot{\Delta } $ 
is precisely the full subcategory of $\CAT $ of the 
finite ordinals (considered as partially ordered sets). 
In this context, the object $\mathsf{n}$ is often confused with its image which is the category
$$\mathsf{0}\to \mathsf{1} \to \mathsf{2}\to \cdots \mathsf{n-1}. $$  
It is important to keep in mind that $\dot{\Delta } $ 
is a category, but we often consider it inside the tricategory  $2\textrm{-}\CAT$. 
More precisely, by abuse of language, $\dot{\Delta } $ and $\Delta $ denote respectively the images of 
the categories $\dot{\Delta } $ and 
$\Delta $  by the inclusion $\CAT\to 2\textrm{-}\CAT$. Hence 
$\dot{\Delta } $ is locally discrete and is not a full 
sub-$2$-category of $\CAT $. 
In fact, it is clear that $\Delta (\mathsf{1}, \mathsf{n})$ is the image of $\mathsf{n} $ by the comonad
induced by the right adjoint forgetful functor between the category 
of small categories and the category of sets, 
the counit of which is denoted by $\varepsilon^\textrm{d}$. 
\end{rem}

\begin{defi}\label{DescentDefinition}
A pseudofunctor $\AAA : \Delta \to\hhh $ is called a \textit{pseudocosimplicial object} of $\hhh $. 
The \textit{descent object} of such a pseudocosimplicial object $\AAA $ is $\Ps\Ran _ \j \AAA (\mathsf{0}) $.  
\end{defi}

\begin{rem}\label{conical1}
Since $\mathsf{0}$ is the initial object of $\dot{\Delta } $, the weight $\dot{\Delta }(\mathsf{0}, \j - )$ is terminal. 
By Remark \ref{conical} and Theorem \ref{pointwise}, 
we get that the descent object of $\AAA : \Delta \to\hhh $ is by Definition \ref{DescentDefinition} the conical bilimit of $\AAA $. 
\end{rem}

Theorem \ref{indeeddescent} shows that our definition of descent object agrees with Definition~\ref{Streetdescent}, 
which is the usual definition of the descent object~\cite{Street5, Facets, Galois}.

\begin{defi}
The category $\dot{\Delta } _ {{}_{3}}$ is generated by the diagram: 
$$\xymatrix{  \mathsf{0} \ar[rr]^-d && \mathsf{1}\ar@<1.7 ex>[rrr]^-{d^0}\ar@<-1.7ex>[rrr]_-{d^1} &&& \mathsf{2}\ar[lll]|-{s^0}
\ar@<1.7 ex>[rrr]^{\partial ^0}\ar[rrr]|-{\partial ^1}\ar@<-1.7ex>[rrr]_{\partial ^2} &&& \mathsf{3} }$$
such that:
$$d^1d = d^0d ;\qquad\qquad\partial ^k d ^i= \partial ^{i}d^{k-1}\hspace{1mm}\mbox{if  }\hspace{1mm} i<k; \qquad\qquad s^0 d^0 = s^0d ^1 = \id .$$ 
We denote by
$\j _ {{}_3} : \Delta _{{}_3}\to \dot{\Delta } _ {{}_3} $
the full inclusion of the subcategory $\Delta _ {{}_{3}}$ in which $\obj (\Delta _ {{}_{3}}) = \left\{ \mathsf{1}, \mathsf{2}, \mathsf{3}\right\} $.
Still, there are obvious inclusions:
$\dot{\t } _ {{}_{3}}:\dot{\Delta } _ {{}_3} \to \dot{\Delta }$ and $\t  _ {{}_{3}}:\Delta  _ {{}_3} \to \Delta .$
Again, $\dot{\Delta } _ {{}_{3}}$ herein usually denotes the respective locally discrete $2$-category.

\end{defi}

\begin{defi}\label{Streetdescent}
We denote by $\WW : \Delta _ {{}_{3}}\to \CAT $ the weight below (defined in \cite{Street5}), in which $\nabla \mathsf{n} $ denotes the localization of the category/finite ordinal $\mathsf{n} $ w.r.t all the morphisms. 
\[\xymatrix{  \nabla \mathsf{1}\ar@<2ex>[rr]\ar@<-2ex>[rr] && \nabla\mathsf{2}\ar[ll]
\ar@<2 ex>[rr]\ar[rr]\ar@<-2ex>[rr] && \nabla\mathsf{3} }\] 
Following \cite{Street5}, if $\AAA : \Delta\to\hhh $ is a pseudofunctor, we define 
$$\Desc (\AAA ):= \bilim ( \WW , \AAA\circ \t _ {{}_3} )  .$$
\end{defi}

\begin{rem}\label{conical2}
The weight $\WW$ is pseudonaturally equivalent to the terminal weight. 
Therefore, $\Desc (\AAA )$ is by definition (equivalent to) the conical bilimit of $\AAA\circ \t _ {{}_3} $. 
\end{rem}

In order to prove Theorem \ref{essentialleft}, we need:

\begin{prop}\label{technicalll}
Let $Y$ be any category and $\underline{Y}: \Delta _ {{}_{3}}^\op\to \CAT $ the constant $2$-functor $\mathsf{n}\mapsto Y $. Given any (strict) 
$2$-functor  $\BBB : \Delta _ {{}_{3}} ^\op \to\hhh $ and a pseudonatural transformation $\alpha :  \BBB\longrightarrow\underline{Y}$, 
the following equations hold: 
\begin{equation*}\tag{associativity codescent equation}\label{AssociatividadeFunctorial}
\begin{split}
\xymatrix{
&
Y
\ar@{<->}[rr]^{\Id_{{}_Y} } 
&
&
Y 
&
&
&
Y
&
\\
\BBB (\mathsf{1})
\ar@{}[r]|-{\xRightarrow{\alpha _{{}_{d ^1} }^{-1} } }
\ar[ru]^{\alpha _{{}_\mathsf{1} } }
&
&
\BBB (\mathsf{1})
\ar@{}[l]|-{\xRightarrow{\alpha _{{}_{d ^0} } } }
\ar@{}[r]|-{\xRightarrow{\alpha _{{}_{d ^1} }^{-1} } }
\ar[lu]|-{\alpha _{{}_\mathsf{1} } }
\ar[ru]|-{\alpha _{{}_\mathsf{1} } }
\ar@{}[u]|-{=}
\ar@{}[d]|-{=}
&
&
\BBB (\mathsf{1})
\ar@{}[r]|-{=}
\ar@{}[l]|-{\xRightarrow{\alpha _{{}_{d ^0} } } }
\ar[lu]_-{\alpha _{{}_\mathsf{1} } }
&
\BBB (\mathsf{1})
\ar@{}[r]|-{\xRightarrow{\alpha _{{}_{d ^1} }^{-1} } }
\ar[ru]^{\alpha _{{}_\mathsf{1} } }
&
&
\BBB (\mathsf{1})
\ar@{}[l]|-{\xRightarrow{\alpha _{{}_{d ^0} } } }
\ar[lu]_-{\alpha _{{}_\mathsf{1} } }
\\
&
\BBB (\mathsf{2})
\ar[lu]^{\BBB (d ^1 ) }
\ar[ru]|-{\BBB (d^0 ) }
\ar[uu]|-{\alpha _{{}_\mathsf{2} } }
&
\BBB (\mathsf{3} )
\ar[l]^-{\BBB(\partial ^2) }
\ar[r]_-{\BBB(\partial ^0) }
&
\BBB (\mathsf{2})
\ar[lu]|-{\BBB (d ^1 ) }
\ar[ru]_-{\BBB (d^0 ) }
\ar[uu]|-{\alpha _{{}_\mathsf{2} } }
&
&
&
\BBB (\mathsf{2})
\ar[lu]^-{\BBB (d ^1 ) }
\ar[ru]|-{\BBB (d^0 ) }
\ar[uu]|-{\alpha _{{}_\mathsf{2} } }
&
\BBB (\mathsf{3})
\ar[l]^{\BBB (\partial ^ 1 ) }
}
\end{split}
\end{equation*}
\begin{equation*}\tag{identity of codescent}\label{IdentidadeFunctorial}
\begin{split}
\xymatrix{
&
Y
&
&
&
Y
\\
\BBB (\mathsf{1})
\ar[ru]^{\alpha _{{}_\mathsf{1} }}
\ar@{}[r]|-{\xRightarrow{\alpha _{{}_{d^1} }^{-1} } }
&
&
\BBB (\mathsf{1} )
\ar@{}[l]|-{\xRightarrow{\alpha _{{}_{d^0} } } }
\ar[lu]_{\alpha _{{}_\mathsf{1} }}
&
&
\\
&
\BBB (\mathsf{2})
\ar[lu]^{\BBB (d ^1 )}
\ar[ru]_{\BBB (d ^0 ) }
\ar[uu]|-{\alpha _{{}_{\mathsf{2} } } }
&
\\
&
\BBB (\mathsf{1} )
\ar[u]|{\BBB (s ^0 ) }
&
&
&
\BBB (\mathsf{1})
\ar@{}[uuull]^-{=}
\ar@{}[uuu]|-{=}
\ar@/_3ex/[uuu]_-{\alpha _{{}_{\mathsf{1} } } }
\ar@/^3ex/[uuu]^-{\alpha _{{}_{\mathsf{1} } } }
}
\end{split}
\end{equation*}

\end{prop}
\begin{proof}
We start by proving the \ref{IdentidadeFunctorial}. Indeed, by Definition 
\ref{pseudonaturaltransformation} of pseudonatural transformation (see \cite{FLN}), since $d ^0s ^0= d ^1 s^0 = 
\id _ {{}_{\mathsf{1} }}$, $\BBB $ is a $2$-functor and $\underline{Y}$ is constant equal to $Y$, 
we have that $\alpha _{{}_{d^0 s^0 } }=\Id_{{}_{  \alpha _{{}_{\mathsf{1}  } } } }  = \alpha _{{}_{d^1 s^0 } } $
which implies in particular that 
$$\xymatrix{
\BBB (\mathsf{1})
\ar@{}[rd]|-{\xRightarrow{\alpha _{{}_{s^0 } } } }
\ar[r]^-{\BBB (s ^0)}
\ar[rdd]_-{\alpha _{{}_{\mathsf{1} } }  }
&
\BBB (\mathsf{2})
\ar[dd]|-{\alpha _{{}_{\mathsf{2} } }  }
\ar[r]^{\BBB (d ^0 )}
&
\BBB (\mathsf{1})
\ar@{}[ld]|-{\xRightarrow{\alpha _{{}_{d^0 } } } }
\ar[ldd]^-{\alpha _{{}_{\mathsf{1} } } }
\ar@/^2ex/@{{ }{ }}[dd]|-{=}
&
\BBB (\mathsf{1})
\ar@{}[dd]|-{=}
\ar@/^2ex/[dd]^{\alpha _{{}_{\mathsf{1} } }}
\ar@/_2ex/[dd]_{\alpha _{{}_{\mathsf{1} } } }
&
\BBB (\mathsf{1})
\ar@{}[rd]|-{\xRightarrow{\alpha _{{}_{s^0 } } } }
\ar[r]^-{\BBB (s ^0)}
\ar[rdd]_-{\alpha _{{}_{\mathsf{1} } }  }
\ar@/_2ex/@{{ }{ }}[dd]|-{=}
&
\BBB (\mathsf{2})
\ar[dd]|-{\alpha _{{}_{\mathsf{2} } }  }
\ar[r]^{\BBB (d ^1 )}
&
\BBB (\mathsf{1})
\ar@{}[ld]|-{\xRightarrow{\alpha _{{}_{d^1 } } } }
\ar[ldd]^-{\alpha _{{}_{\mathsf{1} } } }
\\
&
&
&
&
&
&
\\
&
Y
&
&
Y
&
&
Y
&
}$$
and therefore:
$$\xymatrix{
\BBB (\mathsf{1})
\ar@{}[rd]|-{\xRightarrow{\alpha _{{}_{d^1 } }^{-1} } }
\ar[rdd]_-{\alpha _{{}_{\mathsf{1} } }  }
&
\BBB (\mathsf{2})
\ar[l]_-{\BBB (d ^1 )}
\ar[dd]|-{\alpha _{{}_{\mathsf{2} } }  }
&
\BBB (\mathsf{1})
\ar[l]_-{\BBB (s ^0 )}
\ar@{}[ld]|-{\xRightarrow{\alpha _{{}_{s^0 } }^{-1} } }
\ar[ldd]^-{\alpha _{{}_{\mathsf{1} } } }
\ar@{}[rd]|-{\xRightarrow{\alpha _{{}_{s^0 } } } }
\ar[r]^-{\BBB (s ^0)}
\ar[rdd]_-{\alpha _{{}_{\mathsf{1} } }  }
\ar@{{ }{ }}[dd]|-{=}
&
\BBB (\mathsf{2})
\ar[dd]|-{\alpha _{{}_{\mathsf{2} } }  }
\ar[r]^{\BBB (d ^0 )}
&
\BBB (\mathsf{1})
\ar@{}[ld]|-{\xRightarrow{\alpha _{{}_{d^0 } } } }
\ar[ldd]^-{\alpha _{{}_{\mathsf{1} } } }
\ar@/^4ex/@{{ }{ }}[dd]|-{=}
&
&
\BBB (\mathsf{2} )
\ar[dd]|-{\alpha _{{}_{\mathsf{2} } } }
\ar[rd]|-{\BBB (d^0) }
\ar[ld]|-{\BBB (d ^1) }
&
\BBB (\mathsf{1})
\ar[l]_-{\BBB(s ^0) }
\\
&
&
&
&
&
\BBB (\mathsf{1})
\ar@{{ }{ }}[r]|-{\xRightarrow{\alpha _{{}_{d^1 } }^{-1} }}
\ar[rd]|-{\alpha _{{}_{\mathsf{1} } } }
&
&
\BBB (\mathsf{1})
\ar[ld]|-{\alpha _{{}_{\mathsf{1} } } }
\ar@{}[l]|-{\xRightarrow{\alpha _{{}_{d^0 } } } }
&
\\
&
Y\ar@{<->}[rr]_{\Id_{{}_{Y}} }
&
&
Y
&
&
&
Y
&
&
}$$
is equal to the identity on $\alpha _{{}_{\mathsf{1} } } $. This proves that the  \ref{IdentidadeFunctorial} holds.

It remains to prove that the \ref{AssociatividadeFunctorial} holds. Since by the definition of pseudonatural transformation
we have that 
$$ \left(\alpha _{{}_{d^0 } }\ast \Id _ {{}_{\BBB (\partial  ^2 ) }}\right)\cdot \alpha _{{}_{\partial ^2 } }
=
\alpha _{{}_{d^0\partial ^2 }} = \alpha _{{}_{d^1\partial ^0 }} = 
\left(\alpha _{{}_{d^1 } }\ast \Id _ {{}_{\BBB (\partial  ^0 ) }}\right)\cdot \alpha _{{}_{\partial ^0 } },
$$
we conclude that
$$
\xymatrix{
Y\ar@{<->}[rr]^{\Id _ {{}_{Y}} } 
&
&
Y\ar@{{ }{ }}[rdd]|-{=}
&
Y\ar[rr]^{\Id _ {{}_{Y}} }\ar@/^4ex/@{{ }{ }}[dd]|-{\xRightarrow{\alpha _{{}_{\partial^2 }  }^{-1} } }
&
&
Y\ar@/_3ex/@{{ }{ }}[dd]|-{\xRightarrow{\alpha _{{}_{\partial^0 } }} }
\\
&
\BBB (\mathsf{1})
\ar@{}[l]|-{\xRightarrow{\alpha _{{}_{d^0 } } } }
\ar@{}[u]|-{=}
\ar@{}[d]|-{=}
\ar@{}[r]|-{\xRightarrow{\alpha _{{}_{d^1 } }^{-1} } }
\ar[ru]|-{\alpha _ {{}_{\mathsf{1} }} }
\ar[lu]|-{\alpha _ {{}_{\mathsf{1} }} }
&
&
&
&
\\
\BBB (\mathsf{2})
\ar[uu]^{\alpha _ {{}_{\mathsf{2} }} }
\ar[ru]|-{\BBB (d ^0)}
&
\BBB (\mathsf{3})
\ar[l]^{\BBB (\partial ^2) }
\ar[r]_{\BBB (\partial ^0) }
&
\BBB (\mathsf{2})
\ar[uu]_-{\alpha _ {{}_{\mathsf{2} }} }
\ar[lu]|-{\BBB (d ^1)}
&
\BBB (\mathsf{2})
\ar[uu]^{\alpha _ {{}_{\mathsf{2} }} }
&
\BBB(\mathsf{3})
\ar[l]^-{\BBB (\partial ^2) }
\ar@/_2ex/[luu]|-{\alpha _ {{}_{\mathsf{3} }} }
\ar[r]_-{\BBB (\partial ^0)}
\ar@/^2ex/[ruu]|-{\alpha _ {{}_{\mathsf{3} }} }
\ar@{}[uu]|-{=}
&
\BBB(\mathsf{2})\ar[uu]_{\alpha _ {{}_{\mathsf{2} }} } 
&
}
$$
holds. Since $\alpha _{{}_{d^0\partial ^0 }} = \alpha _{{}_{d^1\partial ^0 }}$, 
$\alpha _{{}_{d^1\partial ^2 }} = \alpha _{{}_{d^1\partial ^1 }}$, by the equality above,
the left side of the \ref{AssociatividadeFunctorial} is equal to
$$
\xymatrix@C=1em{
&
Y
\ar[rr]^{\Id _ {{}_{Y}} }
\ar@/^3ex/@{{ }{ }}[dd]|-{\xRightarrow{\alpha _{{}_{\partial^2 }  }^{-1} } }
&
&
Y
\ar@/_3ex/@{{ }{ }}[dd]|-{\xRightarrow{\alpha _{{}_{\partial^0 } }} }
&
&
&
Y\ar[rr]^{\Id _ {{}_{Y}} }
\ar@/^3ex/@{{ }{ }}[dd]|-{\xRightarrow{\alpha _{{}_{\partial^1 }  }^{-1} } }
&
&
Y\ar@/_3ex/@{{ }{ }}[dd]|-{\xRightarrow{\alpha _{{}_{\partial^1 } }} }
&
\\
\BBB (\mathsf{1})
\ar[ru]^{\alpha _ {{}_{\mathsf{1} }}}
\ar@{{ }{ }}[r]|-{\xRightarrow{\alpha _{{}_{d^1 } }^{-1} } }
&
&
&
&
\BBB (\mathsf{1})
\ar[lu]_-{\alpha _ {{}_{\mathsf{1} }} }
\ar@{}[l]|-{\xRightarrow{\alpha _{{}_{d^0 } }}}
\ar@{{ }{ }}[r]|-{=}
&
\BBB (\mathsf{1})
\ar[ru]^{\alpha _ {{}_{\mathsf{1} }}}
\ar@{{ }{ }}[r]|-{\xRightarrow{\alpha _{{}_{d^1 } }^{-1} } }
&
&
&
&
\BBB (\mathsf{1})
\ar[lu]_-{\alpha _ {{}_{\mathsf{1} }} }
\ar@{}[l]|-{\xRightarrow{\alpha _{{}_{d^0 } }}}
\\
&
\BBB (\mathsf{2})
\ar[lu]^{\BBB (d^1) }
\ar[uu]|-{\alpha _ {{}_{\mathsf{2} }} }
&
\BBB(\mathsf{3})
\ar[l]^-{\BBB (\partial ^2) }
\ar@/_2ex/[luu]|-{\alpha _ {{}_{\mathsf{3} }} }
\ar[r]_-{\BBB (\partial ^0)}
\ar@/^2ex/[ruu]|-{\alpha _ {{}_{\mathsf{3} }} }
\ar@{}[uu]|-{=}
&
\BBB(\mathsf{2})\ar[uu]|-{\alpha _ {{}_{\mathsf{2} }} }\ar[ru]_-{\BBB(d ^0 ) } 
&
&
&
\BBB (\mathsf{2})
\ar[lu]^{\BBB (d^1) }
\ar[uu]|-{\alpha _ {{}_{\mathsf{2} }} }
&
\BBB(\mathsf{3})
\ar[l]^-{\BBB (\partial ^1) }
\ar@/_2ex/[luu]|-{\alpha _ {{}_{\mathsf{3} }} }
\ar[r]_-{\BBB (\partial ^1)}
\ar@/^2ex/[ruu]|-{\alpha _ {{}_{\mathsf{3} }} }
\ar@{}[uu]|-{=}
&
\BBB(\mathsf{2})\ar[uu]|-{\alpha _ {{}_{\mathsf{2} }} }\ar[ru]_-{\BBB(d ^0 ) } 
&
}
$$
which is clearly equal to the right side of the \ref{AssociatividadeFunctorial}.
\end{proof}

\begin{rem}\label{redundante}
One 
important difference between (pointwise) pseudo-Kan extensions (weighted bilimits)
and (pointwise) Kan extensions (strict $2$-limits) is the following:  
if we consider the inclusion $\t _ {{}_2}:\Delta _ {{}_{2}}\to \Delta $ of the 
full subcategory with only $\mathsf{1}$ and $\mathsf{2}$
as objects into the category $\Delta $, then $\Lan _{\t _ {{}_2} } \top \cong \top $ while 
$\Ps\Lan _{\t _ {{}_2} } \top \not\simeq \top $, where, by abuse of language, $\top $ always
denotes the appropriate $2$-functor constantly equal to the terminal category. 
Actually, $\Ps\Lan _{\t _ {{}_2} } \top (\mathsf{3}) $
is equivalent to the category with only one object and one nontrivial automorphism.
\end{rem}

\begin{theo}\label{essentialleft}
Let $\top : \Delta _ {{}_{3}}\to \CAT $ and $\top :\Delta\to\CAT $ be the terminal weights. We have that 
$\Ps\Lan _{\t _ {{}_3} } \top \simeq \top $. 
\end{theo}

\begin{proof}
We prove below that, given a constant $2$-functor 
$ \underline{ Y}: \Delta _ {{}_{3}} \to \CAT$,
$$\left[ \Delta _ {{}_{3}} ^\op , \CAT \right] _ {PS}(\Delta (\t _ {{}_3}-, \mathsf{n}), \underline{ Y} )\simeq  
\CAT (\nabla \mathsf{n}, Y) $$
which, by the dual of Theorem \ref{pointwise} given in Remark \ref{leftpseudo}, 
completes our argument since it proves that 
$\bicolim (\Delta (\t _ {{}_3}-, \mathsf{n}), \top )\simeq \nabla \mathsf{n}\simeq \top (\mathsf{n})$.

Let $\varepsilon^\textrm{d}$ be the counit of the discrete  comonad on the category of small categories 
(see \ref{notationofmorphisms}), 
we define the functor
\small
$$\CAT (\nabla \mathsf{n}, Y)\to
\left[ \Delta _ {{}_{3}} ^\op , \CAT \right] _ {PS}(\Delta (\t _ {{}_3}-, \mathsf{n}), \underline{ Y} ),
\quad A \mapsto \xi ^A, \quad\left( \xxxx : A\rightarrow B\right) \mapsto \left(\xi ^\xxxx : \xi ^A\Longrightarrow \xi ^B\right)
 $$
\normalsize 
in which, given a functor $A: \nabla \mathsf{n} \to Y$ and a natural transformation $\xxxx : A\rightarrow B $,
$\xi ^A$ and $\xi ^\xxxx$ are defined by:
\begin{equation*}
\begin{aligned}
\xi ^A _{{}_{\mathsf{1}}} &: =  A\circ \varepsilon^\textrm{d}_{{}_{\mathsf{n} }},&\\
\xi ^A _{{}_{\mathsf{2}}} &: =   A\circ\varepsilon^\textrm{d}_{{}_{\mathsf{n} }} \circ \Delta (\t _ {{}_3}( d^1 ), \mathsf{n}),& \\
 \xi ^A _{{}_{\mathsf{3}}}& : =  A\circ \varepsilon^\textrm{d}_{{}_{\mathsf{n} }} \circ \Delta (\t _ {{}_3}( d^1\partial ^2), \mathsf{n}),& 										
\end{aligned}
\qquad
\begin{aligned}
\xi ^A _{{}_{d^1}} &: = \Id _ {{}_{\xi  _ {{}_{\mathsf{2} }} }}, &\\
\xi ^A _{{}_{s^0}}& : = \Id _ {{}_{\xi  _ {{}_{\mathsf{2} }} }}, & \\
\xi ^A _{{}_{\partial ^ 1 }}& : = \Id _ {{}_{\xi  _ {{}_{\mathsf{2} }} }}, & 
\end{aligned}
\qquad
\begin{aligned}
\left(\xi ^A _{{}_{d^0}}\right) _{{}_{f:\mathsf{2}\to\mathsf{n} }}&: = A(f(\mathsf{0})\leq f(\mathsf{1}) ),&\\
 \xi ^A _{{}_{\partial ^0}}& : = \Id _ {{}_{\Delta (\t _ {{}_3}(\partial ^2), \mathsf{n}) }}
 \ast \xi ^A _{{}_{d^0}}. &
\end{aligned}
\end{equation*} 
$$
\xi ^\xxxx _{{}_{\mathsf{1}}} : = \xxxx\ast \Id _ {{}_{\varepsilon^\textrm{d}_{{}_{\mathsf{n} }}}},\quad
\xi ^\xxxx_{{}_{\mathsf{2}}} : =  \xxxx\ast \Id _ {{}_{\varepsilon^\textrm{d}_{{}_{\mathsf{n} }} \circ \Delta (\t _ {{}_3}( d^1 ), \mathsf{n})}}, \quad
 \xi ^\xxxx  _{{}_{\mathsf{3}}} : =  \xxxx\ast \Id _ {{}_{\varepsilon^\textrm{d}_{{}_{\mathsf{n} }} \circ \Delta (\t _ {{}_3}( d^1\partial ^2), \mathsf{n})}}. 										
$$
We prove that this functor is actually an equivalence. Firstly, we define the inverse equivalence 
\small
$$\left[ \Delta _ {{}_{3}} ^\op, \CAT \right] _ {PS}(\Delta (\t _ {{}_3}-, \mathsf{n}), \underline{ Y} )
\to \CAT (\nabla \mathsf{n}, Y),
\quad \alpha\mapsto \wp ^\alpha , 
\quad\left( \yyyy : \alpha\Longrightarrow \beta\right) \mapsto \left(\wp ^\yyyy : \wp ^\alpha
\Longrightarrow \wp ^\beta\right)
 $$
\normalsize 
where $\left(\wp ^\yyyy\right)_{{}_{\mathsf{j} }}:= 
\left(\yyyy _{{}_{\mathsf{1} }}\right) _ {{}_{\mathsf{j} }}$ and $\wp ^\alpha (\mathsf{i}\leq\mathsf{j}) $ is the component of the natural 
transformation below on the object $(\mathsf{i}, \mathsf{j} ): \mathsf{2}\to \mathsf{n} $ of 
$\Delta (\t _ {{}_3}(\mathsf{2}), \mathsf{n})$.
$$
\xymatrix{
&
Y
&
\\
\Delta(\t _ {{}_3}(\mathsf{1}), \mathsf{n})
\ar@{}[r]|-{\xRightarrow{\alpha _{{}_{d ^1} }^{-1} } }
\ar[ru]^{\alpha _{{}_\mathsf{1} } }
&
&
\Delta (\t _ {{}_3}(\mathsf{1}), \mathsf{n} )
\ar@{}[l]|-{\xRightarrow{\alpha _{{}_{d ^0} } } }
\ar[lu]_-{\alpha _{{}_\mathsf{1} } }
\\
&
\Delta (\t _ {{}_3}(\mathsf{2}), \mathsf{n} )
\ar[lu]^-{\Delta (d ^1, \mathsf{n} ) }
\ar[ru]_-{\Delta (d^0, \mathsf{n} ) }
\ar[uu]|-{\alpha _{{}_\mathsf{2} } }
&
}$$
It remains to show that $\wp ^\alpha $ defines a functor $ \nabla \mathsf{n}\to Y $. Indeed, 
 this follows from the \ref{AssociatividadeFunctorial} and the \ref{IdentidadeFunctorial} of 
Proposition \ref{technicalll}. More precisely, $\alpha $ satisfies the equations of this proposition, 
since $\Delta (\t _ {{}_3}-, \mathsf{n} )$
is a $2$-functor. Given $\mathsf{i}\leq \mathsf{j}\leq \mathsf{k} $ of 
$\nabla \mathsf{n} $,  by the definition of $\wp ^\alpha$, 
$\wp ^\alpha(\mathsf{j}\leq \mathsf{k} )\wp ^\alpha(\mathsf{i}\leq \mathsf{j} )$ is the component of the 
natural transformation of the
left side of the \ref{AssociatividadeFunctorial} on 
$(\mathsf{i}, \mathsf{j}, \mathsf{k} ): \mathsf{3}\to \mathsf{n} $, while the component of the right 
side on $(\mathsf{i}, \mathsf{j}, \mathsf{k} )$ is equal to
$\wp ^\alpha(\mathsf{i}\leq \mathsf{k} )$.
Analogously,  the \ref{IdentidadeFunctorial} implies that 
$\wp ^\alpha (\id _{{}_{\mathsf{i} }}) = \id _{{}_{\wp ^\alpha (\mathsf{i}) }} $.

Finally, since it is clear that $\wp ^{\xi ^{(-)}} = \Id _{{}_{\CAT (\nabla \mathsf{n}, Y)}}$, 
the proof is completed by showing the natural isomorphism 
$$\varGamma : \xi ^{\wp ^{(-)}}\Longrightarrow
\Id _{{}_{\left[ \Delta _ {{}_{3}} ^\op , \CAT \right] _ {PS}(\Delta (\t _ {{}_3}-, \mathsf{n}), \underline{ Y} )}}  $$
where each component is the invertible modification defined by:
$$\left(\varGamma _{{}_{\alpha  }}\right)_{{}_{\mathsf{1}}}: =\Id _{{}_{\alpha _{{}_{\mathsf{1} } }}},   \quad 
\left(\varGamma _{{}_{ \alpha}}\right) _{{}_{\mathsf{2} } }: = \alpha _{{}_{ d ^1 } }, \quad 
\left(\varGamma_{{}_{ \alpha}}\right) _{{}_{ \mathsf{3} } } : =  \alpha _{{}_{ d ^1 \partial^2  } }.
$$ 
\end{proof}

\begin{theo}[Descent Objects]\label{indeeddescent}
Let $\AAA : \Delta \to\hhh $ be a pseudofunctor. 
We have that $\Desc (\AAA )\simeq \Ps\Ran _ \j \AAA (\mathsf{0}) $.
\end{theo}
\begin{proof}
By Remarks \ref{conical1} and \ref{conical2}, 
we need to prove that the conical bilimit of $\AAA $ is equivalent to the 
conical bilimit of $\AAA\circ \t _ {{}_3} $. Indeed, by Theorems \ref{relation} and \ref{essentialleft},
$$\bilim ( \top , \AAA\circ \t _ {{}_3} )\simeq \bilim ( \Ps\Lan _{\t _ {{}_3} }\top, \AAA )
\simeq \bilim ( \top , \AAA ). $$
\end{proof}

Observe that, by Theorem \ref{indeeddescent}, 
if $\AAA : \dot{\Delta} \to\hhh $ is a pseudofunctor, then $\AAA $ is of 
(almost/effective) $\j$-descent if and only if $\AAA\circ\dot{\t } _ {{}_{3}} $ is of 
(almost/effective) $\j _ {{}_3}$-descent.

\subsection{Strict Descent Objects}\label{strictdescent}
To finish this section, we show how we can see descent objects 
via (strict/enriched) Kan extensions of $2$-diagrams. 
Although this construction gives a few strict 
features of descent theory (such as the strict factorization), 
we do not use the results of this part in the rest of the paper.
For this reason, the reader can skip this part and consider it to be technical observations on strict results.

Clearly, since the point of these observations is to consider strict results, unlike the general 
viewpoint of this paper, 
 we have to deal closely with coherence theorems. 
The coherence replacements used here follow from the $2$-monadic approach to general coherence 
results~\cite{SteveLack, Power, FLN}. Also, to formalize some observations of 
free $2$-categories, we use the concept of computad, defined in \cite{Street3}.  
Therefore it is clear that this part assumes knowledge on coherence~\cite{SteveLack, FLN},  
\textit{icons}~\cite{SLACK2007, FLN4},  
computads~\cite{Street3, FLN4} and flexible weighted limits~\cite{Flexible}. Moreover, we omit most of the 
proofs of this last part of this section.

The first step is actually older than the general coherence results: the strict replacement of a bicategory.
Consider the locally full inclusion $\Icon \to \Bicat $ of the $2$-category $\Icon $ of $2$-categories and $2$-functors 
into the $2$-category $\Bicat $  of $2$-categories, pseudofunctors and \textit{icons}.
By the general coherence result~\cite{SteveLack, FLN}, this inclusion has a left $2$-adjoint  
$\mathrm{Str}: \Bicat \to \Icon$,
and the unit of
this $2$-adjunction is a pseudonatural equivalence (which means that it is pointwise an equivalence in 
$\Bicat$).

\begin{lem}\label{strictreplacementexplained}
If there is an equivalence $\mathrm{Str}(\aaa )\to\aaa _{{}_{\mathrm{Str}}}$ 
in $\Icon $, the inclusion
$$\left[ \aaa_{{}_{\mathrm{Str}}}, \hhh \right] \to\left[ \aaa_{{}_{\mathrm{Str}}}, \hhh \right] _ {PS}$$
is essentially surjective (which means that every pseudofunctor  $\aaa_{{}_{\mathrm{Str}}}\to \hhh$
is pseudonaturally isomorphic to a $2$-functor $\aaa_{{}_{\mathrm{Str}}}\to \hhh$).
\end{lem}
\begin{proof}
This follows from the fact the composition of the equivalences
$$\Icon (\mathrm{Str}(\aaa ) , \hhh )\cong \Bicat (\aaa , \hhh)\simeq 
\Bicat (\mathrm{Str}(\aaa ) , \hhh ), $$
in which $\Bicat (\aaa , \hhh)\simeq 
\Bicat (\mathrm{Str}(\aaa ) , \hhh )$ is the precomposition of the component of the unit on $\aaa $,
gives the inclusion $\Icon (\mathrm{Str}(\aaa ) , \hhh )\to
\Bicat (\mathrm{Str}(\aaa ) , \hhh ) $. Therefore this inclusion is an equivalence of categories.

In particular, for each pseudofunctor  
$\AAA : \mathrm{Str}(\aaa )\to \hhh$, there are a $2$-functor
$\AAA ' : \mathrm{Str}(\aaa )\to \hhh$ and 
an invertible \textit{icon} $\AAA \longrightarrow \AAA '$. This property clearly holds for any $\aaa_{{}_{\mathrm{Str}}}$ such that there is an equivalence
 $\aaa_{{}_{\mathrm{Str}}}\to \mathrm{Str}(\aaa ) $ in $\Icon $.

Since invertible \textit{icons} are pseudonatural isomorphisms
with identity $1$-cell components, this fact proves that 
$\left[ \aaa_{{}_{\mathrm{Str}}}, \hhh \right] \to\left[ \aaa_{{}_{\mathrm{Str}}}, \hhh \right] _ {PS}$ is indeed essentially surjective.
\end{proof}

Herein, given a small $2$-category $\aaa $, a \textit{strict replacement} of $\aaa $ is a $2$-category $\aaa _{{}_{\mathrm{Str}}}$ such that there is an equivalence 
$\aaa _{{}_{\mathrm{Str}}}\to\mathrm{Str}(\aaa ) $ in $\Icon $. Thus strict replacements are clearly unique up to equivalence and choices of strict replacements
define a left biadjoint to
the inclusion $\Icon\to\Bicat$.

\begin{defi}
A $2$-category $\aaa $ is \textit{locally groupoidal} if every hom-category $\aaa (a,b) $ 
is a groupoid. Moreover, $\aaa $ is \textit{locally thin} if
there is at most one $2$-cell $f\Rightarrow g $ for every ordered pair of $1$-cells $(f,g) $ of $\aaa $.
Finally, $\aaa $ is \textit{locally thin groupoidal}, or, for short, \textit{locally t.g.},  if it is locally groupoidal and locally thin.
\end{defi}

\begin{defi}\label{Freelocallypreordered}
We denote by $\dot{\Delta}_{{}_{\mathrm{Str}}}  $ the locally t.g. $2$-category freely generated by the diagram
\[\xymatrix{  \mathsf{0} \ar[rr]^-d && \mathsf{1}\ar@<1.7 ex>[rrr]^-{d^0}\ar@<-1.7ex>[rrr]_-{d^1} &&& \mathsf{2}\ar[lll]|-{s^0}
\ar@<1.7 ex>[rrr]^{\partial ^0}\ar[rrr]|-{\partial ^1}\ar@<-1.7ex>[rrr]_{\partial ^2} &&& \mathsf{3} }\] 
with the $2$-cells:
\begin{equation*}
\begin{aligned}
\sigma_{01} &:&  \partial ^1 d ^0\Rightarrow \partial^{0}d^{0},\\
\sigma_{02} &:&  \partial ^2 d ^0\Rightarrow \partial^{0}d^{1},\\
\sigma_{12} &:&  \partial ^2 d ^1\Rightarrow \partial^{1}d^{1},									
\end{aligned}
\qquad\qquad\qquad
\begin{aligned}										
 n_0        &:&  s^0d^0\Rightarrow \id _{{}_\mathsf{1}},  \\
 n_1        &:&  s^{0}d^{1}\Rightarrow \id _{{}_\mathsf{1}},\\
 \vartheta        &:&  d^1d\Rightarrow d^0d.
\end{aligned}
\end{equation*}
We consider the full inclusion $\j _{{}_{\mathrm{Str}}} : \Delta_{{}_{\mathrm{Str}}} \to\dot{\Delta}_{{}_{\mathrm{Str}}}  $ in which 
$\obj (\Delta_{{}_{\mathrm{Str}}} ) = \left\{ \mathsf{1}, \mathsf{2}, \mathsf{3}\right\} $.
\end{defi}

\begin{rem}
Observe that the diagram and the invertible $2$-cells described above define a computad~\cite{Street3} (or, more appropriately, a groupoidal computad~\cite{FLN4})  
which we denote by $\dot{\mathfrak{\Delta}}$. 
Thereby Definition \ref{Freelocallypreordered} is precise in the following sense: 
there is a forgetful functor between the category of locally  {t.g.} $2$-categories and 
the category of (groupoidal) computads. This forgetful functor has a left adjoint which gives the 
locally {t.g.}
$2$-category freely generated  by each computad. 
The (locally groupoidal) $2$-category $\dot{\Delta}_{{}_{\mathrm{Str}}}  $ is, by definition, 
the image of the computad $\dot{\mathfrak{\Delta}}$ by this left adjoint functor. 
\end{rem}

The $2$-categories $\dot{\Delta}_{{}_{\mathrm{Str}}}  $ and $\Delta_{{}_{\mathrm{Str}}}$ are strict 
replacements of the $2$-categories $\dot{\Delta } _ {{}_3}$ and $\Delta  _ {{}_3}$ respectively. 
Actually, $\j _{{}_{\mathrm{Str}}}$ is a strict replacement of $\j _ {{}_3} $.
By this fact and by the result that descent objects are flexible~\cite{Flexible}, we get:

\begin{prop}\label{coherencestrict}
There are obvious biequivalences $\Delta_{{}_{\mathrm{Str}}}\approx\Delta _ {{}_{3}}$ and $\dot{\Delta}_{{}_{\mathrm{Str}}}\approx\dot{\Delta } _ {{}_{3}}$ which are bijective on objects. Also, if $\hhh $ is any $2$-category, 
$\left[ \Delta_{{}_{\mathrm{Str}}}, \hhh \right] \to\left[ \Delta_{{}_{\mathrm{Str}}}, \hhh \right] _ {PS}$ is essentially surjective.
Moreover, for any $2$-functor $\CCCC : \Delta_{{}_{\mathrm{Str}}}\to \CAT $, we have an equivalence
$$\left[ \Delta_{{}_{\mathrm{Str}}}, \CAT \right]  (\dot{\Delta}_{{}_{\mathrm{Str}}}(\mathsf{0},\j _{{}_{\mathrm{Str}}}(-) ), \CCCC )\simeq\left[ \Delta_{{}_{\mathrm{Str}}}, \CAT \right] _ {PS} (\dot{\Delta}_{{}_{\mathrm{Str}}}(\mathsf{0},\j _{{}_{\mathrm{Str}}}(-) ), \CCCC ).$$ 
\end{prop}
\begin{proof}
The last part of the result follows from the fact that the descent object is a flexible weighted limit (see \cite{Flexible}). The rest follows
from the fact that $\Delta_{{}_{\mathrm{Str}}}$ is the strict replacement of $\Delta  _ {{}_3}$ (see Lemma \ref{strictreplacementexplained}).
\end{proof}

\begin{cor}\label{lastresultflexible}
If $\AAA : \Delta _ {{}_{\mathrm{Str}}}\to\hhh $ is a $2$-functor, 
$$
\Ps\Ran _ {\j _ {{}_3}} \check{\AAA }
\simeq 
\Ps\Ran _ {\j _{{}_{\mathrm{Str}}}} \AAA 
\simeq 
\Ran _ {\j _{{}_{\mathrm{Str}}}} \AAA  
$$
provided that the pointwise Kan extension $\Ran _ {\j _{{}_{\mathrm{Str}}}} \AAA$ exists,
in which $\check{\AAA }$ is the composition of $\AAA $ with the biequivalence $\Delta _ {{}_{3}}\approx\Delta_{{}_{\mathrm{Str}}}$.
\end{cor}

Assuming that the pointwise Kan extension $\Ran _ {\j _{{}_{\mathrm{Str}}}} \AAA$ 
exists, $\Ran _ {\j _{{}_{\mathrm{Str}}}} \AAA (\mathsf{0}) $ is called the 
\textit{strict descent diagram} of $\AAA $. By Corollary \ref{lastresultflexible}, 
the descent object of $\AAA $ is equivalent to its strict descent object provided that $\AAA $ 
has a strict descent object.

\begin{rem}\label{FACTSTRICT}
Using the strict descent object, we can construct the ``strict'' factorization described in Section \ref{Basic Problem}. If $\AAA : \dot{\Delta} _{{}_{\mathrm{Str}}}  \to \hhh $ is a $2$-functor and $\hhh $ has strict descent objects, we get the factorization from the universal property of the right Kan extension of $\AAA \circ\j _{{}_{\mathrm{Str}}}: \Delta _{{}_{\mathrm{Str}}}\to \hhh $ along $\j_{{}_{\mathrm{Str}}} $. More precisely, since $\j _ {{}_{\mathrm{Str}}} $ is fully faithful, we can consider that $\Ran _ {{}_{\j_{{}_{\mathrm{Str}}}}}\AAA \circ \j_{{}_{\mathrm{Str}}} $ is actually a strict extension of $\AAA \circ \j _ {{}_{\mathrm{Str}}}$. Thereby we get the factorization
$$\xymatrix{ & & \Ran _ {{}_{\j_{{}_{\mathrm{Str}}}}} (\AAA \circ \j_{{}_{\mathrm{Str}}}) (\mathsf{0})  \ar[dd]^{\Ran _ {{}_{\j_{{}_{\mathrm{Str}}}}} (\AAA \circ \j_{{}_{\mathrm{Str}}}) (d)}\\
&&\\
\AAA  (\mathsf{0}) \ar[rruu]^{\eta ^\mathsf{0}_ {{}_{\AAA }} }\ar[rr]_{ \AAA ( d )} &&\AAA (\mathsf{1})}$$
in which $\eta ^\mathsf{0}_ {{}_{\AAA }} $ is the comparison induced by the unit/comparison 
$\eta _ {{}_{\AAA }}: \AAA \longrightarrow \Ran _ {{}_{\j_{{}_{\mathrm{Str}}}}} (\AAA \circ \j_{{}_{\mathrm{Str}}}) $.      
\end{rem}

\begin{rem}
As observed in Section \ref{pseudoKAN}, the right Kan extension of a $2$-functor $\AAA : \Delta\to\hhh $ 
along $\j $ gives the equalizer of $\AAA (d^0 ) $ and $\AAA (d^1) $. 
This is a consequence of the isomorphism $\Lan _ {{}_{\t _ {{}_2} }} \top \cong \top $ of 
Remark \ref{redundante}.
\end{rem}

We get a glimpse of the explicit nature of the (strict) descent object at 
Theorem \ref{presentation} which gives a presentation to $\dot{\Delta}_{{}_{\mathrm{Str}}}$. 
We denote by $\mathcal{F}_{g} (\dot{\mathfrak{\Delta}})$ the locally groupoidal $2$-category freely 
generated by the diagram and $2$-cells described in Definition \ref{Freelocallypreordered}. 
It is important to note that $\mathcal{F}_{g} (\dot{\mathfrak{\Delta}})$ is not locally thin.
Moreover, there is an obvious $2$-functor $\mathcal{F}_{g} (\dot{\mathfrak{\Delta}})\to\dot{\Delta}_{{}_{\mathrm{Str}}}$, 
induced by the unit of the adjunction between the category of locally groupoidal $2$-categories and the 
category of locally t.g. $2$-categories.

\begin{theo}[\cite{FLN4}]\label{presentation}
Let $\hhh $ be a $2$-category.
There is a bijection between $2$-functors $\AAA : \dot{\Delta}_{{}_{\mathrm{Str}}}\to\hhh $ and $2$-functors $\mathsf{A}:\mathcal{F}_{g} (\dot{\mathfrak{\Delta}})\to\hhh$ satisfying the following equations:
\begin{itemize}\renewcommand\labelitemi{--}
\item Associativity:
\small
$$\xymatrix@C=3em@R=1.5em{  \mathsf{A}(\mathsf{0})\ar[r]^-{\mathsf{A}(d)}\ar[d]_-{\mathsf{A}(d)}\ar@{}[rd]|-{\xRightarrow{\hskip .1em \mathsf{A}(\vartheta ) \hskip .1em } }&
\mathsf{A}(\mathsf{1})\ar[d]|-{\mathsf{A}(d^0)}\ar[r]^-{\mathsf{A}(d^0)}\ar@{}[rd]|-{\xRightarrow{\hskip .1em \mathsf{A}(\sigma_{01}) \hskip .1em }}&
\mathsf{A}(\mathsf{2})\ar[d]^-{\mathsf{A}(\partial ^0)}\ar@{}[rrdd]|-{=}&
&
\mathsf{A}(\mathsf{3})\ar@{}[rd]|-{\xRightarrow{\hskip .1em \mathsf{A}(\sigma_{02}) \hskip .1em }}&
\mathsf{A}(\mathsf{2})\ar[l]_-{\mathsf{A}(\partial ^0)}\ar@{}[rdd]|-{\xRightarrow{\hskip .1em \mathsf{A}(\vartheta )\hskip .1em }}\ar@{=}[r]&
\mathsf{A}(\mathsf{2})
\\
\mathsf{A}(\mathsf{1})\ar[r]|-{\mathsf{A}(d^1)}\ar[d]_-{\mathsf{A}(d^1)}\ar@{}[rrd]|-{\xRightarrow{\hskip .1em \mathsf{A}(\sigma _ {12}) \hskip .1em }}&
\mathsf{A}(\mathsf{2})\ar[r]|-{\mathsf{A}(\partial ^1)}&
\mathsf{A}(\mathsf{3})\ar[d]^-{\mathsf{A}(\id _ {{}_\mathsf{3}})}&
&
\mathsf{A}(\mathsf{2})\ar@{}[rd]|-{\xRightarrow{\hskip .1em \mathsf{A}(\vartheta ) \hskip .1em }}\ar[u]^-{\mathsf{A}(\partial ^2)}&
\mathsf{A}(\mathsf{1})\ar[l]|-{\mathsf{A} (d^0)}\ar[u]|-{\mathsf{A}(d^1)}&
\\
\mathsf{A}(\mathsf{2})\ar[rr]_-{\mathsf{A}(\partial ^2)}&&
\mathsf{A}(\mathsf{3})&
&
\mathsf{A}(\mathsf{1})\ar[u]^-{\mathsf{A}(d^1)}&
\mathsf{A}(\mathsf{0})\ar[l]^-{\mathsf{A}(d)}\ar[u]|-{\mathsf{A}(d)}\ar[r]_ {\mathsf{A}(d)}&
\mathsf{A}(\mathsf{1})\ar[uu]_-{\mathsf{A}(d^0)}   
}$$
\normalsize
\item Identity:
\small
$$\xymatrix@C=0.7em@R=1.2em{  \mathsf{A}(\mathsf{0} )    
\ar[rr]^{\mathsf{A}(d)} 
                    \ar[dd]_{\mathsf{A}(d)} 
                      && 
              {\mathsf{A}(\mathsf{1})}  
              \ar[dd]|-{\mathsf{A}(d^1)}
                    \ar@{=}@/^4ex/[dddr]
                    \ar@{}[dddr]|{\xLeftarrow{\mathsf{A}(n_1)}} &&&
        \mathsf{A}(\mathsf{0})    
        \ar@/_3ex/[ddd]|-{\mathsf{A}(d)}
                    \ar@{}[ddd]|=
                    \ar@/^3ex/[ddd]|-{\mathsf{A}(d)}
                                        \\ 
          & \xLeftarrow{\hskip .1em \mathsf{A}(\vartheta ) \hskip .1em }  && &                            
                                        \\
              {\mathsf{A}(\mathsf{1})}  \ar[rr]|-{\mathsf{A}(d^0)}
                    \ar@{=}@/_4ex/[drrr]
                    \ar@{}[drrr]|{\xLeftarrow{\mathsf{A}(n_0)}} &&
              {\mathsf{A}(\mathsf{2})}\ar[dr]|{\mathsf{A}(s^0)}  \\                                            
          &&& {\mathsf{A}(\mathsf{1})}\ar@{}[uuur]|-{=} && {\mathsf{A}(\mathsf{1})} }$$
\normalsize          
\end{itemize}				
\end{theo}

\begin{rem}[\cite{FLN4}]
$\Delta_{{}_{\mathrm{Str}}}$ is the locally groupoidal $2$-category freely generated by the corresponding diagram and  
$2$-cells $\sigma_{01}$, $\sigma_{02}$, $\sigma_{12}$, $n_0$, $n_1$.
\end{rem}

\begin{rem}[\cite{FLN}]\label{strictdescentstrict}
The $2$-category $\CAT $ is $\CAT$-complete. In particular, $\CAT $ 
has strict descent objects. More precisely, if 
$\AAA : \Delta  _ {{}_{\mathrm{Str}}}\to\CAT $
is a $2$-functor, then
$$\mathrm{lim}(\dot{\Delta }_ {{}_{\mathrm{Str}}}(\mathsf{0}, \j _{{}_{\mathrm{Str}}}(-) ), \AAA ) 
\cong \left[ \Delta_ {{}_{\mathrm{Str}}} , \CAT\right] \left( \dot{\Delta }_ {{}_{\mathrm{Str}}}(\mathsf{0}, \j _{{}_{\mathrm{Str}}}(-)), \AAA \right) .$$
Thereby, we can describe the category the strict descent object of $\AAA :\Delta _ {{}_{\mathrm{Str}}}\to\CAT $ 
explicitly as follows:

\begin{enumerate}
\item Objects are $2$-natural transformations $\mathsf{W}: \dot{\Delta }_ {{}_{\mathrm{Str}}}(\mathsf{0}, -)\longrightarrow \AAA $. We have a bijective correspondence between such $2$-natural transformations and pairs $(W, \varrho _ {{}_{\mathsf{W}}})$ in which $W$ is an object of $ \AAA(\mathsf{1}) $ and $\varrho _ {{}_{\mathsf{W}}}: \AAA (d^1)(W)\to\AAA (d^0)(W) $ is an isomorphism in $ \AAA (\mathsf{2}) $ satisfying the following equations:	
\begin{itemize}\renewcommand\labelitemi{--}
\item Associativity:
\small
$$\left(\AAA (\partial ^0)(\varrho _ {{}_{\mathsf{W} }} )\right) \left( \AAA (\sigma _ {{}_{02}}) _ {{}_{W}}\right)\left(\AAA (\partial ^2)(\varrho _ {{}_{\mathsf{W}}} )\right)\left(\AAA (\sigma _ {{}_{12}} ) ^{-1}_ {{}_W}\right) = \left(\AAA (\sigma _ {{}_{01}}) _ {{}_{W}}\right)\left(\AAA(\partial ^1)(\varrho _ {{}_{\mathsf{W}}})\right)   $$
\normalsize
\item Identity:
\small
$$\left(\AAA(n_0) _ {{}_W}\right)\left(\AAA(s^0) (\varrho _ {{}_{\mathsf{W}}}) \right)\left(\AAA(n_1) _ {{}_W}\right) = \id _ {{}_{W}} $$
\normalsize
\end{itemize}
If $\mathsf{W}: \dot{\Delta }(\mathsf{0}, -)\longrightarrow \AAA $ is a $2$-natural transformation, we get such pair by the correspondence 
$\mathsf{W}\mapsto (\mathsf{W} _ {{}_{\mathsf{1} }}(d), \mathsf{W} _ {{}_{\mathsf{2} }}(\vartheta )) $.

\item The morphisms are modifications. In other words, a morphism $\mathsf{m} : \mathsf{W}\to\mathsf{X} $ is determined by a morphism $\mathfrak{m}: W\to X $ such that $\AAA (d^0)(\mathfrak{m} )\varrho _ {{}_\mathsf{W}} =  \varrho _ {{}_X}\AAA (d^1)(\mathfrak{m} ) $.
\end{enumerate}
\end{rem}

\section{Elementary Examples}\label{Elementary Examples}
We use some particular elementary examples of inclusions $\t : \aaa\to \dot{\aaa} $
for which we can study the $\left\langle \t \right\rangle $-pseudoalgebras/effective $\t$-descent 
diagrams 
in the setting of Section \ref{Formal}. These examples are given herein.

Let $\hhh $ be a $2$-category with enough bilimits to construct our 
pseudo-Kan extensions as global pointwise pseudo-Kan extensions.
The most simple example is taking the final category $ \mathsf{1} $ and the 
inclusion $\mathsf{0}\to \mathsf{1} $ of the empty category/empty ordinal. In this case, a 
pseudofunctor $\AAA : \mathsf{1}\to \hhh $ is of  effective descent if and only if this pseudofunctor 
(which corresponds to an object of $\hhh $) is equivalent to the pseudofinal object of $\hhh $.

If, instead, we take the inclusion $ d^0: \mathsf{1}\to \mathsf{2} $ of the ordinal 
$ \mathsf{1} $ into the ordinal  $\mathsf{2}$ such that $d^0$ is the inclusion of the codomain object, 
then a pseudofunctor $\AAA : \mathsf{2}\to \hhh $ corresponds to a $1$-cell of $\hhh$ and $\AAA $ is of 
effective $ d^0$-descent if and only if its image is an equivalence $1$-cell. Moreover, $\AAA $ is almost 
$ d^0$-descent/$d^0$-descent if and only if its image is faithful/fully faithful. Precisely, 
the comparison morphism would be the image $\AAA (\mathsf{0}\stackrel{d}{\rightarrow} \mathsf{1})$ of 
the only nontrivial $1$-cell of $\mathsf{2}$.

Furthermore, we may consider the following $2$-categories $ \dot{\bbb }$. The first one corresponds to the bilimit notion of comma object, while the second corresponds to the notion of pseudopullback.  

$$\xymatrix{ \mathsf{b} \ar[r] \ar[d] & e \ar[d] && \mathsf{b} \ar[r] \ar[d] & e\ar[d] \\\ar@{}|{\Rightarrow}[ru]
c\ar[r] &o && 
c\ar[r] &o }$$

As explained in Remark \ref{conicalpseudo}, all the examples above but the comma object are conical bilimits: it is clear that we 
can get every conical bilimit via a pseudo-Kan extension.  Actually, we can study the exactness of any 
weighted bilimit in our setting. More precisely, if $\WWW : \aaa \to \CAT$ is a weight, we can define $\dot{\aaa } $ adding an extra object $\mathsf{a}$ and defining 
$$\dot{\aaa }(\mathsf{a}, \mathsf{a}):= \ast\qquad\qquad \dot{\aaa }(\mathsf{a}, b):= \WWW (b) \qquad\qquad \dot{\aaa }(b, \mathsf{a}):= \emptyset $$
for each object $b$ of $\aaa $. Hence, it remains just to define the unique nontrivial composition, that is to say, we define the functor composition
$ \circ : \dot{\aaa }(b, c)\times\dot{\aaa }(\mathsf{a}, b)\to \dot{\aaa }(\mathsf{a}, c) $
for each pair of objects $b, c $ of $\aaa $ to be the ``mate'' of 
$$ \WWW _{{}_{bc}} : \dot{\aaa }(b, c)\to \CAT (\WWW (b), \WWW(c)). $$
Thereby, a pseudofunctor $\AAA : \dot{\aaa}\to\hhh $ is of effective $\t $-descent/$\t$-descent/almost  $\t $-descent if the canonical comparison $1$-cell
$\AAA (\mathsf{a})\to \bilim ( \WWW, \AAA\circ\t )$
is an equivalence/fully faithful/faithful.

\section{Eilenberg-Moore Objects}\label{EilenBERG}
Let $\hhh $ be a $2$-category as in \ref{Elementary Examples}. The $2$-category 
$\mathsf{Adj}$ such that an adjunction in a $2$-category corresponds to a 
$2$-functor $\mathsf{Adj}\to\hhh $ is described in \cite{Street}.
There is a full inclusion  
$\m : \mathsf{Mnd}\to \mathsf{Adj}$ such that monads in $\hhh $ correspond to $2$-functors
$\mathsf{Mnd}\to \hhh $. 
We describe this $2$-category below, and we show how it (still) works in our setting. The $2$-category $\mathsf{Adj}$ has two objects: $\mathsf{alg} $ and $\mathtt{b}$. The hom-categories are defined as follows:
$$
\mathsf{Adj}(\mathtt{b}, \mathtt{b}):= \dot{\Delta }\qquad \mathsf{Adj}(\mathsf{alg}, \mathtt{b}):= \Delta _- \qquad
\mathsf{Adj}(\mathsf{alg}, \mathsf{alg}):=  \Delta _ -^+ \qquad
\mathsf{Adj}(\mathtt{b}, \mathsf{alg}):=  \Delta ^+
$$
in which $\Delta _ - $ denotes the subcategory of $\Delta $ with the same objects such that its morphisms preserve initial objects and, analogously,  $\Delta _ + $ is the subcategory of $\Delta $ with the same objects and last-element-preserving arrows. Finally, $\Delta _ -^+$ is just the intersection of both  $\Delta _-$ and $\Delta ^+$.

Then the composition of $\mathsf{Adj} $  is such that
$
\mathsf{Adj}(\mathtt{b}, w)\times \mathsf{Adj}(c, \mathtt{b})\to \mathsf{Adj}(c, w)
$
is given by the usual ``ordinal sum'' $+$ (given by the usual strict monoidal structure of $\Delta $) for all objects $c,w $ of $\mathsf{Adj} $ and   
\begin{eqnarray*}
&\mathsf{Adj}(\mathsf{alg}, w)\times \mathsf{Adj}(c, \mathsf{alg})&\to \mathsf{Adj}(c, w)\\
&(x,y)&\mapsto x+y-\mathsf{1}\\
&(\phi : x\to x' , \upsilon : y\to y' ) & \mapsto \phi\oplus \upsilon
\end{eqnarray*}
in which
$$
    \phi\oplus \upsilon (i) := 
\begin{cases}
    \upsilon (i),& \text{if } i < y\\
    \phi (i-m) - 1+y'              & \text{otherwise}.
\end{cases}
$$

It is straightforward to verify that $\mathsf{Adj}$ is a $2$-category. We denote by $u$ the $1$-cell $\mathsf{1}\in \mathsf{Adj}(\mathsf{alg}, \mathtt{b}) $ and by $l$ the $1$-cell $\mathsf{1}\in \mathsf{Adj}(\mathtt{b}, \mathsf{alg})$. Also, we consider the following $2$-cells
$$
\dot{\Delta}(\mathsf{0}, \mathsf{1})\ni n: \id _ {{}_{\mathtt{b}}}\Rightarrow ul,\qquad\qquad\qquad
\Delta_-^+(\mathsf{1}, \mathsf{2})\ni e: lu\Rightarrow \id _{{}_{\mathsf{alg}}}.
$$
The $2$-category $\mathsf{Mnd} $ is defined to be the full sub-$2$-category of $\mathsf{Adj} $ with the unique object $\mathtt{b} $. As mentioned above, we denote its full inclusion by $\m : \mathsf{Mnd}\to \mathsf{Adj}$.

Firstly, observe that $(l\dashv u, n, e)$ is an adjunction in $\mathsf{Adj}$, therefore the image of $(l\dashv u, n, e)$ by a $2$-functor is an adjunction.
Also, if $(L\dashv U , \eta , \varepsilon ) $ is an adjunction in $\hhh $, then there is a unique $2$-functor
$\AAA : \mathsf{Adj}\to \hhh $
such that $\AAA (u) : = U $, $\AAA (l) : = L $, $\AAA (e):= \varepsilon $ and $\AAA (u) := \eta $. Thereby, it gives a bijection between adjunctions in $\hhh $ and $2$-functors $ \mathsf{Adj}\to\hhh $~\cite{Street}.

Secondly, as observed in \cite{Street}, there is a similar bijection between $2$-functors $\mathsf{Mnd}\to\hhh $ and monads in the $2$-category $\hhh $. Also, if the pointwise (enriched) Kan extension of a $2$-functor $ \mathsf{Mnd}\to\hhh $ along $\m $ exists, it gives the usual Eilenberg-Moore adjunction.
Moreover, given a $2$-functor $\AAA: \mathsf{Adj}\to\hhh $, if the pointwise Kan extension $\Ran _ \m \left( \AAA\circ\m \right) $ exists, the usual comparison $\AAA (\mathsf{alg})\to \Ran _ \m \left( \AAA\circ\m \right)(\mathsf{alg}) $ is the Eilenberg-Moore comparison $1$-cell.

If, instead, $\AAA :\mathsf{Adj}\to \hhh $ is a pseudofunctor, we also get that $\AAA (l)\dashv \AAA(u) $ and 
\begin{equation*}\tag{strict adjunction}\label{adjuncaoinduzida}
\left( \AAA (l)\dashv \AAA(u), \aaaa _{{}_{ul}}^{-1}\AAA(n) \aaaa _ {{}_{\mathtt{b}}},  \aaaa _ {{}_{\mathsf{alg} }}^{-1} \AAA (e) \aaaa _{{}_{lu}}  \right)
\end{equation*}
is an adjunction in $\hhh $. It is straightforward to verify that
the unique $2$-functor $\AAA ': \mathsf{Adj}\to\hhh $ corresponding to this adjunction is pseudonaturally isomorphic to $\AAA $. Furthermore, the Eilenberg-Moore object is a flexible limit as it is shown in \cite{Flexible}.

\begin{prop}
If $\hhh $ is any $2$-category, $\left[ \mathsf{Adj}, \hhh \right] \to\left[ \mathsf{Adj}, \hhh \right] _ {PS}$ is essentially surjective. 
Moreover, for any $2$-functor $\CCCC : \mathsf{Adj}\to \CAT $, we have an equivalence
$$\left[ \mathsf{Mnd}, \CAT \right]  (\mathsf{Adj} (\mathsf{alg},\m (-) ), \CCCC )\simeq\left[ \mathsf{Mnd}, \CAT \right] _ {PS} (\mathsf{Adj} (\mathsf{alg},\m (-) ), \CCCC ).$$ 
\end{prop}
\begin{proof}
In order to prove the first part, as mentioned above, it is enough to show that there is a pseudonatural isomorphism between $\AAA $ and $\AAA ' $. The $1$-cell components of this pseudonatural isomorphism $\alpha $ are identities, while the component $2$-cells are induced by the structure of pseudofunctor of $\AAA $ (the constraints/invertible $2$-cells).
The second part follows from the fact that Eilenberg-Moore objects are flexible weighted limits~\cite{Flexible}. 
\end{proof}

\begin{cor}
If $\AAA : \mathsf{Mnd}\to\hhh $ is a pseudofunctor, 
$$
\Ps\Ran _ {{\j _ {{}_3}} } \AAA 
\simeq 
\Ps\Ran _ {\m } \check{\AAA }
\simeq 
\Ran _ {\m } \check{\AAA }  
$$
provided that the pointwise Kan extension $\Ran _ {\m }  \check{\AAA } $ exists,
in which $\check{\AAA }$ is a $2$-functor pseudonaturally isomorphic to $\AAA$.
\end{cor}

Therefore, if $\hhh $ has Eilenberg-Moore objects, a pseudofunctor $\AAA : \mathsf{Adj}\to\hhh $ is of effective $\m $-descent/$\m $-descent/almost $\m$-descent if and only if 
$\AAA(u)$ is monadic/premonadic/almost monadic. Also, the ``factorizations''
\small
$$\xymatrix@C=1em{ \AAA (\mathtt{b})\ar[rr]|{\AAA (l)}\ar[dr]|{l_{{}_{{}_\AAA }} }&& \AAA (\mathsf{alg})&
\AAA (\mathsf{alg})\ar[rr]|{\AAA (u)}\ar[dr]|{\eta ^\mathsf{alg}_ {{}_{{}_{\AAA } }} }&& \AAA (\mathtt{b})\\
&\Ps\Ran _ \m (\AAA\circ\m ) (\mathsf{alg})\ar@{}[u]|{\cong }\ar[ru]|{\eta ^\mathsf{alg}_ {{}_{{}_{\AAA } }}} & & &\Ps\Ran _ \m (\AAA\circ\m ) (\mathsf{alg})\ar@{}[u]|{\cong }\ar[ru]|{u_{{}_{{}_\AAA }}} &
}$$  
\normalsize
described in Theorem \ref{fact} are pseudonaturally equivalent to the usual Eilenberg-Moore factorizations. Henceforth, these factorizations are called Eilenberg-Moore factorizations (even if the $2$-category $\hhh $ does not have the strict version of it).

\section{The Beck-Chevalley Condition}\label{BECK}
With these elementary examples, we  
can already give generalizations of Theorems \ref{facil1provado} and \ref{facil1}.
We keep our setting in which $\t : \aaa\to\dot{\aaa } $ is an $\mathsf{a}$-inclusion as in  
\ref{commu}.

Let $\TTT $ be an idempotent pseudomonad over the $2$-category $\hhh $. 
The most obvious consequence of the commutativity results of Section \ref{Formal} is the following: 
if an object $X$ of $\hhh $  can be endowed with a $\TTT $-pseudoalgebra structure and there is an 
equivalence $X\to W$, then $W$ can be endowed with a $\TTT $-pseudoalgebra as well.

In the case of pseudo-Kan extensions, we have the following: 
let $\AAA , \BBB : \dot{\aaa }\to\hhh $ be pseudofunctors. 
A pseudonatural transformation $\alpha: \AAA\longrightarrow \BBB $ can be seen as a pseudofunctor 
$\CCCC _ \alpha : \mathsf{2}\to \left[\dot{\aaa }, \hhh \right] _ {PS}$. By Corollaries \ref{comutatividade} 
and \ref{comutatividadealmostdescent}, we get the following: 
if  $\CCCC_\alpha (\mathsf{1}) $ is of effective $\t $-descent/$\t$-descent/almost  
$\t $-descent and the images of the mate $\dot{\aaa }\to \left[\mathsf{2}, \hhh \right] _ {PS}$ 
of $\CCCC _ \alpha $ are of effective $d ^0 $-descent/$d^0$-descent/almost $d^0$-descent as well, 
then $\CCCC_\alpha (\mathsf{0}) $ is also of effective $\t $-descent/$\t$-descent/almost  $\t $-descent. 
In Section \ref{proofsclassical}, we show that Theorem \ref{facil1} is a particular case of:

\begin{prop}\label{proofemb}
Let $\alpha : \AAA\longrightarrow \BBB $ be a 
pseudonatural transformation. If $\BBB $ is of effective $\t $-descent/$\t $-descent/almost  $\t $-descent 
and $\alpha $ is a pseudonatural equivalence/objectwise fully faithful/objectwise faithful, 
then $\AAA $ is of effective $\t $-descent/$\t $-descent/almost  $\t $-descent as well.
\end{prop}

\begin{defi}[Beck-Chevalley condition]
A pseudonatural transformation 
$\alpha: \AAA\longrightarrow \BBB $ 
 satisfies the \textit{Beck-Chevalley condition} 
if every $1$-cell component of $\alpha $ is left adjoint and, for each $1$-cell $f:w\to c $ of 
the domain of $\AAA $, 
the mate of the invertible $2$-cell 
$\alpha _ {{}_{f}} : \BBB (f)\alpha _{{}_w}\Rightarrow \alpha _ {{}_{c}}\AAA(f) $ w.r.t. 
 the adjunctions $\widehat{\alpha}^{{}^{w}}\dashv\alpha _ {{}_w}$ and 
 $\widehat{\alpha}^{{}^{c}}\dashv\alpha _ {{}_c}$ is invertible.
\end{defi}

By doctrinal adjunction~\cite{DoctrinalAdjunction}, $\alpha: \AAA\longrightarrow \BBB $ 
satisfies the Beck-Chevalley condition if and only if $\alpha $ is itself a right adjoint in the 
$2$-category  $\left[\dot{\aaa }, \hhh \right] _ {PS}$. In other words, we get:

\begin{lem}\label{BECKzero}
Let $\alpha: \AAA\longrightarrow \BBB $ be a pseudonatural transformation and $\CCCC _ \alpha : \mathsf{2}\to \left[\dot{\aaa }, \hhh \right] _ {PS}$ the corresponding pseudofunctor. 
Consider the inclusion $\mathtt{u}: \mathsf{2}\to\mathsf{Adj} $ of the morphism $u$. There is a pseudofunctor  $\widehat{\CCCC } _ \alpha: \mathsf{Adj}\to \left[\dot{\aaa }, \hhh \right] _ {PS}$ such that $ \widehat{\CCCC } _ \alpha\circ \mathtt{u} = \CCCC _ \alpha $
if and only if $\alpha $ satisfies the Beck-Chevalley condition.
\end{lem}

Thereby, as straightforward  
consequences of Corollaries \ref{comutatividadeequi} and \ref{comutatividadealmostcara}, using the terminology
of Lemma \ref{BECKzero}, 
we get what can be called a generalized version of  B\'{e}nabou-Roubaud Theorem: 

\begin{theo}\label{BFirst}
Assume that $\alpha: \AAA\longrightarrow \BBB $ is a pseudonatural transformation 
satisfying the Beck-Chevalley condition and all components of $\alpha\t = \alpha\ast Id _{{}_{\t }} $ are monadic. 
\begin{itemize}\renewcommand\labelitemi{--}

\item If $\BBB $ is of almost  $\t $-descent, then: 
$\alpha _ {{}_{\mathsf{a} }}$ is of almost $\m $-descent 
if and only if $\AAA $ is of almost  $\t $-descent;

\item If $\BBB $ is of $\t $-descent, then: $\alpha _ {{}_{\mathsf{a} }}$ is premonadic if and only if $\AAA $ is of $\t $-descent;

\item If $\BBB $ is of effective $\t $-descent, then: $\alpha _ {{}_{\mathsf{a} }}$ is monadic if and only if $\AAA $ is of effective $\t $-descent.

\end{itemize}

\end{theo}
\begin{proof}
Indeed, by the hypotheses, for each item, there is a pseudofunctor $\widehat{\CCCC } _ \alpha: \mathsf{Adj}\to \left[\dot{\aaa }, \hhh \right] _ {PS}$ satisfying the hypotheses of Corollary \ref{comutatividadeequi} or Corollary \ref{comutatividadealmostcara}.
\end{proof}

\begin{rem}
It is important to observe that the hypothesis of the theorem obviously does not include
the monadicity of $\alpha _{{}_{\mathsf{a} }}$, since $\t : \aaa\to \dot{\aaa } $ is an $\mathsf{a}$-inclusion.
\end{rem}

\section{Descent Theory}\label{proofsclassical}

In this section, we establish the setting of \cite{Facets, Facets2} and 
prove all the classical results mentioned in Section \ref{Basic Problem} 
for pseudocosimplicial objects, except Theorem \ref{mergulho} which is postponed to Section \ref{Further}. 

Henceforth, let $\CCC , \DDD $ be categories with pullbacks and $\hhh $ be a $2$-category
with the weighted bilimits whenever needed
as in the previous sections. In the context of \cite{Facets2}, given a pseudofunctor 
$\AAA : \CCC ^{\op }\to \hhh $, the morphism $p: E\to B $ of $\CCC $ is of 
\textit{ effective $\AAA $-descent/$\AAA $-descent/almost $\AAA $-descent} if 
$\AAA _p: \dot{\Delta }\to\hhh $ is of
effective $\j $-descent/$\j $-descent/almost $\j $-descent, where $\AAA _ p $ is the composition of 
the diagram 
$$\DDDD _ p :\dot{\Delta}^{\op }\to\CCC  $$
$$\xymatrix{ \cdots \ar@<0.3ex>[r]\ar@<-0.3ex>[r]\ar@<0.9ex>[r]\ar@<-0.9ex>[r] &
E\times _ p E\times _ p E \ar@<0.9 ex>[r]\ar[r]\ar@<-0.9ex>[r]\ar@/_2ex/[l]\ar@/^2ex/[l]\ar@/_3ex/[l] & E\times _ p E\ar@/_2ex/[l]\ar@/^2ex/[l]\ar@<0.9 ex>[r]\ar@<-0.9ex>[r] & E\ar[l]\ar[r]^-p& B } $$
with the pseudofunctor $\AAA $, in which the diagram above is given by the pullbacks of $p$ along itself, its projections and diagonal morphisms.
By the results of Section \ref{descent}, for $\hhh =\CAT $, this definition of effective $\AAA $-descent morphism
coincides with the classical one in the context of \cite{Facets, Facets2}.

We get the usual factorizations of (Grothendieck) $\AAA $-descent theory~\cite{Facets, Facets2} from Theorem \ref{fact}, although the usual strict factorization comes from Remark \ref{FACTSTRICT}. More precisely, if $p:E\to B $ is a morphism of $\CCC $, we get:

$$\xymatrix{ \AAA _ p(\mathsf{0}) = \AAA (B)\ar[rr]|{\AAA (p)}\ar[dr]|{\eta ^\mathsf{0}_ {{}_{{}_{\AAA\circ\DDDD _ p } }} }&& \AAA _ p(\mathsf{1}) = \AAA (E)\\
  &\Desc _ \AAA (p)\simeq \Ps\Ran _ \j (\AAA _ p\circ\j ) (\mathsf{0})\ar@{}[u]|{\cong }\ar[ru]|{d_{{}_{{}_{\AAA _ p} }}} &
}$$

In descent theory, a \textit{morphism} $(U, \alpha )$ between pseudofunctors $\AAA : \CCC ^{\op }\to \hhh $ and  $\BBB:\DDD ^{\op }\to \hhh $  is a  pullback preserving functor $U: \CCC\to\DDD $ with a pseudonatural transformation $\alpha : \AAA\longrightarrow\BBB\circ U $. Such a morphism is called \textit{faithful/fully faithful} if $\alpha $ is objectwise faithful/fully faithful.

For each morphism $p:E\to B $ of $\CCC $, a morphism $(U, \alpha )$ between pseudofunctors $\AAA : \CCC ^{\op }\to \hhh $ and  $\BBB:\DDD ^{\op }\to \hhh $ induces a pseudonatural transformation $\alpha ^{{}^p} : \AAA _ p \longrightarrow \BBB _{U(p)} $. Of course, $\alpha ^{{}^p} $ is objectwise faithful/fully faithful if $(U, \alpha ) $ is faithful/fully faithful.

We say that such a morphism $(U, \alpha )$ between pseudofunctors $\AAA : \CCC ^{\op }\to \hhh $ and  $\BBB:\DDD ^{\op }\to \hhh $ reflects almost descent/descent/effective descent morphisms if, whenever $U(p) $ is of almost $\BBB $-descent/$\BBB $-descent/effective $\BBB $-descent, $p$ is of almost $\AAA $-descent/$\AAA $-descent/effective $\AAA $-descent.

\begin{rem}\label{inducedbasicproblem}
Consider the pseudofunctor given by the 
basic fibration $(\mbox{ })^\ast : \CCC ^{\op }\to \CAT $ in which 
$$(p)^\ast : \CCC /B\to\CCC /E $$
is the change of base functor, given by the pullback along $p:E\to B $. 
For short, we say that a morphism $p:E\to B $ is of effective descent if $p$ is 
of effective $(\mbox{ })^{\ast }$-descent.

In this case, a pullback preserving functor $U:\CCC\to\DDD $ induces a morphism $(U, \mathtt{u} ) $ 
between the basic fibrations
$(\mbox{ })^\ast : \CCC ^{\op }\to \CAT $  and $(\mbox{ })^\ast : \DDD ^{\op }\to \CAT $ 
in which, 
for each object $B$ of $\CCC $, 
$\mathtt{u} _ {{}_B}:\CCC /B\to \DDD /U(B)  $ is given by the evaluation of $U$.  
If $U$ is faithful/fully faithful, so is the induced morphism $(U, \mathtt{u} ) $ between the basic fibrations.
\end{rem}  

We study pseudocosimplicial objects $\AAA : \dot{\Delta}\to\hhh $ and verify the obvious 
implications within the setting described above. We start with the embedding results 
(which are particular cases of \ref{proofemb}):

\begin{theo}[Embedding Results]\label{facil1provado}
Let  $\alpha : \AAA\longrightarrow \BBB $ be a pseudonatural transformation. 
If $\alpha $ is objectwise faithful and $\BBB $ is of almost $\j $-descent, then so is $\AAA $.
Furthermore, if $\BBB $ is of $\j$-descent and $\alpha $ is objectwise fully faithful, then $\AAA $ is of $\j$-descent as well. 
\end{theo}

Of course, we have that, if $\AAA\simeq \BBB$, then $\AAA $ is of almost 
$\j $-descent/$\j $-descent/effective $\j $-descent if and only if $\BBB $ is of 
almost $\j $-descent/$\j $-descent/effective $\j $-descent as well.

\begin{cor}\label{embdddd}
Let $(U, \alpha )$ be a morphism between the pseudofunctors $\AAA : \CCC ^{\op }\to \hhh $ and  
$\BBB:\DDD ^{\op }\to \hhh $ (as defined above). 
\begin{itemize}\renewcommand\labelitemi{--}
  \item If $(U, \alpha )$ is faithful, it reflects almost descent morphisms;
	\item If $(U, \alpha )$ is fully faithful, it reflects descent morphisms;
	\item If $\alpha $ is a pseudonatural equivalence, $(U, \alpha )$ reflects and preserves effective descent morphisms, descent morphisms and almost descent morphisms. 
\end{itemize}
\end{cor}

We finish this section by proving B\'{e}nabou-Roubaud Theorems. A functor $F$ is a \textit{pseudosection}
if there is $G$ such that $G\circ F$ is naturally isomorphic to the identity. We use the 
following straightforward result: 

\begin{lem}[Monadicity of pseudosections]\label{Monadicitykey}
If a pseudosection is right adjoint, then it is monadic.  In particular, if $\AAA $ is a 
pseudocosimplicial object,
then $\AAA (d ^i: \mathsf{n}\to\mathsf{n}+\mathsf{1} ) $ is monadic whenever it has a left adjoint.
\end{lem}
\begin{proof}
Assume that $F\circ G$ is isomorphic to the identity. Given an absolute colimit diagram $G\circ D$, 
it follows that $F\circ G\circ D\cong D $ is an absolute colimit diagram. The result follows, then,
from the monadicity theorem~\cite{Beck}.

The second part of the lemma follows from the fact that $d ^i$ is a retraction and, hence, since $\AAA $ is a pseudofunctor,
$\AAA (d ^i: \mathsf{n}\to\mathsf{n}+\mathsf{1} ) $ is a pseudosection for any $i\leq\mathsf{n}$.
\end{proof}

Recall that $\mathsf{1}$  is a monoid in $\dot{\Delta }$, as explained in Remark \ref{chosenmonoid}. On the one hand,
the monad induced by this monoid, considered, for instance, in
\cite{Street4} and \cite{Lawvere}, is denoted by $\suc : = (\mathsf{1}+-)$ on $\dot{\Delta }$. On the other hand,
this monad induces a pseudomonad 
$$\Suc : = \left[\suc , \hhh \right] _{PS} $$
on the $2$-category $\left[\dot{\Delta } , \hhh \right] _{PS}$ of pseudocosimplicial objects of $\hhh $. This is the 
$2$-dimensional (dual) analogue of the notion of d\'{e}calage of simplicial sets as in \cite{Duskin}.

In particular,
for each $\AAA : \dot{\Delta }\to\hhh $ the component of the unit of $\Suc $ on $\AAA $ gives a pseudonatural
transformation $\Suc ^\AAA : \AAA \longrightarrow \AAA \circ \Suc $ whose correspondent pseudofunctor is denoted by 
$\CCCC _ {{}_\AAA }:\mathsf{2}\to \left[ \dot{\Delta }, \hhh\right] _ {PS} $. 

Observe that,  $\CCCC _ {{}_\AAA }:\mathsf{2}\to \left[ \dot{\Delta }, \hhh\right] _ {PS} $ is 
given by the mate of $\AAA\circ\n : \mathsf{2}\times\dot{\Delta } \to\hhh $, where $\n $ is the mate of the unit 
of
$\suc $ viewed as a functor $\mathsf{2}\to \left[\dot {\Delta } , \dot{\Delta }\right]$, defined by 
$$
\n : \mathsf{2}\times\dot{\Delta } \to \dot{\Delta }$$
$$(a,b)\mapsto b + a \qquad\qquad (d,\id _ {{}_{b}}) \mapsto \left( d^0: b\to (b+1) \right) $$
$$(\id _ {{}_{a}}, d^i )  \mapsto 
\begin{cases}
    d^{i}:b\to (b+1),& \text{if } a=\mathsf{0}\\
    d^{i+1}: (b+1)\to (b+2),               & \text{otherwise}
\end{cases}$$
$$(\id _ {{}_{a}}, s^i )  \mapsto 
\begin{cases}
    s^{i}:b\to (b+1),& \text{if } a=\mathsf{0}\\
    s^{i+1}: (b+1)\to (b+2),               & \text{otherwise.}
\end{cases} $$

$$\xymatrix{  \mathsf{0} \ar[d]_{d}\ar[rr]^{d} && \mathsf{1} \ar[d]_ {d^0}\ar@<1.2 ex>[rr]\ar@<-1.2ex>[rr] && \mathsf{2}\ar[d]_{d^0} \ar[ll]
\ar@<1.2 ex>[rr]\ar[rr]\ar@<-1.2ex>[rr] && \mathsf{3} \ar[d]_ {d^0}\ar@/_2ex/@<-1 ex>[ll]|-{s^0}\ar@/_4ex/@<-1ex>[ll]|-{s^1}\ar@<1.2 ex>[rr]\ar@<0.4ex>[rr]\ar@<-0.4ex>[rr]\ar@<-1.2ex>[rr]&& \cdots \ar@/_1ex/@<-1 ex>[ll]|-{s^0}\ar@/_3ex/@<-1ex>[ll]|-{s^1}\ar@/_5ex/@<-1ex>[ll]|-{s^2}\ar[d]_{d^0} \\ 
\mathsf{1} \ar[rr]_-{d^1} && \mathsf{2} \ar@<1.2 ex>[rr]|{d^1}\ar@<-1.2ex>[rr]|{d^2} && \mathsf{3}\ar[ll]|-{s^1}
\ar@<1.2 ex>[rr]|-{d^1}\ar[rr]|-{d^2}\ar@<-1.2ex>[rr]|-{d^3} && \mathsf{4}\ar@/^2ex/@<1 ex>[ll]|-{s^1}\ar@/^4ex/@<1ex>[ll]|-{s^2}\ar@<1.2 ex>[rr]\ar@<0.4ex>[rr]\ar@<-0.4ex>[rr]\ar@<-1.2ex>[rr]&&\cdots\ar@/^1ex/@<1 ex>[ll]|-{s^1}\ar@/^3ex/@<1ex>[ll]|-{s^2}\ar@/^5ex/@<1ex>[ll]|-{s^3} }$$

We say that a pseudofunctor $\AAA : \dot{\Delta }\to\hhh $  satisfies the descent shift property 
(or just \textit{shift property} for short) if $\AAA\circ \Suc$ 
is of effective $\j$-descent. We get, then, a version 
of B\'{e}nabou-Roubaud Theorem for pseudocosimplicial objects: 

\begin{theo}\label{formalbenabou}
Let $\AAA : \dot{\Delta }\to\hhh $ be a pseudofunctor satisfying the shift property.
If the  
pseudonatural transformation $\Suc ^\AAA $  satisfies the Beck-Chevalley condition, 
then the Eilenberg-Moore factorization of $\AAA (d) $ is pseudonaturally equivalent to its usual factorization 
of $\j $-descent theory. In particular,
\begin{itemize}\renewcommand\labelitemi{--}
\item $\AAA$ is of effective $\j $-descent iff $\AAA (d) $ is monadic; 
\item $\AAA$ is of $\j $-descent iff $\AAA (d) $ is premonadic;
\item $\AAA$ is of almost $\j $-descent iff $\AAA (d) $ is almost monadic.
\end{itemize} 
\end{theo}
\begin{proof}
By Lemma \ref{Monadicitykey}, the components of  $\Suc ^\AAA\j = (\Suc ^\AAA)\ast \Id _ {{}_{\j }} $ are monadic. 
\end{proof}

It is known that in the context of \cite{Facets, Facets2} introduced in this section,  the natural
morphism $ E\times _p E\to E$ 
is always of effective $\AAA $-descent. It follows from this fact that $\AAA _p $ always satisfies 
the shift property. More precisely:

\begin{lem}
Let $\AAA : \CCC ^{\op}\to\CAT $ be a pseudofunctor, in which $\CCC $ is a category with pullbacks. 
If $p$ is a morphism of $\CCC $, $\AAA _ p$ (defined above as $\AAA _ p:=\AAA\circ\DDDD _ p $) 
satisfies the shift property.
\end{lem}
\begin{proof}
This follows from the fact that, for any pseudofunctor $\AAA : \CCC ^{\op}\to\CAT $, 
given a morphism $p:E\to B$ of $\CCC $, the natural
morphism $ E\times _p E\to E$ between the pullback of $p$ along $p$ and $E$ (being a split epimorphism)
is of effective $\AAA $-descent. In particular, $\AAA _ p\circ \Suc \simeq \AAA _ {E\times _B E\to E}$ is of
effective $\j $-descent.
\end{proof}

The usual \emph{B\'{e}nabou-Roubaud Theorem} (Theorem \ref{Roubaud}) follows
from Theorem \ref{formalbenabou}, as it is shown below. 

\begin{proof}
Assuming that $\AAA :\CCC ^{\op}\to\hhh  $ satisfies the hypotheses of Theorem \ref{Roubaud}, we have in particular that
$\Suc ^{\AAA _p} $ satisfies the Beck Chevalley condition. Therefore, since $\AAA _ p$ 
satisfies the shift property,
$\AAA _p (d) = \AAA (p) $ 
is monadic/premonadic/almost monadic
iff $\AAA _p $ is of effective $\j $-descent/$\j $-descent/almost $\j $-descent. 
\end{proof}

Finally, the most obvious consequence of the commutativity properties is that bilimits of effective 
$\j$-descent diagrams are effective $\j$-descent diagrams. For instance,
taking into account Remark~\ref{inducedbasicproblem} and realizing that 
pseudopullbacks of functors induce pseudopullback of overcategories we already get a weak version of
Theorem \ref{pseudopullback}.

Next section, we study stronger results on bilimits and apply them to descent theory.

\section{Further on Bilimits and Descent}\label{Further}

Henceforth, let $\t : \aaa\to\dot{\aaa }, \h : \bbb\to\dot{\bbb } $ be inclusions as in \ref{commu} and 
let $\hhh $ be a bicategorically complete $2$-category.

\begin{defi}[Pure Structure]\label{structuregeneral}
A morphism $f:\mathsf{a}\to b$ of $\dot{\aaa } $ is called a \textit{$\t$-irreducible} morphism  if 
$b\neq  \mathsf{a}$ and $f$ is not  in the image of
$$\circ : \dot{\aaa }(c,b)\times \dot{\aaa }(\mathsf{a},c)\to \dot{\aaa }(\mathsf{a}, b), $$
for every $b\neq c $ in $\aaa $.

An object $c$ of $\aaa  $ is called a \textit{$\t$-pure structure object}  if each $1$-cell 
$g$ of $\dot{\aaa }(\mathsf{a}, c) $ can be factorized through some $\t $-irreducible 
morphism $f:\mathsf{a}\to b $ such that $b\neq c $. That is to say, $c$ is a $\t$-pure structure object if,
for all $ g\in\dot{\aaa }(\mathsf{a}, c)$, there are a morphism $g'$ and a $\t $-irreducible morphism $f$ 
such that $g'f = g$. 

The full sub-$2$-category of the $\t $-pure structure objects of $\aaa  $ 
is denoted by $\mathfrak{S}_{{}_{\t }}$, while the full sub-$2$-category of $\dot{\aaa }  $ 
of the objects that are not in $\mathfrak{S}_{{}_{\t }}$  (including $\mathsf{a}$) is denoted by 
$\mathfrak{I} _{{}_{\t }} $. 
 We have the full inclusion  $\i_{{}_{\t}}:\mathfrak{I}_{{}_{\t }}\to \dot{\aaa } $.
\end{defi}

In particular, if $f:\mathsf{a}\to b $ is a $\t$-irreducible morphism of 
$\dot{\aaa } $, then $b$ is an object of $ \mathfrak{I}_{{}_{\t }}$.
We denote by 
$\gggg _ {{}_{\t}} :\overline{\mathfrak{I} _{{}_{\t }}
\times\mathsf{2} }\to \mathfrak{I} _{{}_{\t }}\times\mathsf{2}  $ the full inclusion in which 
$$\obj \left(\overline{\mathfrak{I} _{{}_{\t }}\times\mathsf{2} } \right):= 
\obj \left( \mathfrak{I} _{{}_{\t }}\times\mathsf{2}\right)-\left\{ (\mathsf{a}, \mathsf{0})\right\}. $$

\begin{theo}\label{GALOISI}
Let  $\alpha : \AAA\longrightarrow\BBB $ be an objectwise fully faithful pseudonatural transformation
such that
$\BBB $ is of effective $\t$-descent. We consider the mate of $\alpha $, denoted by 
$\CCCC _ \alpha : \dot{\aaa }\times \mathsf{2}\to \hhh $.
The pseudofunctor $\AAA $ is of effective $\t $-descent if and only if 
$\CCCC _ \alpha\circ \left(\i_{{}_{\t}}\times \Id _ {{}_{\mathsf{2} }}\right) : \mathfrak{I}_{{}_{\t }}\times\mathsf{2}\to \hhh  $ is of  effective $\gggg _ {{}_{\t}}$-descent.
\end{theo}

\begin{proof}
Without losing generality, we prove it to $\hhh = \CAT $ and get the general result via representable $2$-functors.
We just need to prove that $\Ps\Ran _ \t \AAA\circ \t (\mathsf{a}) $ is equivalent to  
$\Ps\Ran _{{\gggg _ {{}_{\t}} }} \left(\CCCC _ \alpha\circ\left(\i_{{}_{\t}}\times \Id _ {{}_{\mathsf{2} }}\right) \circ \gggg _ {{}_{\t}} \right)(\mathsf{a}, \mathsf{0}) $.

The category of pseudonatural transformations 
$\varrho ' : \dot{\aaa}(\mathsf{a}, \t (-) )\to \AAA \circ\t $ 
is equivalent to the category of
pseudonatural transformations
 $\varrho : \dot{\aaa}(\mathsf{a}, \t (-) )\longrightarrow \BBB \circ\t $ 
 that can be factorized through $\alpha\t $, since $\alpha\t $ is objectwise fully faithful. 
Also, given $\varrho : \dot{\aaa}(\mathsf{a}, \t (-) )\longrightarrow \BBB \circ\t $, 
there exists  $\varrho ' : \dot{\aaa}(\mathsf{a}, \t (-) )\to \AAA \circ\t $ such that 
$\varrho\cong (\alpha\t) \varrho '$ if and only if the image of $(\alpha\t) _ {{}_{b}}$ is 
essentially surjective in the image of $\varrho _ {{}_{b}}$ for every $b$ of $\aaa $. 
Also, if such $\varrho '$ exists, it is unique up to isomorphism: 
it is the pseudopullback of $\varrho$ along $(\alpha\t)$.

Actually, we claim that, for the existence of such $\varrho '$, it is (necessary and) sufficient 
$(\alpha\t) _ {{}_{b}}$ be essentially surjective onto the image of 
$\varrho _ {{}_{b}}$ for every object $b$ of $\mathfrak{I}_{{}_{\t }}$. That is to say, we just need to 
verify the lifting property for the objects in $\mathfrak{I}_{{}_{\t }}$.

Indeed, assume that $\varrho\i_{{}_{\t}}$ can be lifted by $\alpha\t\i_{{}_{\t}}$. Given an object $c$ of $\mathfrak{S}_{{}_{\t }}$ and a morphism $g: \mathsf{a}\to c $, we prove that $\varrho _ {{}_{c}}(g) $ is in the image of $\left(\alpha\t\right) _ {{}_{c}} $ up to isomorphism. Actually, there is a $\t $-irreducible morphism $f:\mathsf{a}\to b $ such that $g'f = f $ for some $g':b\to c $ morphism of $\aaa $, and, by hypothesis, there is an object $u$ of $\AAA (b) $ such that $\left(\alpha\t\right) _ {{}_{b}}(u)\cong \varrho _ {{}_{b}} (f) $, thereby:

$$\varrho _ {{}_{c}}(g) = \varrho _ {{}_{c}}\cdot\left(\dot{\aaa } (\mathsf{a}, \t (g'))\right)(f)\cong \BBB (g') \varrho _ {{}_{b}}(f)\cong \BBB (g') \left(\alpha\t\right) _ {{}_{b}}(u)\cong  \left(\alpha\t\right) _ {{}_{c}}(\AAA (g')(u)) .
$$

This completes the proof that it is enough to test the lifting property for the objects in $\mathfrak{I}_{{}_{\t }}$. Now, one should observe that, since $\BBB $ is of effective $\t $-descent, a pseudonatural transformation
$$\mathfrak{I} _{{}_{\t }}\times\mathsf{2}((\mathsf{a}, \mathsf{0}), \gggg _ {{}_{\t}} - ) \longrightarrow\CCCC _ \alpha\circ\left(\i_{{}_{\t}}\times \Id _ {{}_{\mathsf{2} }}\right)\circ \gggg _ {{}_{\t}} $$
is precisely determined (up to isomorphism) by a pseudonatural transformation $$\varrho : \dot{\aaa}(\mathsf{a}, \t (-) )\longrightarrow \BBB \circ\t .$$ (\textit{i.e.}, an object of $\BBB (\mathsf{a}) $), such that $\varrho\i_{{}_{\t}}$ can be lifted by 
$\alpha\t\i_{{}_{\t}}$. That is to say, as we proved, this is just a pseudonatural transformation 
$$\varrho ' : \dot{\aaa}(\mathsf{a}, \t (-) )\to \AAA \circ\t .$$
\end{proof}

\begin{rem}
Definition \ref{structuregeneral} and Theorem \ref{GALOISI} are part of a 
general perspective over generalizations
of classical theorems of cubes and pullbacks. 
The exhaustive exposition of such is outside the scope of this paper.

\end{rem}

We return to the context of Section \ref{Formal}. Let $\TTT $ be an idempotent pseudomonad on a $2$-category $\hhh $ and $X$ be an object of $\hhh $. We say that $X$ is of \textit{$\TTT $-descent} if the comparison $\eta _ {{}_{X}}: X\to\TTT (X) $ is fully faithful.
It is important to note that, if $\AAA : \dot{\aaa }\to\hhh $ is of $\t $-descent (following Definition \ref{almost descent}), then $\AAA $ is of $\Ps\Ran _ {\t }(-\circ\t ) $-descent.

\begin{cor}\label{strong commutativity}
Let $\TTT $ be an idempotent pseudomonad on  $\hhh $ and 
$\AAA : \dot{\aaa}\to\hhh $ a pseudofunctor such that all the objects in the image of $\AAA \circ \t $ are
$\TTT $-descent objects. Assume that both $\AAA, \TTT\circ\AAA $ are of effective $\t $-descent.
We assume that $\AAA (b) $ can be 
endowed with a $\TTT $-pseudoalgebra structure for every object $b\not\in \mathfrak{S}_{{}_{\t }}$ in $\aaa $. Then $\AAA (\mathsf{a}) $ can be endowed with a $\TTT $-pseudoalgebra structure.
\end{cor}

\begin{cor}\label{comutatividadeforte}
Let $\AAA : \dot{\aaa }\to \left[ \dot{\bbb } , \hhh\right] _ {PS}$ be an effective $\t$-descent pseudofunctor such that all the pseudofunctors in the image of $\AAA\circ\t $ are of $\h$-descent. Furthermore, we assume that $\AAA (b) $ is of effective $\h $-descent for every $b\not\in \mathfrak{S}_{{}_{\t }}$ in $\aaa $. Then $\AAA (\mathsf{a}) $ is of effective $\h $-descent.
\end{cor}

Recall the following full inclusion of $2$-categories $\h : \bbb\to\dot{\bbb }$ described in Section \ref{Elementary Examples}.
\begin{equation}\tag{$\mathfrak{P}$}\label{ppullback}
\xymatrix{ & e\ar[d]\ar@{}[rd]|-{\mapsto } & \mathsf{b}\ar[d]\ar[r]&e\ar[d] \\
c\ar[r]&o & c\ar[r]&o
}
\end{equation}
As explained there, a diagram $\dot{\bbb }\to \hhh $ is of effective $\h $-descent if and only if it is a pseudopullback.
In this case, the unique object in $\mathfrak{S} _ {{}_{\h}} $ is $o$. Thereby we get:

\begin{cor}
Assume that $\AAA : \dot{\bbb }\to [\dot{\aaa }, \hhh ] _ {PS}$ is a pseudopullback diagram. If $\AAA (c) , \AAA (e): \dot{\aaa } \to \hhh $ are of effective $\t $-descent and $\AAA (o): \dot{\aaa } \to \hhh $ is of $\t $-descent, then $\AAA (\mathsf{b} ) $ is of effective $\t $-descent.
\end{cor}

Taking into account Remark~\ref{inducedbasicproblem} and realizing that pseudopullbacks of functors induce pseudopullback of overcategories, we get Theorem \ref{pseudopullback} as a corollary.

\subsection{Applications}
In this subsection, we finish the paper giving applications of our results and proving the remaining theorems presented in Section \ref{Basic Problem}.
Firstly, considering our inclusion $\j : \Delta\to \dot{\Delta } $, it is important to observe that $\mathsf{1}\not\in\mathfrak{S} _ {{}_{\j }} $, while all the other objects of $\Delta $ belong to $\mathfrak{S} _ {{}_{\j }} $. We start proving Theorem 4.2 of  \cite{StreetJane}, which is presented therein as a generalized Galois Theorem.

\begin{theo}[Galois]\label{GALOISII}
Let $\AAA , \BBB :\dot{\Delta }\to\CAT $ be pseudofunctors and $\alpha : \AAA\longrightarrow\BBB $ be an objectwise fully faithful pseudonatural transformation.
We assume that $\BBB $ is of effective $\j $-descent. The pseudofunctor $\AAA $ is also of effective $\j $-descent if and only if the diagram below is a pseudopullback.
$$\xymatrix{ \AAA(\mathsf{0})\ar[d]_{\alpha _ {{}_{\mathsf{0} }} }\ar[r]^{\AAA (d) }\ar@{}[rd]|-{\xRightarrow{\alpha _ {{}_{d}}} }& \AAA (\mathsf{1})\ar[d]^-{\alpha _ {{}_{\mathsf{1} }}}\\
\BBB (\mathsf{0})\ar[r]_ {\BBB (d) }&\BBB (\mathsf{1}) 
}$$ 
\end{theo}
\begin{proof}
Since, in this case,  $\mathfrak{I}_{{}_{\j }} = \mathsf{2} $ and the inclusion $\gggg _ {{}_{\j }} :\overline{\mathfrak{I} _{{}_{\j }}\times\mathsf{2} }\to \mathfrak{I} _{{}_{\j }}\times\mathsf{2}  $ is precisely equal to the inclusion described in the diagram \ref{ppullback}, by Theorem \ref{GALOISI}, the proof is complete.
\end{proof}

As a consequence of Theorem \ref{GALOISII},  we get a generalization of Theorem \ref{mergulho}. More precisely, in the context of Section \ref{proofsclassical} and using the definitions presented there, we get:

\begin{cor}
Let $(U, \alpha )$ be a fully faithful morphism between pseudofunctors $\AAA : \CCC ^{\op }\to \hhh $ and  $\BBB:\DDD ^{\op }\to \hhh $, in which $\CCC $ and $\DDD $ are categories with pullbacks. Assume that $U(p)$ is an effective $\BBB $-descent morphism of $\DDD $. Then $p: E\to B$ is of effective $\AAA $-descent if and only if, whenever there are $u\in \BBB (B), v\in \AAA (E) $ such that $\alpha ^{{}^{p}} _{{}_{\mathsf{1} }} (u)\cong \BBB _ {{}_{U(p) }}(d) (v) $, there is $w\in \AAA(B) $ such that $\alpha ^{{}^{p}} _{{}_{\mathsf{0} }}(w)\cong u $. 
\end{cor}
\begin{proof}
Recall the definitions of $\AAA _ {{}_{p}}, \BBB _ {{}_{U(p)}}, \alpha ^{{}^{p}}$. Since we already know that $\AAA _ {{}_{p}}$ is $\j$-descent, the condition described is precisely the condition necessary and sufficient to conclude that the diagram of Theorem \ref{GALOISII} is a pseudopullback.
\end{proof}

Indeed, taking into account Remark \ref{inducedbasicproblem}, we conclude that Theorem \ref{mergulho} is actually a immediate consequence of last corollary.

Given a category with pullbacks $V$, we denote by $\Cat (V) $ the category of internal categories in $V$. If $V$ is a category with products, we denote by $V$-$\Cat $ the category of small categories enriched over $V$. We give a simple application of the Theorem \ref{pseudopullback} below.

\begin{lem}
If $(V, \times , I) $ is an infinitary lextensive category such that 
\begin{eqnarray*}
J: \Set & \to & V\\
A &\mapsto & \sum _{a\in A} I_a
\end{eqnarray*}
is fully faithful, then the pseudopullback of the projection of the object of objects $U_0: \Cat (V)\to V$ along $J$ is the category $V$-$\Cat $. 
\end{lem}
\begin{proof}
We denote by $\textit{Span}(V)$ the usual bicategory of objects of $V$ and spans between them and by $V$-$\mathrm{Mat}$ the usual bicategory of sets and $V$-matrices between them. Let  $\textit{Span} _{{}_{\Set }}(V) $ be the full sub-bicategory of $\textit{Span}(V)$ in which the objects are in the image of $\Set $.

Assuming our hypotheses, we have that $\textit{Span} _{{}_{\Set }}(V) $ is biequivalent to $V$-$\mathrm{Mat}$. Indeed, we define ``identity'' on the objects and, if $A, B $ are sets, take a matrix
$M: A\times B\to \obj (V) $ to the obvious span given by the coproduct $\displaystyle\sum _ {(x,y)\in A\times B} M(x,y) $, that is to say,  the morphism $\sum _ {(x,y)\in A\times B} M(x,y)\to A $ is induced by the morphisms $M(x,y)\to I_x $ and the morphism   $\sum _ {(x,y)\in A\times B} M(x,y)\to B $ is analogously defined.

Since $V$ is lextensive, this defines a biequivalence. Thereby this completes our proof.
\end{proof}

Corollary 6.2.5 of \cite{Ivan} says in particular that, for lextensive categories, effective descent morphisms of $\Cat(V)$ are preserved by the projection $U_0: \Cat (V)\to V$ to the objects of objects. Thereby, by Theorem \ref{pseudopullback}, we get:

\begin{theo}\label{vcategories}
If $(V, \times, I) $ is an infinitary lextensive category such that each arrow of $V$ can be factorized as a regular epimorphism followed by a monomorphism and 
\begin{eqnarray*}
J: \Set & \to & V\\
A &\mapsto & \sum _{a\in A} I_a
\end{eqnarray*}
is fully faithful, then $I: V\textrm{-}\Cat\to \Cat (V) $ reflects effective descent morphisms.
\end{theo} 

\begin{proof}
We denote by $U: V\textrm{-}\Cat\to \Set $ the forgetful functor and by $U_0: \Cat (V)\to V$ the projection defined above. We have that $U_0, U, J $ and $I$ are pullback preserving functors.
 
If $p:E\to B $ is a morphism of $V$-$\Cat $ such that $I(p)$ is of effective descent, then  $U_0 I(p) $ is of descent (by Corollary 5.2.1 of \cite{Ivan}). 
Therefore $JU(p)$ is of descent. 

Since $J$ is fully faithful, by Theorem \ref{embdddd}, $U(p) $ is of descent. Therefore, since descent morphisms of $\Set $ are of effective descent, we conclude that $U(p)$ is of effective descent.
This completes the proof.
\end{proof}

For instance, Theorem 6.2.8 of \cite{Ivan} and Proposition \ref{vcategories} can be applied to the cases of $V=\Cat $ or $V = \Top $:

\begin{cor}
A $2$-functor $F$ between $\Cat $-categories is of effective descent in $\Cat $-$\Cat $, if
\begin{itemize}
\renewcommand\labelitemi{--}
\item $F$ is surjective on objects;
\item $F$ is surjective on composable triples of $2$-cells;
\item $F$ induces a functor surjective on composable pairs of $2$-cells between the categories of composable pairs of $1$-cells;
\item $F$ induces a functor surjective on $2$-cells between the categories of composable triples of $1$-cells.
\end{itemize}
\end{cor}

\begin{cor}\label{TOP}
A $\Top $-functor $F$ between $\Top $-categories is of effective descent in $\Top $-$\Cat $, if $F$ induces
\begin{itemize}
\renewcommand\labelitemi{--}
\item effective descent morphisms  between the discrete spaces of objects and between the spaces of morphisms in $\Top $;
\item a descent continuous map between the spaces of composable pairs of morphisms in $\Top $;
\item an almost descent continuous map between the spaces of composable triples of morphisms in $\Top $.
\end{itemize}
\end{cor}

Since the characterization of (effective/almost) descent morphisms in $\Top $ is known~\cite{THR, CleJane, reforCleHof}, the result above gives effective descent morphisms of $\Top $-$\Cat $.

\begin{rem}
We can give further formal results on (basic) effective descent morphisms (context of Remark \ref{inducedbasicproblem}). The main technique in this case is to understand our overcategory as a bilimit of other overcategories.

For instance, we study below the categories of morphisms of a given category $\CCC $ with pullbacks. Consider the full inclusion of $2$-categories
$\t : \aaa\to \dot{\aaa} $
$$\xymatrix{\mathsf{0}\ar[d]|-{d}\ar@{}[rd]|{\mapsto }& \mathsf{a}\ar[rd]|-{pro_{{}_{0}} }\ar[rr]|-{pro_{{}_{1}}}&\ar@{}[d]|-{\xRightarrow{\xi} } &\mathsf{0}\ar[ld]|-{d}\\
\mathsf{1}& &\mathsf{1}& 
} $$
Given a morphism of $\CCC $, \textit{i.e.} a functor $F:\mathsf{2}\to\CCC $, we take the overcategory $\textrm{Fun}(\mathsf{2}, \CCC )/F $ and define $\AAA : \dot{\aaa }\to \CAT $ in which 
\begin{equation*}
\begin{aligned}
\AAA(\mathsf{a} ):= &\textrm{Fun}(\mathsf{2}, \CCC )/F
\end{aligned}
\qquad
\begin{aligned}
\AAA(\mathsf{0} ):= & \CCC /F(\mathsf{1})
\end{aligned}
\qquad
\begin{aligned}
\AAA(\mathsf{1} ):= & \CCC /F(\mathsf{0}).
\end{aligned}
\end{equation*}
Finally, $\AAA(pro_{{}_{0}}), \AAA(pro_{{}_{1}}) $ are given by the obvious projections, $\AAA (d) := F(d) ^\ast $ and the component $\AAA (\xi ) $ in a morphism $\varpi : H\to F $ is given by the induced morphism from $H(\mathsf{0} ) $ to the pullback.

Observe that $\AAA $ is of effective $\t $-descent, that is to say, we have that the overcategory $\textrm{Fun}(\mathsf{2}, \CCC )/F $ is a bilimit constructed from overcategories $\CCC /F(\mathsf{0})$ and $\CCC /F(\mathsf{1})$. Also, given a natural transformation $\varpi : F\to G $ between functors $\mathsf{2}\to\CCC $, \textit{i.e.} a morphism of $\textrm{Fun}(\mathsf{2}, \CCC )$, taking Remark \ref{inducedbasicproblem}, we can extend $\AAA $ to a $2$-functor $\overline{\AAA }:\dot{\aaa}\to [\dot{\Delta}, \CAT] $ in which
$\overline{\AAA }(\mathsf{a}) : = (\mbox{ })^\ast _\varpi $, $\overline{\AAA }(\mathsf{0}) : = (\mbox{ })^\ast _{\varpi _\mathsf{1}} $ and $\overline{\AAA }(\mathsf{1}) : = (\mbox{ })^\ast _{\varpi _\mathsf{0}} $.

The $2$-functor $\overline{\AAA }$ is also of effective $\t $-descent. Therefore, by our results, we conclude that, \textit{if the components  $\varpi _\mathsf{1}, \varpi _\mathsf{0} $ are of (basic) effective descent, so is $\varpi $}. Analogously, considering the category of spans in $\CCC $, the morphisms between spans which are objectwise of effective descent are of effective descent. 

\end{rem}


\begin{thebibliography}{99}
\bibitem{Barr} M. Barr and C. Wells.  Toposes, triples and theories. Corrected reprint of the 1985 original [MR0771116]. Repr. {\em Theory Appl. Categ.} no. 12 (2005), x+288 pp.
\bibitem{Beck} J.M. Beck. Triples, algebras and cohomology. Repr. Theory Appl. Categ. 2003, No.2, 1--59 (2003).

\bibitem{BE} J. B\'{e}nabou. Introduction to bicategories, in {\it Reports of the Midwest Category Seminar}, 1--77, Springer, Berlin.

\bibitem{Equi} J. B\'{e}nabou and J. Roubaud. Monades et descente. {\em C. R. Acad. Sc. Paris, t.} 270 (12 Janvier 1970), Serie A, 96--98.
\bibitem{Flexible} G.J. Bird,  G.M. Kelly, A.J. Power, R.H. Street. Flexible limits for 2-categories. 
{\em J. Pure Appl. Algebra } 61 (1989), no. 1, 1--27. 





\bibitem{Power} R. Blackwell, G.M. Kelly and A.J. Power. Two-dimensional monad theory. 
{\em J. Pure Appl. Algebra} 59 (1989), no. 1, 1--41.



\bibitem{Bour} F. Borceux. Handbook of categorical algebra 2. Categories and structures. {\em Encyclopedia of Mathematics and its Applications}, 51. Cambridge Univ. Press, Cambridge, 1994. xviii+443 pp. 



\bibitem{reforCleHof} M.M. Clementino and D. Hofmann. Triquotient maps via ultrafilter convergence. {\em Proc. Amer. Math. Soc.} 130 (2002), 3423--3431.

\bibitem{Clementino1} M.M. Clementino and D. Hofmann. Effective descent morphisms in categories of lax algebras. {\em Appl. Categ. Structures} 12 (2004), no. 5--6, 413--425.

\bibitem{CleJane} M.M. Clementino and G. Janelidze. A note on effective descent morphisms of topological spaces and relational algebras. {\em Topology Appl. } 158 (2011), 2431--2436.


\bibitem{Dubuc} E. Dubuc. Adjoint triangles. 1968 {\em Reports of the Midwest Category Seminar}, II pp. 69--91 Springer, Berlin. 

\bibitem{Dubuc2} E. Dubuc. Kan extensions in enriched category theory. {\em Lecture Notes in Mathematics, Vol.} 145 Springer-Verlag, Berlin-New York xvi$+$173 pp. (1970).

\bibitem{Duskin} J. Duskin. Simplicial methods and the interpretation of ``triple'' 
cohomology. Mem. Am. Math. Soc. 163, 135 p. (1975).


\bibitem{Gd}  J. Giraud. M\'{e}thode de la descente. {\em Bull. Soc. Math. France Mém.} 2 1964 viii+150 pp. 




\bibitem{Gray}  J.W. Gray. Quasi-Kan extensions for 2-categories. {\em Bull. Amer. Math. Soc. } 80 (1974), 142--147. 



\bibitem{Gray3} J.W. Gray.  Formal category theory: adjointness for $2$-categories. 
{\em Lecture Notes in Mathematics}, Vol. 391. Springer-Verlag, Berlin-New York, 1974. xii+282 pp.







\bibitem{Grothendieck} A. Grothendieck. \textit{Technique de descente et th\'{e}orems d'existence en g\'{e}ometrie algebrique, I. Generaliti\'{e}s. Descente par morphismes fid\'{e}lement plats}. S\'{e}minaire Bourbaki 190 (1959).








\bibitem{Her} C. Hermida. Descent on $2$-fibrations and strongly regular $2$-categories. {\em Appl. Categ. Struct.} 12 (2004), 427--459.







\bibitem{Sobral} G. Janelidze, M. Sobral, W. Tholen, Beyond Barr exactness: effective descent morphisms, in: \textit{Categorical foundations}, pp. 359--405, Encyclopedia Math. Appl., 97, Cambridge Univ. Press, Cambridge, 2004.

\bibitem{StreetJane} G. Janelidze, D. Schumacher and R.Street. Galois theory in variable categories. {\em Appl. Categ. Structures} 1 (1993), no. 1, 103--110.

\bibitem{Galois} G. Janelidze. \textit{Descent and Galois Theory}, Lecture Notes of the Summer School in Categorical Methods in Algebra and Topology, Haute Bodeux, Belgium (2007).

\bibitem{Facets} G. Janelidze and W. Tholen. Facets of descent I. {\em Appl. Categ. Structures} 2 (1994), no. 3, 245--281.

\bibitem{Facets2} G. Janelidze and W. Tholen. Facets of descent II. {\em Appl. Categ. Structures} 5 (1997), no. 3, 229--248.

\bibitem{Facets3} G. Janelidze and W. Tholen. Facets of descent III. Monadic descent for rings and algebras. {\em Appl. Categ. Structures} 12 (2004), no. 5--6, 
461--477.







\bibitem{KELLYLACK1997} G.M. Kelly and S. Lack.
On property-like structures.{\em Theory Appl. Categ.} 3 (1997), No. 9, 213--250 







\bibitem{DoctrinalAdjunction} G.M. Kelly. Doctrinal adjunction. {\em Category Seminar } (Proc. Sem., Sydney, 1972/1973), pp. 257-280. {\em Lecture Notes in Math.}, Vol. 420, Springer, Berlin, 1974.






\bibitem{Kelly} G.M. Kelly. Basic concepts of enriched category theory. 
{\em London Mathematical Society Lecture Note Series}, 64. Cambridge University Press, Cambridge-New York, 1982. 245 pp.


\bibitem{SteveLack} S. Lack. Codescent objects and coherence. Special volume celebrating the 70th birthday of Professor Max Kelly. 
{ \em J. Pure Appl. Algebra } 175 (2002), no. 1--3, 223--241.

\bibitem{SteveLack2} S. Lack. A coherent approach to pseudomonads. {\em Adv. Math.} 152 (2000), no. 2, 179--202.

\bibitem{SLACK2007} S. Lack. Icons, Appl. Categ. Structures {\bf 18} (2010), no.~3, 289--307.



\bibitem{Lawvere} F.W. Lawvere. Ordinal sums and equational doctrines. Lecture Notes in math. 80, Springer (1969), 141--155.

\bibitem{Ivan} I.J.  Le Creurer, \textit{Descent of internal categories}, PhD Thesis, Universit\'{e} catholique de Louvain, Louvain-la-Neuve, 1999.		


\bibitem{CME}  I.J. Le Creurer, F. Marmolejo and E.M. Vitale. Beck's theorem for pseudo-monads.
{ \em J. Pure Appl. Algebra} 173 (2002), no. 3, 293--313. 




\bibitem{FLN} F. Lucatelli Nunes. On Biadjoint Triangles. {\em Theory Appl. Categ.}  31 (2016), No. 9, pp 217--256.

\bibitem{FLN3} F. Lucatelli Nunes. On lifting of biadjoints and lax algebras. {\em Categ. Gen. Algebr. Struct. Appl.} Article in Press. Available Online (arXiv:1607.03087).

\bibitem{FLN4} F. Lucatelli Nunes. Freely generated n-categories, coinserters and presentations of low dimensional categories. arXiv:1704.04474.

\bibitem{FLN2} F. Lucatelli Nunes. A coherent approach to biadjoint triangles (in preparation)


\bibitem{MACLANE} S. Mac Lane. {\it Categories for the working mathematician}, second edition, Graduate Texts in Mathematics, 5, Springer-Verlag, New York, 1998.

\bibitem{Marmolejo1} F. Marmolejo. Distributive Laws for pseudomonads. {\em Theory Appl. Categ.} 5 (1999), No. 5, 91--147.

\bibitem{Marmolejo2} F. Marmolejo. Distributive laws for pseudomonads II. J. {\em Pure Appl. Algebra} 194 (2004), no. 1-2, 169--182.





\bibitem{Power88} A.J. Power.
A unified approach to the lifting of adjoints.
{ \em Cahiers Topologie G\'{e}om. Diff\'{e}rentielle Cat\'{e}g.} 29 (1988), No.1, 67--77. 

\bibitem{Power89} A.J. Power. Coherence for bicategories with finite bilimits I. {\em 
Categories in Computer Science and Logic, Contemporary Mathematics} 92 (1989) 341-348.






\bibitem{THR} J. Reiterman and W. Tholen. Effective descent maps of topological spaces. {\em Topology Appl. } 57 (1994), no.1, 53--69




\bibitem{Street} S. Schanuel and R.H. Street. The free adjunction. {\em Cahiers Topologie G\'{e}om. Dif\'{e}rentielle Cat\'{e}g. } 27 (1986), no. 1, 81--83. 

\bibitem{Street2} R.H. Street.  Categorical and combinatorial aspects of descent theory.
{\em Applied Categ. Structures} 12 (2004), no. 5-6, 537--576.

\bibitem{Street3} R.H. Street.  Limits indexed by category-valued 2-functors.
{\em J. Pure Appl. Algebra} 8 (1976), no. 2, 149--181.

\bibitem{Street4} R.H. Street. Fibrations in bicategories.
{\em Cahiers Topologie G\'{e}om. Dif\'{e}rentielle} 21 (1980), no. 2, 111--160.

\bibitem{Street5} R.H. Street. Correction to: ``Fibrations in bicategories'' {\em Cahiers Topologie G\'{e}om. Diff\'{e}rentielle Cat\'{e}g.} 28 (1987), no. 1, 
53--56.
















\end{thebibliography}
 \end{document}